\documentclass[a4paper,12pt]{amsart}
\usepackage{graphicx}
\usepackage{mathtools}
\usepackage{amsmath}
\usepackage{amssymb}
\usepackage{amsopn}
\usepackage{epsfig}
\usepackage{amsfonts}
\usepackage{latexsym}
\usepackage[T1]{fontenc}    %sajat
\usepackage[utf8]{inputenc}%sajat
\usepackage{enumerate}
\usepackage{amssymb}
\usepackage{fullpage}
\usepackage{color}
\usepackage{mathtools}
\def\NN{\mathbb N}
\def\uu{\mathcal U}
\def\uuu{\overline{\mathcal U}}
\def\vv{\mathcal V}
\def\bb{\mathcal B}
\def\ff{\mathcal F}
\def\gg{\mathcal G}

\def\cc{\mathcal C}

\newcommand{\set}[1]{\left\lbrace #1\right\rbrace}%set
\providecommand{\abs}[1]{\left\lvert#1\right\rvert}

\newcommand{\remove}[1]{ }
\newcommand{\qtq}[1]{\quad \text{#1}\quad }%\quad\text{and}\quad

\newtheorem{theorem}{Theorem}[section]
\newtheorem{proposition}[theorem]{Proposition}
\newtheorem{lemma}[theorem]{Lemma}
\newtheorem{corollary}[theorem]{Corollary}

\theoremstyle{definition}

\theoremstyle{remark}
\newtheorem{remark}[theorem]{Remark}

\newtheorem{example}[theorem]{Example}

\numberwithin{equation}{section}
\numberwithin{figure}{section}
\begin{document}
\title[Univoque graphs and multiple expansions]{Univoque graphs and multiple expansions}
\author{Yuru Zou}
\address{College of Mathematics and Computational Science,
Shenzhen University,
Shenzhen 518060,
People’s Republic of China}
\email{yuruzou@szu.edu.cn}
\author{Jian Lu}
\address{College of Mathematics and Computational Science,
Shenzhen University,
Shenzhen 518060,
People’s Republic of China}
\email{jianlu@szu.edu.cn}
\subjclass[2010]{Primary: 11A63, Secondary:  37B10, 11K55, 28A80}
\keywords{Expansions in non-integer bases, univoque graphs, multiple expansions, strong connectedness, Hausdorff dimension}
\author{Vilmos Komornik*}
\address{College of Mathematics and Statistics,
Shenzhen University,
Shenzhen 518060,
People’s Republic of China,
and
Département de mathématique,
Université de Strasbourg,
7 rue René Descartes,
67084 Strasbourg Cedex, France}
\email{komornik@math.unistra.fr}

%\curraddr{}
%\date{Version 2019-11-08b}

\thanks{*Corresponding author; email address: komornik@math.unistra.fr.}
\thanks{This work was supported by the National Natural Science Foundation
of China (NSFC)  \#11871348, \#61972265. The third author thanks  the first and second authors  for their hospitality and for the pleasant working conditions during his visit of  Shenzhen University.}

\begin{abstract}\mbox{}
Unique expansions in non-integer bases $q$ have been investigated in many papers during the last thirty years.
They are often conveniently generated by labeled directed graphs.
In the first part of this paper we give a precise description of the set of sequences generated by these graphs.

Using the description of univoque graphs, the second part of the paper is devoted to the study of multiple expansions.
Contrary to the unique expansions, we prove for each $j\ge 2$ that the set $\uu_q^j$ of numbers having exactly $j$ expansions is  closed only if it is empty.

Furthermore, generalizing an important example of Sidorov \cite{S2009}, we prove for a large class of bases that the Hausdorff dimension of $\uu_q^j$ is independent of $j$.

In the last two sections our results are illustrated by many  examples.
\end{abstract}
\maketitle

\section{Introduction}\label{s1}

Given a positive integer $M$ and a real number $q>1$, by an \emph{expansion} of a real number $x$ in  \emph{base} $q$ over the \emph{alphabet} $A:=\set{0,1,\ldots,M}$ we mean a sequence $(c_i)\in A^{\infty}$ satisfying the equality
\begin{equation*}
(c_i)_q:=\sum_{i=1}^{\infty}\frac{c_i}{q^i}=x.
\end{equation*}
By a classical theorem of R\'enyi \cite{R1957} we have
\begin{equation*}
\set{(c_i)_q\ :\ (c_i)\in A^{\infty}}\subset\left[0,\frac{M}{q-1}\right],
\end{equation*}
with equality if and only if $q\in(1,M+1]$.
More precisely, he proved that every $x\in\left[0,\frac{M}{q-1}\right]$ has a lexicographically largest, called \emph{greedy} expansion.
In the sequel we  denote by $\beta(q)$ the greedy expansion of $x=1$ in base $q$.

In the classical integer base case $q=M+1$ every $x\in [0,1]$ has one or two expansions, and the second possibility occurs only for countably many rational numbers.
If $q>M+1$, then every expansion is unique.
On the other hand, if $q<M+1$, then Lebesgue almost every number in $\left[0,\frac{M}{q-1}\right]$ has a continuum of expansions \cite{S2003}.
Nevertheless, the \emph{univoque sets} $\uu_q$ of numbers $x$ having a unique expansion in base $q$ have many interesting properties: see \cite{DK2009} and \cite{KKL2017} for the topological and fractal  structure of these sets, respectively.

We recall from \cite{DK2009} that the set $\uu_q'$ of unique expansions (sequences) of the numbers $x\in\uu_q$ is a subshift of finite type for almost every $q$.
Jiang and Dajani \cite{DJ2017} have constructed a natural labeled graph $\gg(q)$ generating $\uu_q'$ up to a countable set, and used it to complete some of the results of \cite{DK2009}.
Our first aim is to describe more precisely the structure of these graphs, and to determine exactly the set of sequences generated by them.
This allows us to explain various former results on univoque expansions in a transparent way, instead of the former combinatorial and less intuitive arguments.

In order to state these results we recall some notions from \cite{KL2007,
DK2009,
K2012,
DKL2016}.
A sequence or expansion $(c_i)$ is called \emph{finite} if it has a last non-zero digit, and \emph{infinite} otherwise.
Equivalently, a sequence is infinite if it contains infinitely many non-zero digits, or if it is identically zero.
(Calling the zero sequence infinite simplifies many statements in the present theory.)
Furthermore, a sequence or expansion $(c_i)$ is called \emph{doubly infinite} if both $(c_i)$ and its \emph{reflection} $\overline{(c_i)}:=(M-c_i)$ is infinite.
Equivalently, a sequence $(c_i)$ is doubly infinite if it contains both infinitely many digits $c_i>0$ and infinitely many digits $c_i<M$, or if it is equal to one of the sequences $0^{\infty}$ and $M^{\infty}$.
For a word $w=c_1\cdots c_{n-1}c_n$ we  write  $w^+=c_1\cdots c_{n-1}c_n^+$ with $c_n^+:=c_n+1$ if $c_n<M$, and $w^-=c_1\cdots c_{n-1}c_n^-$ with $c_n^-:=c_n-1$ if $c_n>0$.

We denote by $\uu$ and $\vv$ the sets of bases $q$ in which the number $x=1$ has a unique expansion or a unique doubly infinite expansion, respectively.
Then $\vv$ is closed, and $\uu\subsetneq\uuu\subsetneq\vv$, where $\uuu$ denotes the topological closure of $\uu$.
All three sets have zero Lebesgue measure and Hausdorff dimension one, and both difference sets $\uuu\setminus\uu$ and $\vv\setminus\uuu$ are countably infinite.
We denote by $q_{GR}$ the smallest element of $\vv$, and by $q_{KL}$  the smallest element of $\uuu$.
In fact, we have $q_{KL}\in\uu$.

For example, if $M=1$, then $q_{GR}\approx 1.61803$ is the \emph{Golden Ratio} $\varphi\approx 1.61803$,  while $q_{KL}\approx 1.78723$; see \cite{KL1998}.
We have also $2\in\uu$, and the \emph{Tribonacci number} $\varphi_3\approx 1.83929$, i.e., the positive root of the equation $q^3=q^2+q+1$ belongs to $\uuu\setminus\uu$.

Although the set $\vv$ has been introduced in \cite{KL2007} for aesthetical reasons, it plays an important role in other branches of mathematics as well; see, e.g., Bonnano et al. \cite{BCIT2013}, Dajani et al. \cite{DK2007}. 

Some of the results on $\uu$ have been extended by L\"u et al. \cite{LTW2014} to the sets of bases $\uu(x)$ in which a given number $x\in(0,1)$ has a unique expansion.

We denote by $\uu', \uuu', \vv'$ the sets of the (unique) doubly infinite expansions of $x=1$ in the bases belonging to  $\uu, \uuu, \vv$, respectively.

For any fixed $q\in (1,M+1)$, we denote by $\uu_q$ and $\vv_q$ the sets of numbers $x$ having a unique expansion or a unique doubly infinite expansion, respectively.
We recall from \cite{DK2009,K2012} that
$\vv_q$ is closed and $\uu_q\subset\overline{\uu_q}\subset\vv_q$. Here  $\overline{\uu_q}$ denotes the topological closure of $\uu_q$.
We denote by $\uu_q', \overline{\uu_q}',  \vv_q'$ the corresponding sets of (unique) doubly infinite expansions.

If $q\in\vv\setminus\uu$, then $x=1$ has a finite greedy expansion  in base $q$, say $\beta(q)=\alpha_1\cdots\alpha_N^+\ 0^{\infty}$ (see Section \ref{s2}  for more details), and
consider the points
\begin{equation*}
a_i:=(\alpha_i\cdots\alpha_N^+\ 0^{\infty})_q
\qtq{and}
b_i:=\frac{M}{q-1}-a_i
\end{equation*}
for $i=1,\ldots,N+1$, so that $a_{N+1}=0$ and $b_{N+1}=\frac{M}{q-1}$.

If $x\in \left[\frac{j}{q},\frac{j-1}{q}+\frac{M}{q^2-q}\right]$ for some $j\in\set{1,\ldots,M}$, then $x$ has an expansion starting with the digit $j-1$, and another expansion starting with the digit $j$.
On the other hand, if $x$ does not belong to the \emph{switch region}
\begin{equation*}
S_q:=\bigcup_{j=1}^M\left[\frac{j}{q},\frac{j-1}{q}+\frac{M}{q^2-q}\right]
=:\bigcup_{j=1}^M[\theta_j,\eta_j],
\end{equation*}
then all expansions of $x$ share the same first digit.
Note that $a_N=\theta_{\alpha_N^+}$ and $b_N=\eta_{\overline{\alpha_N}}$.

Let us consider the following labeled graph $\gg(q)$.
Its vertices are the connected components of the set
\begin{equation*}
\left(0,\frac{M}{q-1}\right)\setminus (S_q\cup\set{a_1,\dots, a_N,b_1,\ldots, b_N})
\end{equation*}
(they are disjoint open intervals).
Furthermore, the edges of $\gg(q)$ are the triplets $(I,k,J)$ where $I,J$ are vertices, $k\in\set{0,\ldots,M}$, and $T_k(I)\supset J$ for the affine map $T_k(x):=qx-k$.
See Section \ref{s3} for more details.

We say that a sequence $(c_i)$ is generated by $\gg(q)$ is there exists an infinite path
\begin{equation*}
I_1\xrightarrow{c_1}
I_2\xrightarrow{c_2}
I_3\xrightarrow{c_3}
\cdots
\end{equation*}
in $\gg(q)$.
We denote by $\gg_q'$ the set of sequences generated by $\gg(q)$, and we set
\begin{equation*}
\gg_q:=\set{(c_i)_q\ :\ (c_i)\in\gg_q'}.
\end{equation*}
Now we may state our first  theorem.

\begin{theorem}\label{t11}\mbox{}

\begin{enumerate}[\upshape (i)]
\item If $q\in\uuu\setminus\uu$, then $\gg_q=\vv_q$ and $\gg_q'=\vv_q'$.
\item If $q\in\vv\setminus\uuu$, then $\gg_q=\uu_q$ and $\gg_q'=\uu_q'$.
\end{enumerate}
\end{theorem}

\begin{remark}\label{r12}\mbox{}

\begin{enumerate}[\upshape (i)]
\item We may also state Theorem \ref{t11} in a uniform way in the form $\gg_q=\overline{\uu_q}$ and $\gg_q'=\overline{\uu_q}'$ for all $q\in\vv\setminus\uu$.
Indeed, we recall from \cite{DK2009,K2012} that $\uu_q$ is closed if $q\in\vv\setminus\uuu$, and $\overline{\uu_q}=\vv_q$ if $q\in\uuu\setminus\uu$.
\item Although $\vv\setminus\uu$ is a countable discrete set, Theorem \ref{t11} allows us to describe the univoque sequences for each $q\in(1,\infty)\setminus\vv$, i.e., for almost all bases $q$.
Indeed, we know from \cite{EJK1990,GS2001, KL2002,
B2014,DKL2016} that $\uu_q'=\gg_{q_{GR}}'=\set{0^{\infty},M^{\infty}}$ for every $q\in(1,q_{GR})$, and $\uu_q'=\set{0,\ldots,M}^{\infty}$ for every $q>M+1$.
Furthermore, for each $q\in(q_{GR},M+1)\setminus\vv$ there exist two consecutive elements $p<r$ of $\vv$ such that $p<q<r$, and then  $\uu_q'=\overline{\uu_p}'$ by \cite[Theorem 1.7]{DK2009}.
We also recall from \cite{DK2009} that $\uu_q'=\overline{\uu_q}'$ for all $q\in(1,\infty)\setminus\vv$.
\end{enumerate}
\end{remark}

The proof of Theorem \ref{t11} will rely on many properties of the sets $\uu_q$, established in \cite{DK2009}, and  on the fine structure of the graphs $\gg(q)$.
In order to formulate the following two theorems we need the following more precise description of the topological structure of the sets $\uu$, $\uuu$ and $\vv$.
The open set $(1,M+1)\setminus\uuu$ has infinitely many connected components (disjoint open intervals) $(q_0,q_0^*)$, where $q_0$ runs over $\set{1}\cup(\uuu\setminus\uu)$, and $q_0^*$ runs over a countable subset of $\uu$.
Next, each connected component $(q_0,q_0^*)$ contains countably many elements of $\vv$, forming an increasing sequence $q_1<q_2<\cdots ,$ converging to $q_0^*$.
Now we are ready to state our next two theorems.
Consider an arbitrary connected component $(q_0,q_0^*)$ of $(1,M+1)\setminus\uuu$, and the increasing sequence $(q_m)$ of the elements of $\vv$ in this interval.
Write $\beta(q_0)=\alpha_1\cdots\alpha_n^+\ 0^{\infty}$.

\begin{theorem}\label{t13}
$\gg(q_1)$ is isomorphic to $\gg(q_0)$.
\end{theorem}

In the following theorem we denote by $\abs{V}$ the number of elements of a set $V$.

\begin{theorem}\label{t14}
There exists an infinite graph $\hat\gg(q_0^*)$ and a partition $V_1, V_2,\ldots$ of its vertices having the following properties:
\begin{enumerate}[\upshape (i)]
\item $\abs{V_m}=2^{m-1}n$ for every $m\ge 1$;
\item the subgraph $\cc_m$ spanned by $V_m$ is purely cyclical for every $m\ge 2$;
\item for each $m\ge 1$, the subgraph spanned by $V_1\cup\cdots\cup V_m$ is isomorphic to $\gg(q_m)$;
\item there is a path from $V_j$ to $V_k$ if and only if $j\le k$.
\end{enumerate}
\end{theorem}

We write $\hat\gg(q_0^*)$ instead of $\gg(q_0^*)$ to emphasize that $\hat\gg(q_0^*)$ is a graph of a different structure because $q_0^*\notin\vv\setminus\uu$.
Figures 1.1--1.3  show  the graphs $\gg(q_m)$ for $m=0,1,2$ where $M=1$ and $q_0$ is the Tribonacci number.

\begin{figure}[ht]\label{f11}
\centering
\includegraphics[scale=0.4]{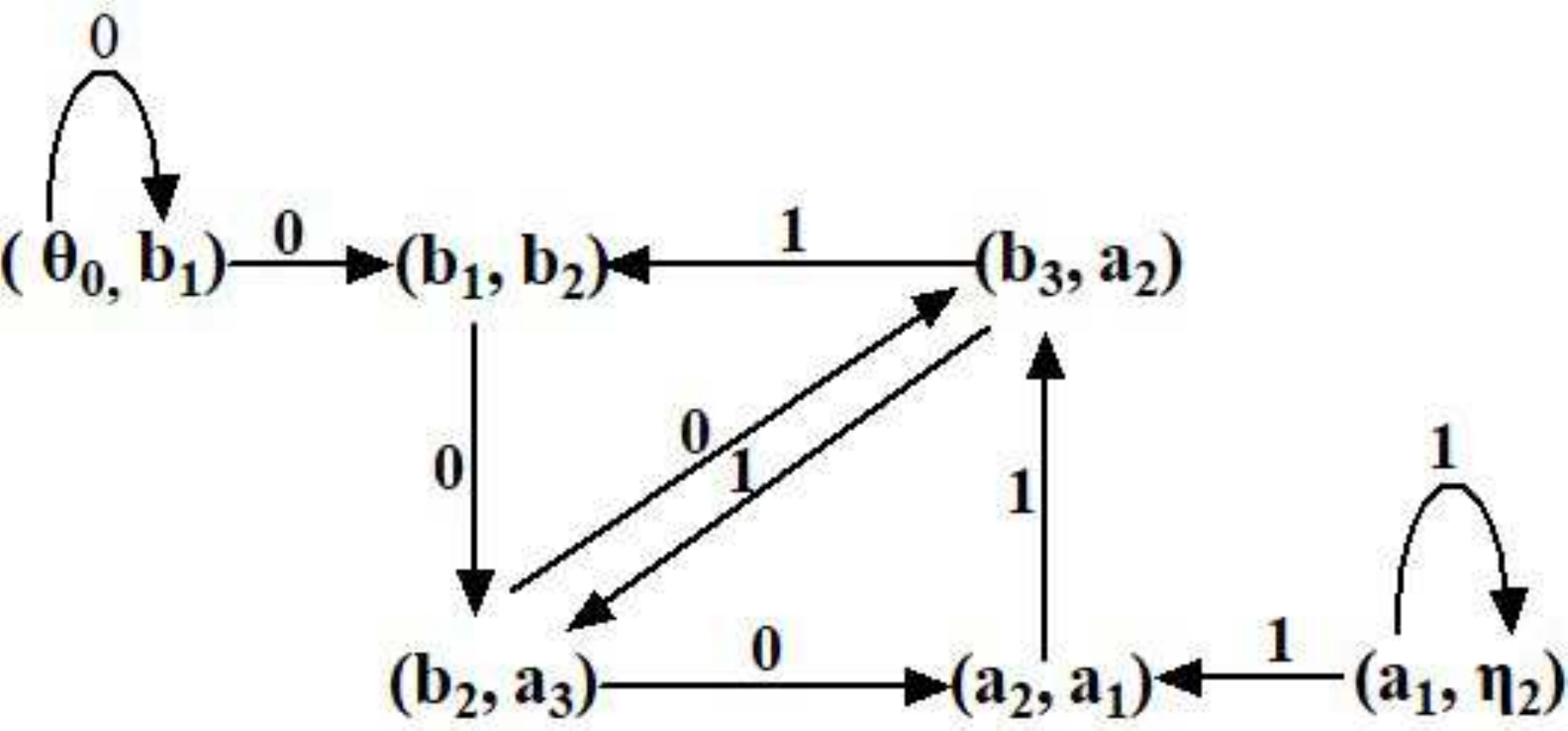}
\caption{The graph $\gg(q_0)$ for the Tribonacci number $q_0$  associated with  $\beta(q_0)=111\ 0^{\infty}$.}
\end{figure}

\begin{figure}[ht]\label{f12}
\centering
\includegraphics[scale=0.4]{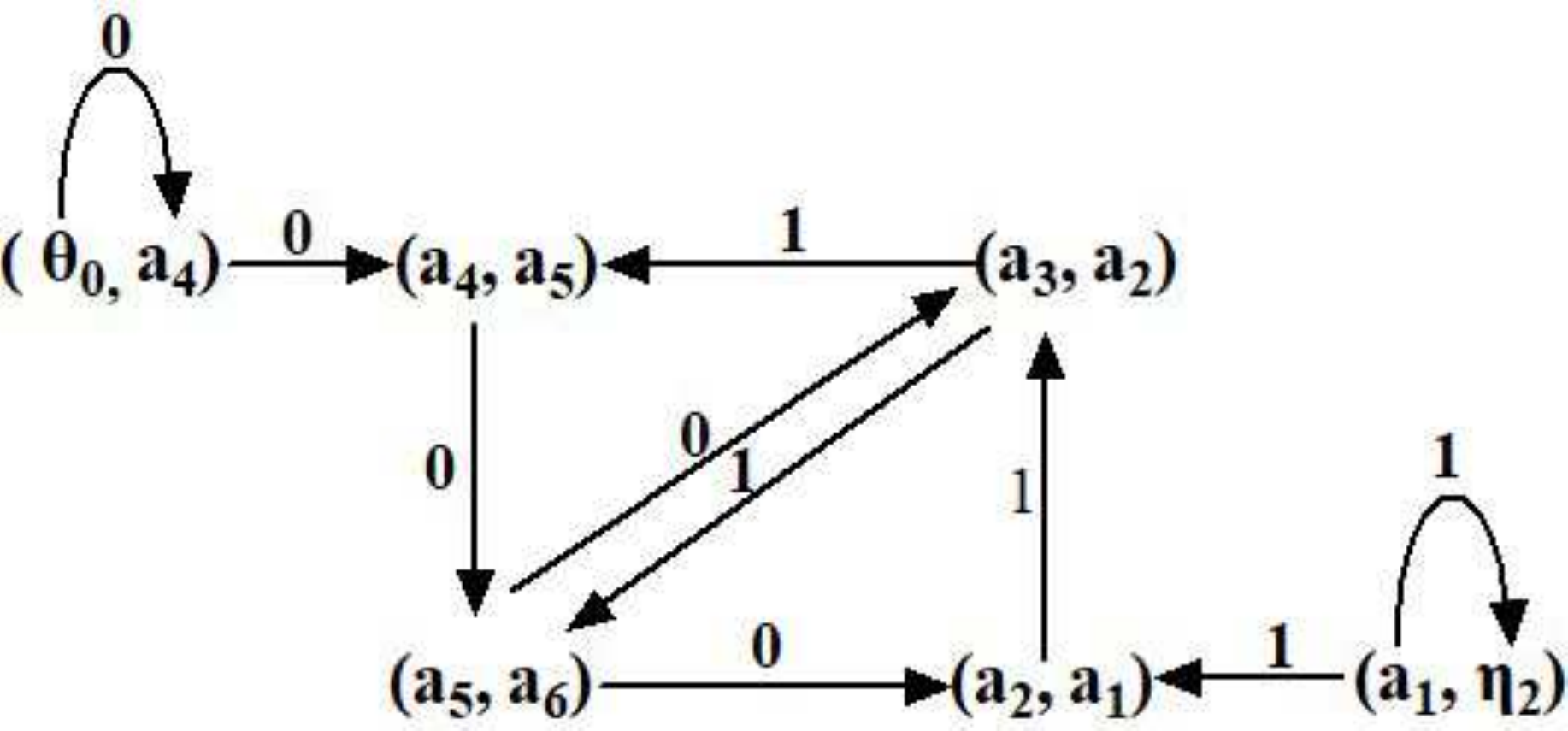}
\caption{The graph $\gg(q_1)$ associated with  $\beta(q_1)=111\ 001\ 0^{\infty}$.}
\end{figure}

\begin{figure}[ht]\label{f13}
\centering
\includegraphics[scale=0.4]{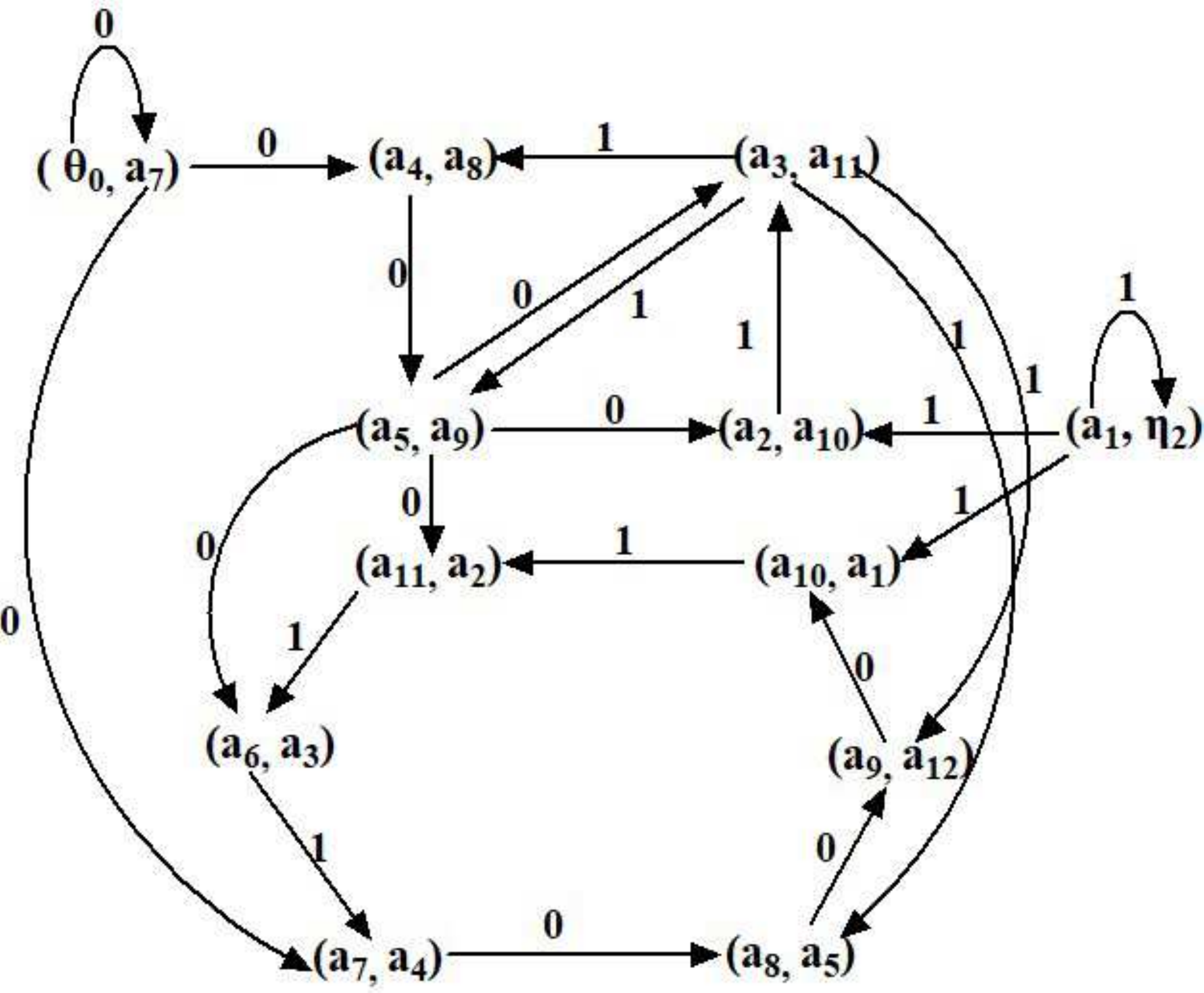}
\caption{The graph $\gg(q_2)$ associated with  $\beta(q_2)=111\ 001\ 000\ 111\ 0^{\infty}$.}
\end{figure}
\medskip

The second part of the present work is devoted to the study of the Hausdorff dimension of the sets
\begin{equation*}
\uu_q^j:=\set{x\ :\ x \text{ has exactly $j$ expansions}},\quad j=1,2,\ldots, \aleph_0 \qtq{or}2^{\aleph_0}.
\end{equation*}
This is well understood today for $\uu_q^1$ \cite{DK2009,KKL2017}, but the theory is far for complete if $j\ge 2$.
However, a number of important theorems have been obtained by Erd{\H{o}}s et al., Sidorov, Baker, Kong et al., Zou et al., Komornik et al.
\cite{EJ1992,
S2009,
BS2014,
B2014,
B2015,
KLZ2017,
ZK2015,
ZWLB2016,
KK2018}.
In order to mention some of these results concerning the two-digit case $M=1$, let us denote by $\bb_j$ the set of bases $q$ for which $\uu_q^j$ is non-empty.
We have $\bb_{2^{\aleph_0}}=(1,M+1)$. Furthermore,
\begin{equation*}
\min\bb_{\aleph_0}=\varphi,\quad
\min\bb_2=q_s\qtq{and}
\min\bb_j=q_f\qtq{for}j=3,4,\ldots,
\end{equation*}
where $\varphi<q_s<q_f$ denote the unique positive zeros of the polynomials
\begin{equation*}
q^2-q-1,\quad
q^4-2q^2-q-1\qtq{and}
q^3-2q^2+q-1,
\end{equation*}
respectively, so that
\begin{equation*}
\varphi\approx 1.61803,\quad
q_s\approx 1.71064\qtq{and} q_f\approx 1.75488.
\end{equation*}

\begin{remark}\label{r15}
Using the notation of Theorem \ref{t13}, if $M=1$ and $q_0=1$, then $q_1=\varphi$ and $q_2=q_f$.
\end{remark}

Although $\bb_{\aleph_0}$ has a smallest element, it is not closed, because it has a smallest accumulation point: $q_f$ , and $q_f\notin\bb_{\aleph_0}$.
Furthermore, $\bb_{\aleph_0}\cap(1,q_f]$ is an infinite discrete set.

On the other hand, $\bb_2$ is a compact set, having infinitely many isolated points and infinitely many accumulation points in $(1,q_{KL})$, where $q_{KL}\approx 1.78723$ denotes the \emph{Komornik--Loreti} constant, and
each derived set $\bb_2^{(1)}, \bb_2^{(1)}, \ldots$ of $\bb_2$ has the same property.
The smallest accumulation point of $\bb_2$ is $q_f$.

Turning back to the general case $M\ge 1$, we recall from \cite[Theorems 1.3, 1.4 and 1.5]{DK2009} that $\uu_q^1$ is closed if and only if $q\notin\uuu$.
The situation for $j>1$ (including the cases $j=\aleph_0$ and $j=2^{\aleph_0}$) is much simpler:

\begin{theorem}\label{t16}
If $\uu_q^j$ is closed for some $q>1$ and $j>1$, then it is empty.
\end{theorem}

\begin{remark}\label{r17}
Denoting by $\overline{\uu_q^j}$ the topological closure of $\uu_q^j$, the  proof of Theorem \ref{t16} will show that if $\uu_q^j$ is non-empty for some $j>1$, then $0\in\overline{\uu_q^j}\setminus\uu_q^j$, and then by reflection $\frac{M}{q-1}\in\overline{\uu_q^j}\setminus\uu_q^j$.
We do not know whether $\uu_q^j\cup\set{0,\frac{M}{q-1}}$ may be a closed set when $\uu_q^j$ is non-empty.
\end{remark}

Next we are going to  investigate the Hausdorff dimension of the sets $\uu_q^j$. We refer to \cite{Falconer_1990} for the rigorous definition of the Hausdorff dimension.
Once again, the Hausdorff dimension of $\uu_q^1$ is well understood today \cite{GS2001, KLD2010,
KL2015,
KKL2017,
AK2018}: we have $\dim\uu_q^1=\dim \overline{\uu_q^1}$ for every $q$,  the function $q\mapsto\dim\uu_q^1$ is continuous on $(1,\infty)$, $\dim\uu_q^1=0$ for $q\in(0,q_{KL}]$, $\dim\uu_q^1=1$ for $q=M+1$, and $0<\dim\uu_q^1<1$ otherwise.
Finally, the  function $q\mapsto\dim\uu_q^1$ has a strange ``Devil's staircase'' behavior: it is continuous, and it has a strictly negative derivative almost everywhere, but
\begin{equation*}
q_{KL}<M+1
\qtq{and}
\dim\uu_{q_{KL}}^1=0<1=\dim\uu_{M+1}^1.
\end{equation*}

For $j\ge 2$ the sets $\uu_q^j$ behave differently.
For example, we have $\dim\uu_q^2<\dim \overline{\uu_q^2}$ in the classical case $q=M+1$ because $\uu_q^2$ is countable, while $\overline{\uu_q^2}=[0,1]$, so that $\dim\uu_q^2=0$ and $\dim \overline{\uu_q^2}=1$.

Another interesting observation was made in \cite[p. 410]{DK2009}: if $q\in\uu$, then $\uu_q^1\subset\overline{\uu_q^2}$.

Let us return to the sets $\uu_q^j$ without closure.
If $x\in\uu_q^j$ with $1<j<2^{\aleph_0}$, then a  bifurcation argument of Sidorov \cite{S2009} shows that $x$ has two expansions $(c_i)$ and $(d_i)$ such that the expansions $(c_{n+i})$ and $(d_{n+i})$ are unique for some $n\ge 1$, i.e., $(c_{n+i})_q,(d_{n+i})_q\in \uu_q^1$, and this implies the inequalities
\begin{equation}\label{11}
\dim_H\uu_q^j\le \dim_H\uu_q^1\qtq{for all}1<j<2^{\aleph_0}.
\end{equation}
If $q\le q_{KL}$, then all these dimensions are equal to zero, and hence equality holds in \eqref{11}.
On the other hand, all the inequalities \eqref{11} are strict if $q\ge M+1$ because all but countably many expansions are unique, so that $\dim_H\uu_q^1=1$, and $\dim_H\uu_q^j=0$ for all $j\ge 2$.

Concerning the intermediary case $q\in(q_{KL},M+1)$, an intriguing example was found by Sidorov  \cite{S2009}: if $M=1$ and $q\approx 1.83929$ is the \emph{Tribonacci number}, i.e., the positive root of the equation $q^3=q^2+q+1$, then
\begin{equation}\label{12}
\dim_H\uu_q^j=\dim_H\uu_q^1\qtq{for all}j\ge 2.
\end{equation}
The main purpose of this work is to extend the validity of the equalities \eqref{12} to a large infinite set of bases in $(q_{KL},M+1)$.
We need some notations.
First, following \cite{KL2007} we introduce the set $\vv'$ of sequences $(c_i)$ satisfying the following two lexicographic conditions:
\begin{equation*}
\begin{split}
&c_{n+1}c_{n+2}\cdots\le c_1c_2\cdots\qtq{whenever}c_n<M;\\
&\overline{c_{n+1}c_{n+2}\cdots}\cdots\le c_1c_2\cdots\qtq{whenever}c_n>0.
\end{split}
\end{equation*}

Now we turn back to the graphs $\gg(q)$.
We denote by $\tilde\gg(q)$ the subgraph of $\gg(q)$, obtained  by keeping only the vertices that are subintervals of $(b_1,a_1)$.
Furthermore, we denote by $\tilde \gg_{q}'$ the set of sequences generated by $\tilde\gg(q)$, and we set $\tilde \gg_{q}:=\set{(c_i)_q: (c_i)\in \tilde \gg_{q}'}$.

We recall from \cite[Definition 2.2.13]{LM1996} that a graph is called \emph{ strongly connected} if for every ordered pair $(I,J)$ of vertices $I, J$  there exists a path in the graph starting at $I$ and terminating at $J$.

\begin{theorem}\label{t18}
Let $q_0\in \uuu\setminus\uu$, and assume that $\tilde\gg(q_0)$ is strongly connected.
Write $\beta(q_0)=a_1\cdots a_{N-1}a_N^+\ 0^{\infty}$, and introduce a sequence of bases $q_0=r_0<r_1<r_2<\cdots$ by the formula
\begin{equation*}
\beta(r_k)=\alpha_1\cdots \alpha_{N-1}\alpha_N^+\left(\overline{\alpha_1\cdots \alpha_N}\right)^k\ 0^{\infty},\quad k=0,1,,\ldots .
\end{equation*}
The the equalities \eqref{12} hold for all bases $r_k$.
\end{theorem}

\begin{remark}\label{r19}\mbox{}
\begin{enumerate}[\upshape (i)]
\item The sequence $\alpha_1\cdots \alpha_{N-1}\alpha_N^+\left(\overline{\alpha_1\cdots \alpha_N}\right)^k\ 0^{\infty}$ is the greedy expansion of $x=1$ in base $r_k$ indeed, so that the notation is correct.

\item Theorem \ref{t18} contains Sidorov's results \cite[Proposition 4.6, Remark 4.7]{S2009} because if $M=1$ and $q_0$ is the Tribonacci number, then $\tilde\gg(q_0)$ is strongly connected: see Figure 1.1.

\item  With the notations of Theorem \ref{t14} we have $r_0=q_0$, $r_1=q_1$, and $r_k>q_0^*$ for all $k\ge 2$, so that Theorem \ref{t18} does not apply for the bases $q_2, q_3,\ldots .$
The inequality $r_k>q_0^*$ follows from the fact that $r_k\in\uuu\setminus\uu$ for all $k\ne 1$, see Lemma \ref{l94}.

\item In Section \ref{s12} we will give a necessary and sufficient condition for the strong connectedness of $\tilde\gg(q_0)$ (Proposition \ref{p121}), and we will compare this with a related, but different theorem of Alcaraz Barrera, Baker and Kong \cite{ABBK2019}.
\end{enumerate}
\end{remark}

We may also establish the equalities \eqref{12} in some cases where the conditions of Theorem \ref{t18} are not satisfied.

We denote by $\tilde\gg_1(q)$ the subgraph of $\tilde\gg(q)$,  formed by the vertices of the form $(a_i^-, a_i)$ and $(b_i, b_i^+)$, i.e., whose right endpoints belong to $\set{a_1,\ldots,a_N}$ or whose left endpoints belong to $\set{b_1,\ldots,b_N}$.
Furthermore, we denote by $\tilde \gg_{1,q}'$ the set of sequences generated by $\tilde\gg_1(q)$, and we set $\tilde \gg_{1,q}:=\set{(c_i)_q: (c_i)\in \tilde \gg_{1,q}'}$.

\begin{theorem}\label{t110}\mbox{}
If $q_0\in \uuu\setminus\uu$, and  $\dim\tilde\gg_{1, q_0}=\dim\uu_{q_0}$, then \eqref{12} holds for all bases $q=r_k$ with the notations of Theorem \ref{t18}.
\end{theorem}

\begin{example}
If $M=4$ and $\beta(q_0)=322\ 0^\infty$, then  $\tilde \gg(q_0)$ is not strongly connected (see Figure 1.4), and $\dim\tilde \gg_{1,q_0}=\dim \uu_{q_0}$.
Hence Theorem \ref{t18} does not apply, but the equalities \eqref{12} hold by Theorem \ref{t110} for $q_0=r_0\in\uuu\setminus\uu$ and $q=r_k$ with $\beta(r_k)=322\ (123)^k\ 0^\infty$.
See Example \ref{e1210} for the details.
\end{example}

\begin{figure}[ht]\label{f14}
\centering
\includegraphics[scale=0.45 ]{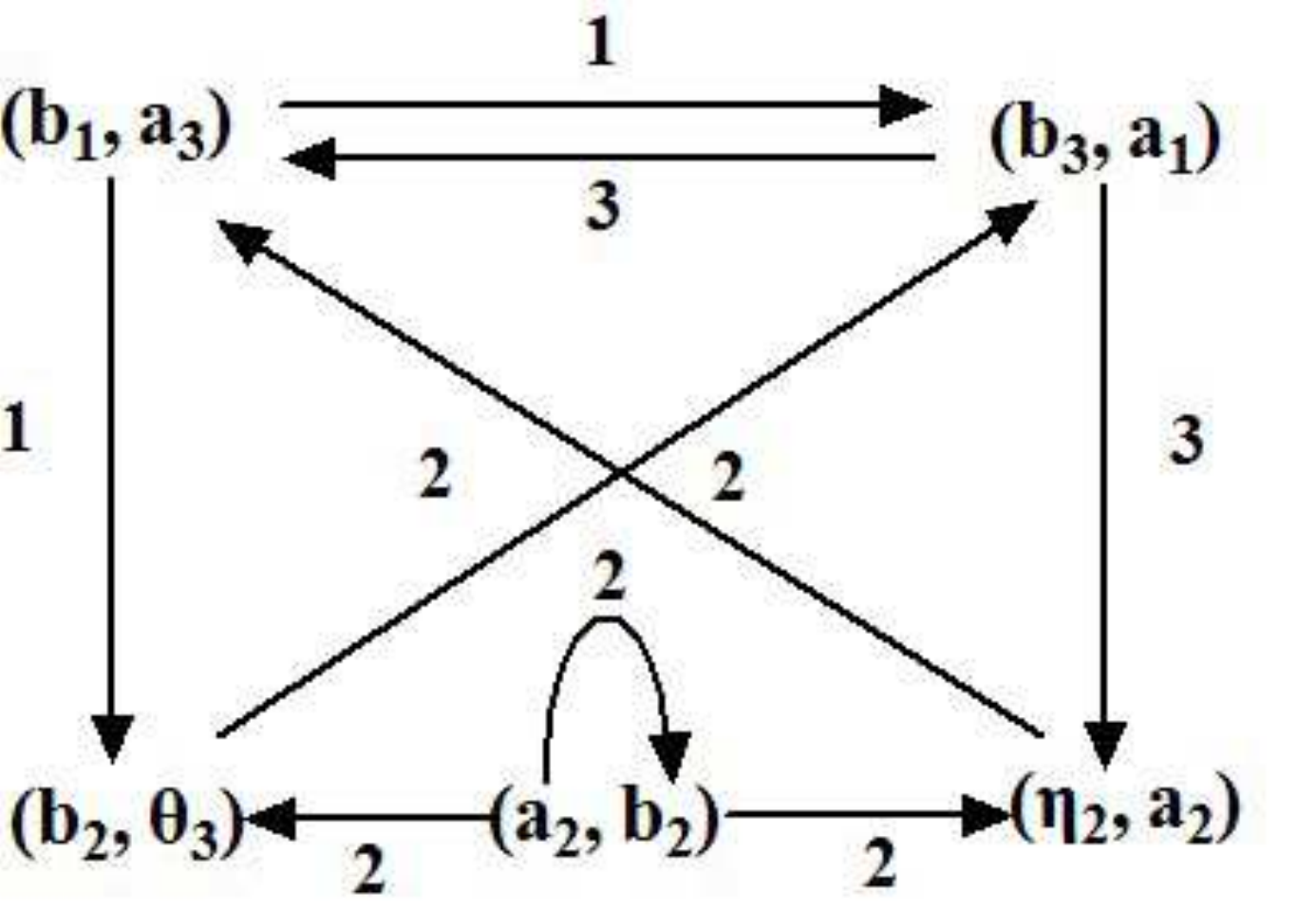}
\caption{The graph $\tilde\gg(q_0)$ associated with  $\beta(q_0)=322\ 0^{\infty}$.}
\end{figure}

It is an interesting open question whether  \eqref{12} holds for  all $q\in\vv\setminus\uu$.
We explain in Remark \ref{r1002} below why our proof breaks down in the general case.

The structure of the paper is the following.
In the next section we collect, for the reader's convenience, a number of earlier results that we will need in the sequel.
The univoque graph is introduced in Section \ref{s3}, and many preliminary results are established.
The main theorems are proved in Sections \ref{s4}--\ref{s10}.
In the last two sections we illustrate our theorems by a number of concrete examples.

\section{A review of quasi-greedy expansions, and a list of notations}\label{s2}

For the reader's convenience we briefly recall  some notions and basic results on expansions that we need later.
For a systematic treatment we refer to the papers \cite{BK2007,
K2011,
DK2016,
DKL2016}.
At the end of the section we give a list of the principal notations used in this paper (some of them are explained more precisely in the following sections).

\subsection{A review of some earlier results}\label{ss21}
We fix a positive integer $M$, and by a \emph{sequence} we mean an element of $\set{0,\ldots,M}^{\NN}$.
A sequence will usually be denoted by $(c_i)$ or $c_1c_2\cdots ,$ and we also use the notation from symbolic dynamics; for example, the sequence $101010\cdots$ is also denoted by $(10)^{\infty}$.

For any fixed base $1<q\le M+1$, every $x\in[0,\frac{M}{q-1}]$ has a lexicographically largest expansion $b(x,q)=(b_i)$, obtained by the greedy algorithm, and a lexicographically largest infinite expansion $a(x,q)=(a_i)$.
They are called the \emph{greedy} and \emph{quasi-greedy expansions} of $x$ in base $q$, respectively.
If the greedy expansion is infinite, then they coincide; otherwise $b(x,q)$ has a last nonzero digit $b_n$, and then $a(x,q)=b_1\cdots b_{n-1}b_n^-a(1,q)$.

The case $x=1$ being particularly important, we introduce the shorter notations $\beta(q):=b(1,q)$ and $\alpha(q):=a(1,q)$.

There is a very useful lexicographic characterization of greedy and quasi-greedy expansions in terms of the sequence $\alpha(q)$:

\begin{proposition}\label{p21}
Fix $q\in (1, M+1]$ arbitrarily.
\begin{enumerate}[\upshape (i)]
\item The map $x\mapsto b(x,q)$ is an increasing bijection between the interval $[0,\frac{M}{q-1}]$ and the sequences $(b_i)$ satisfying the lexicographic inequalities
\begin{equation*}
(b_{n+i})<\alpha(q)\qtq{whenever}b_n<M.
\end{equation*}
\item The map $x\mapsto a(x,q)$ is an increasing bijection between the interval $[0,\frac{M}{q-1}]$ and the infinite sequences $(a_i)$ satisfying the lexicographic inequalities
\begin{equation}\label{21}
(a_{n+i})\le \alpha(q)\qtq{whenever}a_n<M.
\end{equation}
\end{enumerate}
\end{proposition}

We note that the condition \eqref{21} implies the a priori stronger condition
\begin{equation*}
(a_{n+i})\le \alpha(q)\qtq{whenever}a_1\cdots a_n<M^n.
\end{equation*}

An expansion of $x$ is unique if and only if it is at the same time the lexicographically largest and smallest (called \emph{lazy}) expansion of $x$.
Since $(c_i)$ is an expansion of $x$ if and only if $(M-c_i)$ is an expansion of $\frac{M}{q-1}-x$, and since this map changes the sense of lexicographic inequalities, we deduce from the preceding proposition the

\begin{corollary}\label{c22}
An expansion $(c_i)$ of a number $x$ in a base $q$ is the unique one if and only if the following two conditions are satisfied:
\begin{align*}
&(c_{n+i})<\alpha(q)\qtq{whenever}c_n<M
\intertext{and}
&\overline{(c_{n+i})}<\alpha(q)\qtq{whenever}c_n>0.
\end{align*}
\end{corollary}

There is a similar characterization of the sequences $\alpha(q)$.
Setting $\beta(1)=10^{\infty}$ and  $\alpha(1)=0^{\infty}$ for commodity, we have the following

\begin{proposition}\label{p23}\mbox{}
\begin{enumerate}[\upshape (i)]
\item The map $q\mapsto \beta(q)$ is an increasing bijection between the interval $[1,M+1]$ and the set of  sequences $(\beta_i)$ satisfying the lexicographic inequalities
\begin{equation*}
(\beta_{n+i})<(\beta_i)\qtq{whenever}\beta_n<M.
\end{equation*}

\item The map $q\mapsto \alpha(q)$ is an increasing bijection between the interval $[1,M+1]$ and the set of infinite sequences $(\alpha_i)$ satisfying the lexicographic inequalities
\begin{equation*}
(\alpha_{n+i})\le (\alpha_i)\qtq{whenever}\alpha_n<M.
\end{equation*}
Moreover, in this case these inequalities are satisfied for all $n\ge 0$.
\end{enumerate}
\end{proposition}

Using the sequences $\alpha(q)$ and $\beta(q)$ we may also characterize the sets $\uu$, $\uuu$ and $\vv$:

\begin{proposition}\label{p24}
Let $q\in (1,M+1]$ and write $\alpha(q)=(\alpha_i)$, $\beta(q)=(\beta_i)$.
We have
\begin{align*}
q\in\uu&\Longleftrightarrow \overline{(\beta_{n+i})}< (\beta_i)\qtq{whenever}\beta_n>0;\\
q\in\uuu&\Longleftrightarrow \overline{(\alpha_{n+i})}< (\alpha_i)\qtq{whenever}\alpha_n>0;\\
q\in\vv&\Longleftrightarrow \overline{(\alpha_{n+i})}\le (\alpha_i)\qtq{whenever}\alpha_n>0.
\end{align*}
Moreover, in each case the inequalities are satisfied for all $n\ge 0$.
\end{proposition}

We recall that $\uu$ has a smallest element: $q_{KL}$.
The set $\vv$ also has a smallest element.
Since it is equal to the Golden Ratio if $M=1$, it is called the \emph{generalized Golden Ratio} (the terminology comes from  \cite{KLP2011}), and it is denoted by $q_{GR}$.
We have
\begin{equation*}
q_{GR}=
\begin{cases}
(m+\sqrt{m^2+4m})/2\qtq{with}\beta(q_{GR})=mm\ 0^{\infty}&\text{if $M=2m-1$,}\\
m+1\qtq{with}\beta(q_{GR})=(m+1)\ 0^{\infty}&\text{if $M=2m$,}
\end{cases}
\end{equation*}
$m=1,2,\ldots .$

We will need  the following  properties of  $\alpha(q)$ where $q\in\vv\setminus\uu$:

\begin{proposition}\label{p25}
Let $q\in\vv\setminus\uu$ and write $\alpha(q)=\alpha_1\alpha_2\cdots .$
\begin{enumerate}[\upshape (i)]
\item The sequence $\alpha(q)$ is periodic, say $\alpha_1\alpha_2\cdots=(\alpha_1\cdots\alpha_N)^{\infty}$,  and $\alpha_N<M$.
\item If $\alpha(q)=(\alpha_1\cdots\alpha_N)^{\infty}$ with the shortest period $N$, and $q^+=\min(\vv\cap(q,\infty))$, then
\begin{equation*}
\alpha(q^+)=\left(\alpha_1\cdots\alpha_N^+\overline{\alpha_1\cdots\alpha_N^+}\right)^{\infty},
\end{equation*}
and hence $q^+\in\vv\setminus\uuu$.
\item If $M$ is even and $q=\min\vv$, then $\overline{\alpha_1}=\alpha_1$; otherwise we have $\overline{\alpha_1}<\alpha_1$.
\item If $q\in\uuu\setminus\uu$, then $\alpha(q)$ is also the lazy expansion of  $x=1$ in base $q$.
\item If $q\in\vv\setminus\uuu$, $q>\min\vv$ and
\begin{equation*}
\alpha(q)=\left(\alpha_1\cdots\alpha_N\overline{\alpha_1\cdots\alpha_N}\right)^{\infty}
\end{equation*}
where $N$ is chosen to be minimal, then
\begin{equation*}
\overline{\alpha_{i+1}\cdots\alpha_N}<\alpha_1\cdots\alpha_{N-i},\quad i=0,\ldots,N-1.
\end{equation*}
\end{enumerate}
\end{proposition}

\begin{example}\label{e26}
Consider for $M=1$ the first interval  $(q_0,q_0^*)=(1,q_{KL})$ of $(1,\infty)\setminus\uuu$, and let $q_1<q_2<\cdots$ be an enumeration of the elements of $\vv\cap(1,q_{KL})$.
Then
\begin{align*}
\alpha(q_1)&=(10)^{\infty},\\
\alpha(q_2)&=(1100)^{\infty},\\
\alpha(q_3)&=(11010010)^{\infty},\ldots .
\end{align*}
Similarly, if we take the connected component $(q_0,q_0^*)$ of $(1,\infty)\setminus\uuu$ where $q_0$ is the Tribonacci number, then
\begin{align*}
\alpha(q_0)&=(110)^{\infty},\\
\alpha(q_1)&=(111000)^{\infty},\\
\alpha(q_2)&=(111001000110)^{\infty},\ldots .
\end{align*}
\end{example}

As usual, we write $\uu_q$ instead of $\uu_q^1$ for the set of numbers $x$ having a unique expansion in base $q$, and we denote by $\uu_q'$ the set of corresponding expansions (sequences).
Motivated by Proposition \ref{p24} we \emph{define} the set $\vv_q$ by the relation
\begin{equation}\label{22}
x\in\vv_q\Longleftrightarrow \overline{(a_{n+i}(x,q))}\le\alpha(q)\qtq{whenever}a_n(x,q)>0,
\end{equation}
and we denote by $\vv_q'$ the set of corresponding quasi-greedy expansions $a(x,q)$.
By  Proposition \ref{p210} (vii) below  this coincides with the definition of $\vv_q'$ given in the introduction.

We recall from \cite{DK2009, KKL2017} the following three results:

\begin{proposition}\label{p27}
Let $q\in(1,M+1]$.
\begin{enumerate}[\upshape (i)]
\item $\uu_q\subset\overline{\uu_q}\subset\vv_q=\overline{\vv_q}$, and $\vv_q\setminus\uu_q$ is countable.
\item If $q\in\uuu$, then $\uu_q\ne\overline{\uu_q}=\vv_q$.
\item If $q\in\vv\setminus\uuu$, then $\uu_q=\overline{\uu_q}\ne\vv_q$.
\item If $q\in(1,M+1]\setminus\vv$, then $\uu_q=\overline{\uu_q}=\vv_q$.
\end{enumerate}
\end{proposition}

Let us describe the structure of $\vv\setminus\uu$ more closely:

\begin{proposition}\label{p28}
Let $(q_0,q_0^*)$ be an arbitrary connected component of $(1,M+1)\setminus\uuu$.
\begin{enumerate}[\upshape (i)]
\item If $(q_0,q_0^*)$ runs over the connected components of $(1,M+1)\setminus\uuu$, then $q_0$ runs over $\set{1}\cup(\uuu\setminus\uu)$, and $q_0^*$ runs over a subset of $\uu$.
\item For each component $(q_0,q_0^*)$,
\begin{equation*}
\vv\cap(q_0,q_0^*)=(\vv\setminus\uuu)\cap(q_0,q_0^*)
\end{equation*}
is formed by an increasing sequence $(q_n)_{n=1}^{\infty}$, converging to $q_0^*$:
\begin{equation*}
q_0<q_1<q_2<\cdots,\quad q_n\nearrow q_0^*.
\end{equation*}
\item We have the following relations:
\begin{equation*}
\uu_{q_0}'\subsetneqq
\vv_{q_0}'=\uu_{q_1}'\subsetneqq
\vv_{q_1}'=\uu_{q_2}'\subsetneqq
\vv_{q_2}'=\uu_{q_3}'\subsetneqq\cdots ,
\end{equation*}
and
\begin{equation*}
\vv_{q_n}'=\uu_q'=\vv_q'=\uu_{q_{n+1}}'\qtq{whenever}q_n<q<q_{n+1},\quad n=0,1,\ldots .
\end{equation*}
\item $\uu_{q_0}$ is not closed, and $\overline{\uu_{q_0}}=\vv_{q_0}$.
\item If $n\ge 1$, then  $\uu_{q_n}$ and $\vv_{q_n}$ are different closed sets.
\end{enumerate}
\end{proposition}

We denote by $h({\mathcal W})$ the topological entropy of a set ${\mathcal W}$ of sequences, if it exists.
It always exist if ${\mathcal W}$ is a subshift of finite type, but it may exist in other cases, too.
For example, $h(\uu_q')$ does exist even if $\uu_q'$ is not a subshift of finite type.

\begin{proposition}\label{p29}
Let $q\in(1,M+1]$.
\begin{enumerate}[\upshape (i)]
\item $\overline{\uu_q}'$ and $\vv_q'$ are subshifts of finite type.
Moreover,
\begin{equation*}
h(\uu_q')=h(\overline{\uu_q}')
=h(\vv_q'),
\end{equation*}
and the entropy function $q\mapsto h(\uu_q')$ is constant on each connected component $(q_0,q_0^*)$ of $(1,M+1)\setminus\uuu$.
\item $\uu_q$, $\overline{\uu_q}$ and $\vv_q$ have the same Hausdorff dimension, given by the formula
\begin{equation*}\label{24}
\dim \uu_q=\dim\overline{\uu_q}
=\dim \vv_q
=\frac{h(\uu_q')}{\log_2q}.
\end{equation*}
\end{enumerate}
\end{proposition}

We will also need from \cite{K2012} the following characterization of the sets $\uu$, $\uuu$, $\vv$, $\uu_q$, $\overline{\uu_q}$ and $\vv_q$ that do not use lexicographic relations.

\begin{proposition}\label{p210}\mbox{}
Let $q\in (1,M+1)$ and $x\in\left[0,\frac{M}{q-1}\right]$.

\begin{enumerate}[\upshape (i)]
\item The quasi-greedy expansion $a(x,q)$ is doubly infinite for every $x\in\left[0,\frac{M}{q-1}\right]$.
\item $q\in\uu\Longleftrightarrow y=1$ has a unique expansion.
\item $q\in\uuu\Longleftrightarrow y=1$ has a unique infinite expansion.
\item $q\in\vv\Longleftrightarrow y=1$ has a unique doubly infinite expansion.
\item $x\in\uu_q\Longleftrightarrow x$ has a unique expansion.
\item $x\in\overline{\uu_q}\Longleftrightarrow x$ or $\frac{M}{q-1}-x$ has a unique infinite expansion.
More precisely, if $x$ has a finite greedy expansion, then it has a unique infinite expansion; otherwise $\frac{M}{q-1}-x$ has a finite greedy expansion.
\item $x\in\vv_q\Longleftrightarrow x$ has a unique doubly infinite expansion.\end{enumerate}
\end{proposition}

\begin{remark}\label{r211}
The second part of (vi) follows from the proof given in \cite{K2012}.
Properties (ii)-(vii) remain valid for $q\in(1,M+1]$ if we change (vii) to
\begin{equation*}
x\in\vv_q\Longleftrightarrow x\qtq{has at most one doubly infinite expansion;}
\end{equation*}
see \cite{Steiner}.
\end{remark}

\subsection{List of the principal notations}\label{ss22}

\begin{itemize}
\item Words and expansions
\begin{itemize}
\item $(c_i)_q:=\sum_{i=1}^{\infty}\frac{c_i}{q^i}$.
\item $\overline{c_i}:M-c_i$ denotes the reflection of the digit $c_i$.
\item $\overline{w}=\overline{c_1\cdots c_n}:=\overline{c_1}\ \cdots\ \overline{c_n}$ denotes the reflection of the word $w=c_1\cdots c_n$.
\item $\overline{(c_n)}:=(\overline{c_n})$ denotes the reflection of the sequence $(c_n)$.
\item For a word $w=\omega_1\cdots\omega_n$, $w^+=\omega_1\cdots\omega_n^+$ denotes $\omega_1\cdots\omega_{n-1}(\omega_n+1)$ if $\omega_n<M$, and $w^-=\omega_1\cdots\omega_n^-$ denotes $\omega_1\cdots\omega_{n-1}(\omega_n-1)$ if $\omega_n>0$.
\item $b(x,q)=(b_i)$ denotes the lexicographically largest (or greedy) expansion  of  $x\in[0,\frac{M}{q-1}]$ in base $q$.
\item $a(x,q)=(a_i)$ denotes the lexicographically largest infinite (or quasi-greedy) expansion  of  $x\in[0,\frac{M}{q-1}]$ in base $q$; it is doubly infinite if $q\in(1,M+1)$.
\item $\beta(q)=b(1,q)$ denotes the greedy expansion  of $1$ in base $q$.
\item $\alpha(q)=a(1,q)$ denotes the quasi-greedy expansion  of $1$ in base $q$; it is always doubly infinite.
\end{itemize}
\medskip

\item Graphs
\begin{itemize}
\item If $\beta(q)=\alpha_1\cdots\alpha_N^+0^\infty$, then
$a_i=(\alpha_i\cdots\alpha_N^+0^\infty)_q$ and $b_i=\frac{M}{q-1}-a_i$ for $i=1,\ldots, N$.
\item  $\theta_j=\frac{j}{q}$ for $j=0,\cdots, M$; in particular, $\theta_0=0$.
\item  $\eta_j=\frac{j-1}{q}+\frac{M}{q^2-q}$ for $j=1,\cdots, M+1$; in particular, $\eta_{M+1}=\frac{M}{q-1}$.
\item $S_q=:\bigcup_{j=1}^M[\theta_j,\eta_j]$ denotes the switch region.
\item $\gg(q)$ denotes the labeled graph whose vertices are the connected components of the set
$\left(0,\frac{M}{q-1}\right)\setminus (S_q\cup\set{a_1,\dots, a_N,b_1,\ldots, b_N})$; its edges are defined in Section \ref{s3} below.
\item $\gg_q'$ denotes the set of sequences generated by $\gg(q)$.
\item $\gg_q:=\set{(c_i)_q\ :\ (c_i)\in\gg_q'}$.
\item $\tilde\gg(q)$ denotes the subgraph of $\gg(q)$, obtained  by keeping only the vertices that are subintervals of $(b_1,a_1)$.
\item $\tilde\gg_q'$ denotes the set of sequences generated by $\tilde\gg(q)$.
\item $\tilde\gg_q:=\set{(c_i)_q\ :\ (c_i)\in\tilde\gg_q'}$.
\item $\tilde\gg_1(q)$ denotes the subgraph of $\tilde\gg(q)$  formed by the vertices of the form $(a_i^-, a_i)$ and $(b_i, b_i^+)$.
\item $\tilde \gg_{1,q}'$ denote the set of sequences generated by $\tilde\gg_1(q)$.
\item $\tilde \gg_{1,q}:=\set{(c_i)_q: (c_i)\in \tilde \gg_{1,q}'}$.
\end{itemize}
\medskip

\item Univoque and related sets
\begin{itemize}
\item $\uu=\set{q\in(1, M+1]: 1 \text{\ has  a unique expansion in base\ }q}$.
\item $\uuu=\set{q\in(1, M+1]: 1 \text{\ has  a unique infinite expansion in base\ }q}$.
\item $\vv=\set{q\in(1, M+1]: 1 \text{\ has  a unique doubly infinite expansion in base\ }q}$.
\item $\uu'=\set{\alpha(q)\ :\ q\in \uu}$ (also the set of unique expansions of $x=1$ in bases $q\in\uu$).
\item $\uuu'=\set{\alpha(q)\ :\ q\in \uuu}$ (also the set of unique infinite expansions of $x=1$ in bases $q\in\uu$).
\item $\vv'=\set{\alpha(q)\ :\ q\in \vv}$ (also the set of doubly infinite unique expansions of $x=1$ in bases $q\in\vv$).
\item $\uu_q$  denotes the univoque set, i.e., the set of numbers $x$ having a unique expansion in base $q$.
We recall that $x\in\uu_q\Longleftrightarrow \frac{M}{q-1}-x\in\uu_q$.
\item $\overline{\uu_q}$ denotes  the topological closure of $\uu_q$.
\item $\vv_q$  denotes the  set of numbers $x$ having a unique doubly infinite expansion  in base $q$.
\item $\tilde\vv_q=
\set{x\in\vv_q\ :\ \overline{\alpha(q)}\le(a_{n+i}(x,q))\le\alpha(q)\qtq{for all}n\ge 0}$.
\item $\uu_q'=
\set{a(x,q)\ :\ x\in \uu_q}$.
\item $\uuu_q'=
\set{a(x,q)\ :\ x\in \uuu_q}$.
\item $\vv_q'=
\set{a(x,q)\ :\ x\in \vv_q}$.
\item $\tilde\vv_q'=
\set{a(x,q)\ :\ x\in \tilde\vv_q}$.
\item $\uu_q^j:=\set{x\ :\ x \text{ has exactly\ } j\text{ expansions}},\ j=1,2,\ldots, \aleph_0 \text{\ or\ }2^{\aleph_0}$.
\item $\overline{\uu_q^j}$ denotes  the topological closure of $\uu_q^j$.
\end{itemize}
\end{itemize}

\section{Definition of the univoque graphs for $q\in\vv\setminus\uu$}\label{s3}

Given a base $1<q<M+1$ and a sequence $(c_i)\in\set{0,\ldots,M}^{\infty}$, we define the numbers
\begin{equation*}
x_k:=(c_kc_{k+1}\cdots)_q,\quad k=1,2,\ldots .
\end{equation*}
We recall (see, e.g., \cite{S2003b}) that if some $x_k$ belongs to the \emph{switch region}
\begin{equation*}
S_q:=\bigcup_{j=1}^M\left[\frac{j}{q},\frac{j-1}{q}+\frac{M}{q^2-q}\right]
=:[\theta_j,\eta_j],
\end{equation*}
then $x_k$ has another expansion $d_kd_{k+1}\cdots$ with $d_k\ne c_k$.
Otherwise, every expansion of $x_k$ must start with the digit $c_k$.
We have thus  the following simple result:

\begin{lemma}\label{l31}
An expansion $(c_i)$ of a number $x_1$ is the unique one if and only if all numbers $x_i:=(c_ic_{i+1}\cdots)_q$ lie outside the switch region.
\end{lemma}

Henceforth, in this section we assume  that $q\in\vv\setminus\uu$.
Since $\max\vv=\max\uu=M+1$ and $1<\min\vv<\min\uu$, this implies that $\min\vv\le q<M+1$.
For convenience we also introduce the numbers $\theta_0:=0$ and $\eta_{M+1}:=\frac{M}{q-1}$.

\begin{remark}\label{r32}\mbox{}
\begin{enumerate}[\upshape (i)]
\item The  intervals in the definition of the switch region do not overlap.
Indeed, if $M=2m-1$ is odd, then using the generalized Golden Ratio $q_{GR}=\min\vv$ we have
\begin{equation*}
q\ge\min\vv=q_{GR}=\frac{m+\sqrt{m^2+4m}}{2}
>\frac{2m+1}{2},
\end{equation*}
and hence
\begin{equation*}
\theta_{j+1}-\eta_j
=\frac{2(q-1)-(2m-1)}{q^2-q}
=\frac{2q-2m-1}{q^2-q}>0
\end{equation*}
for all $1\le j\le M-1$.
If $M=2m$ is even, then
\begin{equation*}
q\ge\min\vv=m+1,
\end{equation*}
and hence
\begin{equation*}
\theta_{j+1}-\eta_j
=\frac{2(q-1)-2m}{q^2-q}
=\frac{2(q-m-1)}{q^2-q}\ge 0
\end{equation*}
for all $1\le j\le M-1$.

Observe that if $M=2m$ is even and $q=q_{GR}=m+1=\min\vv$, then
\begin{equation*}
0=\theta_0<\theta_1<\eta_1
=\theta_2<\eta_2
=\cdots<\eta_{M-1}
=\theta_M<\eta_M<\eta_{M+1}=\frac{M}{q-1}
\end{equation*}
and $S_q=[\theta_1,\eta_M]$; otherwise all points $\theta_j$ and $\eta_j$ are distinct:
\begin{equation*}
0=\theta_0<\theta_1<\eta_1
<\theta_2<\eta_2
<\cdots<\eta_{M-1}
<\theta_M<\eta_M<\eta_{M+1}=\frac{M}{q-1}.
\end{equation*}
\item It follows from the definition of the switch region that every expansion of $x$ starts with the digit
\begin{equation*}
\begin{cases}
0\qtq{if}x\in[\theta_0,\theta_1),\\
j-1\qtq{if}x\in(\eta_{j-1},\theta_j),\quad j=2,\ldots,M,\\
M\qtq{if}x\in(\eta_M,\eta_{M+1}].
\end{cases}
\end{equation*}
\end{enumerate}
\end{remark}

Let us introduce the increasing affine maps
\begin{equation*}
T_j(x):=qx-j,\quad j=0,1,\ldots, M;
\end{equation*}
we will often write $x\xmapsto{j} y$ instead of $T_j(x)=y$.

\begin{lemma}\label{l33}\mbox{}
\begin{enumerate}[\upshape (i)]
\item For any $x,i,y$  the following properties hold with $\overline{i}:=M-i$:
\begin{equation*}
x\xmapsto{i} y\Longleftrightarrow \frac{M}{q-1}-x\xmapsto{\overline{i}}\frac{M}{q-1}-y.
\end{equation*}
\item We have
\begin{equation*}
0\xmapsto{0}0,\qtq{and}
\theta_j\xmapsto{j}0,\quad
\theta_j\xmapsto{j-1}1\qtq{for}j=1,\ldots,M.
\end{equation*}
\item We have
\begin{equation*}
\eta_j=\frac{M}{q-1}-\theta_{M+1-j}\qtq{for}j=1,\ldots,M,
\end{equation*}
and
\begin{equation*}
\frac{M}{q-1}\xmapsto{M}\frac{M}{q-1},\qtq{and}
\eta_j\xmapsto{j-1}\frac{M}{q-1},\quad
\eta_j\xmapsto{j}\frac{M}{q-1}-1\qtq{for}j=1,\ldots,M.
\end{equation*}
\end{enumerate}
\end{lemma}

\begin{proof}
(i)  follows from the equivalence
\begin{equation*}
qx-i=y\Longleftrightarrow
q\left(\frac{M}{q-1}-x\right)-\overline{i}=\frac{qM}{q-1}-M-(qx-i)=\frac{M}{q-1}-y.
\end{equation*}

(ii) follows from the equalities
\begin{equation*}
T_k(\theta_j)=q\left(\frac{j}{q}\right)-k=j-k.
\end{equation*}

(iii) We have
\begin{equation*}
\eta_j=\frac{j-1}{q}+\frac{M}{q^2-q}
=\frac{j-1}{q}+\frac{M}{q-1}-\frac{M}{q}
=\frac{M}{q-1}-\frac{M+1-j}{q}
=\frac{M}{q-1}-\theta_{M+1-j}.
\end{equation*}
The other properties follow from (i) and (ii).
\end{proof}

Now we assume that $q\in \vv\setminus\uu$, and we denote by
\begin{equation*}
\beta(q)=\alpha_1\cdots\alpha_N^+\ 0^{\infty}
\qtq{and}
\alpha(q)=(\alpha_1\cdots\alpha_N)^{\infty},
\end{equation*}
as before, the greedy and quasi-greedy expansions of $x=1$ in base $q$.

For example, the generalized Golden Ratios $q_{GR}$ belong to $\vv\setminus\uuu$ with $N=1$ and $\alpha_1^+=m+1$ if $M=2m$ is even, and with $N=2$ and $\alpha_1=\alpha_2^+=m$ if $M=2m-1$ is odd.
In case $M=1$ the Tribonacci number belongs to $\uuu\setminus\uu$ with $N=3$ and $\alpha_1=\alpha_2=\alpha_3^+=1$; otherwise we have $N\ge 4$ and $\alpha_1=\alpha_2=\alpha_N^+=1$.

Let us  introduce the points
\begin{equation*}
a_i:=(\alpha_i\cdots\alpha_N^+\ 0^{\infty})_q
\qtq{and}
b_i:=\frac{M}{q-1}-a_i
\qtq{for}
i=1,\ldots,N+1.
\end{equation*}
For example,
\begin{equation*}
a_1=1\qtq{and}b_1=\frac{M+1-q}{q-1}.
\end{equation*}
We have the obvious relations
\begin{equation*}
a_i<a_j \Longleftrightarrow b_j<b_i
\qtq{and}
a_i<b_j \Longleftrightarrow a_j<b_i.
\end{equation*}

\begin{lemma}\label{l34}
Let $q\in\vv\setminus\uu$.
\begin{enumerate}[\upshape (i)]
\item The greedy expansions of the numbers $a_i$ and $\theta_j$ are finite:
\begin{align*}
&b(a_i,q)=\alpha_i\cdots\alpha_N^+\ 0^{\infty}\qtq{for}i=1,\ldots,N-1,\\
&b(\theta_j,q)=j\ 0^{\infty}\qtq{for}j=0,1,\ldots,M.
\end{align*}
\item All numbers $a_i, b_i, \theta_j, \eta_j$ belong to $\vv_q$.
\item
The numbers $a_1,\ldots, a_{N-1}$, $\theta_0,\ldots,\theta_M$ and $\eta_{M+1}$ are  distinct.
Furthermore,
\begin{equation*}
a_N=\theta_{\alpha_N^+},
\quad a_{N+1}=\theta_0,
\quad b_N=\eta_{\overline{\alpha_N}}\qtq{and}
b_{N+1}=\eta_{M+1}.
\end{equation*}
\item Among the numbers $a_i, b_i$ the smallest two are $a_{N+1}<b_1$, and the greatest two are $a_1<b_{N+1}$.
\item We have
\begin{align*}
&a_1\xmapsto{\alpha_1} a_2
\xmapsto{\alpha_2}a_3\cdots a_{N-1}
\xmapsto{\alpha_{N-1}}a_N
\xmapsto{\alpha_N} a_1\qtq{and}
a_N\xmapsto{\alpha_N^+} 0
\intertext{and}
&b_1\xmapsto{\overline{\alpha_1}} b_2
\xmapsto{\overline{\alpha_2}}b_3\cdots b_{N-1}
\xmapsto{\overline{\alpha_{N-1}}}b_N
\xmapsto{\overline{\alpha_N}} b_1\qtq{and}
b_N
\xmapsto{\overline{\alpha_N^+}} \frac{M}{q-1}.
\end{align*}
\end{enumerate}
\end{lemma}

\begin{proof}
(i) They have different greedy expansions:
\begin{equation*}
\alpha_i\cdots\alpha_N^+\ 0^{\infty}, 1\le i\le N-1
\qtq{and}
j\ 0^{\infty}, 0\le j\le M.
\end{equation*}

\medskip

(ii) By the symmetry of $\vv_q$ it suffices to consider the numbers $a_i$ and $\theta_j$.
We have obviously $a_{N+1}=\theta_0=0\in\uu_q\subset\vv_q$.
Since $q\in\vv\setminus\uu$, $\alpha(q)$ satisfies the lexicographic conditions stated in Propositions \ref{p23} (ii) and \ref{p24}.
It follows that the quasi-greedy expansions
\begin{equation*}
\alpha_i\cdots\alpha_N(\alpha_1\cdots\alpha_N)^{\infty}\qtq{and}
j^-(\alpha_1\cdots\alpha_N)^{\infty}
\end{equation*}
of the remaining numbers $a_i$ for $\theta_j$ satisfy the definition of $\vv_q$ given in \eqref{22}.
\medskip

(iii) The first assertion follows by observing that by (i) the numbers $a_i, \theta_j$  have different finite greedy expansions, and $\frac{M}{q-1}$ has an infinite greedy expansion $M^{\infty}$.

The last four equalities follow from the equalities
\begin{align*}
&a_N=\frac{\alpha_N^+}{q},
\quad a_{N+1}=0,
\quad b_{N+1}=\frac{M}{q-1}-a_{N+1}=\frac{M}{q-1}
\intertext{and}
&b_N=\frac{M}{q-1}-\frac{\alpha_N^+}{q}=\frac{\overline{\alpha_N^+}}{q}+\frac{M}{q^2-q}.
\end{align*}
\medskip

(iv)
It is clear that $a_{N+1}=0$ is the smallest and $b_{N+1}=\frac{M}{q-1}$ is the greatest among these numbers.
By symmetry it remains to show that $a_i\le a_1$ and $b_i\le a_1$ for all $1\le i\le N$.
Since $q<M+1$, and since all  numbers $a_i, b_i$ belong to $\vv_q$, each of them  has a unique doubly infinite expansion: its quasi-greedy expansion.
In view of Proposition \ref{p21} (ii) it suffices to show that the corresponding lexicographic inequalities between their unique doubly infinite expansions, i.e., the inequalities
\begin{equation*}
\alpha_i\cdots\alpha_N(\alpha_1\cdots\alpha_N)^{\infty}
\le(\alpha_1\cdots\alpha_N)^{\infty}
\end{equation*}
and
\begin{equation*}
\overline{\alpha_i\cdots\alpha_N(\alpha_1\cdots\alpha_N)^{\infty}}
\le (\alpha_1\cdots\alpha_N)^{\infty}.
\end{equation*}
Since  $q\in\vv$, they follow from Propositions \ref{p23} (ii) and \ref{p24}.
\medskip

(v) We have
\begin{equation*}
T_{\alpha_i}(a_i)
=q(\alpha_i\cdots\alpha_N^+)_q-\alpha_i
=(\alpha_{i+1}\cdots\alpha_N^+)_q
=a_{i+1}
\end{equation*}
for $i=1,\ldots,N-1$, and (we recall that $\theta_{\alpha_N^+}=a_N$)
\begin{equation*}
T_{\alpha_N}(a_N)=qa_N-\alpha_N=1=a_1,\quad
T_{\alpha_N^+}(a_N)=qa_N-\alpha_N^+=0.
\end{equation*}
The remaining relations  follow by reflection, i.e., applying Lemma \ref{l33} (i).
\end{proof}

Now we mention some properties that are specific to the cases $q\in\uuu\setminus\uu$ and $q\in\vv\setminus\uuu$.

\begin{lemma}\label{l35}
Let $q\in\uuu\setminus\uu$ with
\begin{equation*}
\beta(q)=\alpha_1\cdots\alpha_N^+\ 0^{\infty}
\qtq{and}
\alpha(q)=(\alpha_i)=(\alpha_1\cdots\alpha_N)^{\infty}.
\end{equation*}
\begin{enumerate}[\upshape (i)]
\item
The numbers $a_i, b_i, \theta_j, \eta_j$ for $1\le i\le N-1$ and $1\le j\le M$ belong to $\uuu_q\setminus\uu_q$.
\item The greedy expansions of the numbers $b_i$ and $\eta_j$ are infinite:
\begin{align*}
&b(b_i,q)=\overline{\alpha_i\cdots\alpha_N(\alpha_1\cdots\alpha_N)^{\infty}}\qtq{for}i=1,\ldots,N-1,\\
&b(\eta_j,q)=j\ \overline{(\alpha_1\cdots\alpha_N)^{\infty}}\qtq{for}j=1,\ldots,M,\\
&b(\eta_{M+1},q)=M^{\infty}.
\end{align*}
\item The numbers
\begin{equation*}
\theta_0,\ldots, \theta_M,\quad
\eta_1,\ldots, \eta_{M+1},\quad
a_1,\ldots, a_{N-1}
\qtq{and}
b_1,\ldots, b_{N-1}
\end{equation*}
are  distinct.
\item
We have
\begin{equation*}
S_q\cap\set{a_i,b_i\ :\ 1\le i\le N+1}=\set{a_N,b_N}=
\set{\theta_{\alpha_N^+},\eta_{\overline{\alpha_N}}}.
\end{equation*}
\end{enumerate}
\end{lemma}

\begin{proof}
(i) By symmetry it suffices to consider $a_i$ and $\theta_j$.
They have non-zero greedy expansions, hence they have infinitely many other expansions, so that they do not belong to $\uu_q$.
Since $\uuu_q=\vv_q$ is by Proposition \ref{p27} (ii), it remains to show that they belong to $\vv_q$.
This follows from Lemma \ref{l34} (ii). \medskip

(ii) The case of $\eta_{M+1}=M/(q-1)$ is obvious.
We deduce from Lemma \ref{l34} (ii) that the quasi-greedy expansions of the numbers $a_i$ ($1\le i\le N-1$) and $\theta_{M+1-j}$ ($1\le j\le M$)  are
$\alpha_i\cdots\alpha_N(\alpha_1\cdots\alpha_N)^{\infty}$ and $(M-j)(\alpha_1\cdots\alpha_N)^{\infty}$.
They are also lazy by (i) and Proposition \ref{p25} (iv).
Taking reflections we get the indicated greedy expansions of the numbers $b_i$ and $\eta_j$, and none of them has a last nonzero digit.
\medskip

(iii) The numbers
\begin{equation*}
a_1,\ldots, a_{N-1}\qtq{and}\theta_0,\ldots,\theta_M
\end{equation*}
are distinct by Lemma \ref{l34} (iii).
Since they have finite greedy expansions, by (ii) the numbers
\begin{equation*}
a_1,\ldots, a_{N-1}, \theta_0,\ldots,\theta_M\qtq{and}b_1,\ldots, b_{N-1}
\end{equation*}
are also distinct.
Taking reflections hence we infer that the numbers
\begin{equation*}
a_1,\ldots, a_{N-1}, \eta_1,\ldots,\eta_{M+1}\qtq{and}b_1,\ldots, b_{N-1}
\end{equation*}
are also distinct.
We conclude by recalling from Remark \ref{r32} (i) that, since $q\in\uuu\setminus\uu$ and therefore $q\ne\min\vv$, the numbers $\theta_j$ and $\eta_k$ are pairwise distinct.
\medskip

(iv) It follows from (i), (ii) and the meaning of the switch region that
\begin{equation*}
S_q\cap\set{a_i,b_i\ :\ 1\le i\le N-1}=\varnothing.
\end{equation*}
Also, $a_{N+1}=0<\theta_1$ and $b_{N+1}=M/(q-1)>\eta_M$, so that $a_{N+1},b_{N+1}\notin S_q$.
On the other hand, $a_N=\theta_{\alpha_N^+}$ and $b_N=\eta_{\overline{\alpha_N}}$ belong to $S_q$ by definition.
\end{proof}

The situation is different if $q\in\vv\setminus\uuu$.
First we consider the case of the generalized Golden Ratio:

\begin{example}\label{e36}
Let $q=\min\vv$.
If $M=2m$ is even, then we recall from Remark \ref{r32} (i) the following relations:
\begin{align*}
&N=1\qtq{and}q=m+1,\\
&\beta(q)=(m+1)\ 0^{\infty},
\qtq{and} \alpha(q)=m^{\infty},\\
&\theta_0<\theta_1<\eta_1
=\theta_2<\eta_2
=\cdots<\eta_{M-1}
=\theta_M<\eta_M<\eta_{M+1},\\
&a_1=\theta_{m+1}=\eta_m=b_1=1,\\
&S_q=[\theta_1,\eta_M].
\end{align*}

If $M=2m-1$ is odd, then $q=\frac{m+\sqrt{m^2+4m}}{2}$, $N=2$ and
\begin{align*}
&\beta(q)=mm\ 0^{\infty},
\qtq{and} \alpha(q)=(m\overline{m})^{\infty},\\
&b(\eta_j,q)=jm\ 0^{\infty}\qtq{for}j=1,\ldots,M,\qtq{and}
b(\eta_{M+1},q)=M^{\infty},\\
&\theta_0<\theta_1<\eta_1
<\theta_2<\eta_2
<\cdots<\eta_{M-1}
<\theta_M<\eta_M<\eta_{M+1},\\
&b_1=a_2=\theta_m\qtq{and}b_2=a_1=\eta_m.
\end{align*}
By Remark \ref{r32} (i) it suffices to prove the equalities $a_1+a_2=M/(q-1)$ and $b(\eta_j,q)=jm\ 0^{\infty}$.
The first two equalities are equivalent to $a_1+a_2=M/(q-1)$, i.e., to
\begin{equation*}
\frac{m}{q}+\frac{m}{q^2}+\frac{m}{q}=\frac{2m-1}{q-1}.
\end{equation*}
It is easily seen to be equivalent to the quadratic equation $q^2-mq-m=0$ defining $q$.

The expansions $jm\ 0^{\infty}$ are greedy because $m/q<1$, and the equalities $(jm\ 0^{\infty})_q=\eta_j$ follow by a direct computation.
Indeed, we have
\begin{equation*}
\eta_j=\frac{j-1}{q}+\frac{M}{q^2-q}
=\frac{j}{q}+\frac{2m-q}{q^2-q}
=\frac{j}{q}+\frac{m}{q^2}
\end{equation*}
because the last equality is again equivalent to the quadratic equation $q^2-mq-m=0$.
\end{example}

Our second example extends to all other elements of $\vv\setminus\uuu$:

\begin{lemma}\label{l37}
Let $q\in\vv\setminus\uuu$ and $q>\min\vv$.
\begin{enumerate}[\upshape (i)]
\item $N$ is even, say $N=2n$, $\beta(q)$ and $\alpha(q)$ have the form
\begin{equation*}
\beta(q)=\alpha_1\cdots\alpha_n^+\overline{\alpha_1\cdots\alpha_n}\ 0^{\infty}
\qtq{and}
\alpha(q)=\left(\alpha_1\cdots\alpha_n^+\overline{\alpha_1\cdots\alpha_n^+}\right)^{\infty},
\end{equation*}
and
\begin{equation}\label{31}
b_{j}=a_{n+j}\qtq{and}
b_{n+j}=a_{j}\qtq{for}j=1,\ldots,n.
\end{equation}
\item The greedy expansions of the numbers $\eta_j$ are the following:
\begin{align*}
&b(\eta_j,q)=j\ \overline{\alpha_1\cdots\alpha_n}0^\infty\qtq{for}j=1,\ldots,M,\\
&b(\eta_{M+1},q)=M^{\infty}.
\end{align*}
\item
Writing $\beta(q)=\alpha_1\cdots\alpha_n^+\overline{\alpha_1\cdots\alpha_n}\ 0^{\infty}$ we have
\begin{equation*}
S_q\cap\set{a_i,b_i\ :\ 1\le i\le N}=\set{a_n,a_{2n}}=\set{\eta_{\alpha_n^+},\theta_{\overline{\alpha_n}}}.
\end{equation*}
\item The numbers
\begin{equation*}
\theta_0,\ldots, \theta_M,\quad
\eta_1,\ldots, \eta_{M+1}
\qtq{and}
a_1,\ldots, a_{n-1}, a_{n+1},\ldots, a_{2n-1}
\end{equation*}
are  distinct.
Furthermore,
\begin{equation*}
a_n=\eta_{\alpha_n^+}
\qtq{and}
a_{2n}=\theta_{\overline{\alpha_n}}.
\end{equation*}
\end{enumerate}
\end{lemma}

\begin{proof}
(i) The first part is shown in \cite[Lemma 3.5]{DKL2016}.
For any fixed $1\le j\le n$, it follows that the quasi-greedy expansions of $a_{n+j}$ and $a_j$ are
\begin{equation*}
\overline{\alpha_j\cdots\alpha_n^+}\left(\alpha_1\cdots\alpha_n^+\overline{\alpha_1\cdots\alpha_n^+}\right)^{\infty}
\qtq{and}
\alpha_j\cdots\alpha_n^+
\left(\overline{\alpha_1\cdots\alpha_n^+}\alpha_1\cdots\alpha_n^+\right)^{\infty},
\end{equation*}
respectively.
Since
\begin{equation*}
\overline{\alpha_j\cdots\alpha_n^+}\left(\alpha_1\cdots\alpha_n^+\overline{\alpha_1\cdots\alpha_n^+}\right)^{\infty}
+\alpha_j\cdots\alpha_n^+
\left(\overline{\alpha_1\cdots\alpha_n^+}\alpha_1\cdots\alpha_n^+\right)^{\infty}
=M^{\infty},
\end{equation*}
we conclude that $a_{n+j}+a_j=\frac{M}{q-1}$.

This is equivalent to \eqref{31} by the definition of the numbers $b_i$.
\medskip

(ii) The relation  $b(\eta_{M+1},q)=M^{\infty}$ is clear.
For $j=1,\ldots,M$ it follows from (i) and Lemma \ref{l33} (iii) that $j\overline{\alpha_1\cdots\alpha_n}$ is an expansion of $\eta_j$.
It remains to show that no expansion of $\eta_j$ may begin with a larger digit.
This is true because $b(b_1,q)=\overline{\alpha_1\cdots\alpha_n}$, and therefore, using also Lemma \ref{l34} (iv),
\begin{equation*}
\eta_j=\frac{j+b_1}{q}<\frac{j+a_1}{q}=\frac{j+1}{q}.
\end{equation*}
\medskip

(iii) We have
\begin{equation*}
a_{2n}=\theta_{\overline{\alpha_n}}
\qtq{and}
a_n=b_{2n}=\frac{M}{q-1}-a_{2n}
=\frac{M}{q-1}-\theta_{\overline{\alpha_n}}
=\eta_{M+1-\overline{\alpha_n}}
=\eta_{\alpha_n^+},
\end{equation*}
so that both $a_{2n}$ and $a_n$ belong to the boundary of $S_q$.

It remains to show that if $1\le j\le 2n$ and $a_j\in S_q$, then $j\in\set{n,2n}$.
Fix $1\le i\le M$ such that $a_j\in[\theta_i,\eta_i]$, i.e.,
\begin{equation*}
\frac{i}{q}\le a_j\le \frac{i-1}{q}+\frac{M}{q^2-q},
\end{equation*}
then we infer from (ii) and Lemma \ref{l34} (i) the relations
\begin{equation}\label{32}
i0^\infty\le  \alpha_j\cdots\alpha_n^+\overline{\alpha_1\cdots\alpha_n}0^\infty\le i\overline{\alpha_1\cdots\alpha_n}0^\infty \qtq{if}1\le j\le n
\end{equation}
and
\begin{equation}\label{33}
i0^\infty\le  \overline{\alpha_\ell\cdots\alpha_n}0^\infty\le i\overline{\alpha_1\cdots\alpha_n}0^\infty \qtq{if}n+1\le j=n+\ell\le 2n.\end{equation}
The second inequality in \eqref{32} is an equality if $j=n$ and $i=\alpha_n^+$.
Otherwise \eqref{32} yields
\begin{equation*}
i=a_j \qtq{and} \alpha_{j+1}\cdots\alpha_n^+\le \overline{\alpha_1\cdots\alpha_{n-j}}.
\end{equation*}
Hence
\begin{equation*}
\alpha_{j+1}\cdots\alpha_n< \overline{\alpha_1\cdots\alpha_{n-j}}
\qtq{or}
\overline{\alpha_{j+1}\cdots\alpha_n}
>\alpha_1\cdots\alpha_{n-j},
\end{equation*}
contradicting  Proposition \ref{p25} (v).

Next, the first inequality in \eqref{33} is an equality if $j=2n$ and $i=\overline{\alpha_n}$.
Otherwise \eqref{33} yields $i=\overline{a_\ell}$ and
\begin{equation*}
\overline{\alpha_{\ell+1}\cdots\alpha_n}0^{\infty}\le \overline{\alpha_1\cdots\alpha_n}0^{\infty}.
\end{equation*}
This is equivalent to the following inequality between the corresponding quasi-greedy expansions:
\begin{equation*}
\overline{\alpha_{\ell+1}\cdots\alpha_n^+}\left(\alpha_1\cdots\alpha_n^+\overline{\alpha_1\cdots\alpha_n^+}\right)^{\infty}
\le\overline{\alpha_1\cdots\alpha_n^+}\left(\alpha_1\cdots\alpha_n^+\overline{\alpha_1\cdots\alpha_n^+}\right)^{\infty}
\end{equation*}
Taking reflections and writing $\alpha(q)=\alpha_1\alpha_2\cdots,$ this is equivalent to
\begin{equation*}
\alpha_1\alpha_2\cdots
\le \alpha_{\ell+1}\alpha_{\ell+2}\cdots.
\end{equation*}
The strict inequality here contradicts $q\in\vv$.
In case of equality, Proposition \ref{p23}  (ii) implies that  $\alpha(q)$ has  a period of length $\ell<2n$, and this is impossible by the minimal choice of $N=2n$.
\medskip

(iv) The first part follows from (iii) and Lemma \ref{33} (iii).
The equalities follow from (ii) and Lemma \ref{l34} (i).
\end{proof}

The set
\begin{equation*}
(0,\frac{M}{q-1})\setminus\left(
\set{a_i, b_i \ :\ i=1,\ldots,N-1}\cup\set{\theta_j, \eta_j\ :\ j=1,\ldots,M}\right)
\end{equation*}
is the union of a finite number of disjoint open intervals.
If $(c,d)$ is one of them, then we will also  write it in the form $(c,c^+)$ or $(d^-,d)$.
For example, since none of the points $a_i, b_i$ lie inside $S_q$, we have $\theta_j^+=\eta_j$ for all $j$.
Thus we may write the switch region as
\begin{equation*}
S_q=\cup_{j=1}^M[\theta_j,\theta_j^+].
\end{equation*}

Let us write $I\xrightarrow{k}
J$ if $I,J$ are some of these intervals, $k\in\set{0,1,\ldots, M}$ and $T_k(I)\supset J$.

\begin{lemma}\label{l38}
Let $(c_i)\in\set{0,\ldots,M}^{\infty}$ be an arbitrary sequence, and  $x_n:=(c_nc_{n+1}\cdots)_q$, $n=1,2,\ldots.$
\begin{enumerate}[\upshape (i)]
\item There exists a sequence of intervals $(I_i)$ such that
\begin{equation}\label{34}
I_1\xrightarrow{c_1}
I_2\xrightarrow{c_2}
I_3\xrightarrow{c_3}\cdots.
\end{equation}
\item If there exists a sequence of intervals $(I_i)$ satisfying \eqref{34}, then $x_n\in\overline{I_n}$ for all $n$.
\end{enumerate}
\end{lemma}

\begin{proof}
(i) In the following proof by an interval we mean one of the above defined special open intervals, and by closed intervals we mean their closures.

Choose an interval $I_1$ such that $x_1\in\overline{I_1}$.
Assume by induction that we have already defined $I_1,\ldots, I_n$ for some $n\ge 1$, so that
$x_i\in\overline{I_i}$ for $i=1,\ldots, n$ and
\begin{equation*}
I_1\xrightarrow{c_1}
I_2\xrightarrow{c_2}
\cdots
\xrightarrow{c_{n-1}}
I_n.
\end{equation*}
Then, since
\begin{equation*}
x_{n+1}=T_{c_n}(x_n)\cap\left[0,\frac{M}{q-1}\right]\in
T_{c_n}(\overline{I_n})\cap\left[0,\frac{M}{q-1}\right],
\end{equation*}
and since the right-hand set is either a closed interval or a finite union of (consecutive) closed intervals, we can chose an interval $I_{n+1}$ satisfying $x_{n+1}\in\overline{I_{n+1}}$ and $T_{c_n}(\overline{I_n})\supset I_{n+1}$.
The latter relation is equivalent to $I_n\xrightarrow{c_n}I_{n+1}$.
\medskip

(ii) For each fixed $n\ge 1$, the relations \eqref{34} show that $I_n$ contains a sequence $y_n, y_{n+1}, \ldots$ such that for each $k\ge n$, $y_k$ has expansion starting with $c_n\cdots c_k$.
By compactness there is a subsequence converging to some point in $\overline{I_n}$ of which  $c_nc_{n+1}\cdots$ is an expansion, i.e., to $x_n$.
Hence $x_n\in\overline{I_n}$.
\end{proof}

If $(c_i)$ is a unique expansion, then $x_i\notin S_q$ by Lemma \ref{l31} for all $i$, so that all  intervals $I_i$ in Lemma \ref{l38} (i) are different from the switch intervals $(\theta_j,\eta_j)$, $j=1,\ldots,M$.
This leads to the following graph construction \cite{TK2010, DJ2017}.
The set
\begin{equation*}
(0,\frac{M}{q-1})\setminus\left(S_q\cup\set{a_i, b_i \ :\ i=1,\ldots,N-1}\right)
\end{equation*}
is the union of a finite number of disjoint open intervals.
We take as the vertices of a \emph{labeled  directed graph} $\gg=\gg(q)$ these open intervals; by Lemmas \ref{l34} (iii) and \ref{l35} there are $2N+M-1$ vertices if $q\in\uuu\setminus\uu$ and $N+M-1=2n+M-1$ vertices if $q\in\vv\setminus\uuu$.
Furthermore, the edges are the triplets $(I,k,J)$ where  $I,J$ are vertices of $\gg(q)$,  $k\in\set{0,1, \ldots, M}$ and  $I\xrightarrow{k}J$.

The following lemma shows in particular that if $I\xrightarrow{k}J$ is an edge of $\gg(q)$, then $k$ is uniquely determined by the pair $(I,J)$.

\begin{lemma}\label{l39}
Let $I\xrightarrow{k}J$.
Then
\begin{align*}
&I\subset (\theta_0,\theta_1)\Longrightarrow k=0,\\
&I\subset (\eta_{j-1},\theta_j)\Longrightarrow k=j-1,\quad j=2,\ldots, M,\\
&I\subset (\eta_M,\eta_{M+1})\Longrightarrow k=M.
\end{align*}
\end{lemma}

\begin{proof}
This follows from Remark \ref{r32} (ii).
\end{proof}

\begin{example}\label{e310}
Let us consider the case of the generalized Golden Ratio $q=\min\vv$, already studied in Remark \ref{r32} (i) and Example \ref{e36}.

If $M$ is even, then $S_q=[\theta_1,\eta_M]$, so that $\gg(q)$ has two vertices: $(\theta_0,\theta_1)$ and $(\eta_M,\eta_{M+1})$, and the only edges are
\begin{equation*}
(\theta_0,\theta_1)\xrightarrow{0}(\theta_0,\theta_1)
\qtq{and}
(\eta_M,\eta_{M+1})\xrightarrow{M}(\eta_M,\eta_{M+1}).
\end{equation*}

If $M$ is odd, then $\gg(q)$ has $M+1$ vertices:
\begin{equation*}
(\theta_0,\theta_1),
(\eta_1,\theta_2),
\ldots,
(\eta_{M-1},\theta_M),
(\eta_M,\eta_{M+1}).
\end{equation*}
A direct computations shows that
\begin{align*}
&T_0((\theta_0,\theta_1))=(\theta_0,\eta_m),\\
&T_j((\eta_j,\theta_{j+1}))=(\theta_m,\eta_m),\quad j=1,\ldots,M-1,\\
&T_M((\eta_M,\eta_{M+1}))=(\theta_m,\eta_{M+1}).
\end{align*}
It follows that $\gg(q)$ has the following edges:
\begin{align*}
&(\theta_0,\theta_1)\xrightarrow{0}(\theta_0,\theta_1),
\qtq{and}
(\theta_0,\theta_1)\xrightarrow{0}(\eta_j,\theta_{j+1}),\quad j=1,\ldots,m-1,\\
&(\eta_M,\eta_{M+1})\xrightarrow{M}(\eta_M,\eta_{M+1}),
\qtq{and}
(\eta_M,\eta_{M+1})\xrightarrow{M}
(\eta_j,\theta_{j+1}),\quad j=m,\ldots,M-1.
\end{align*}
\end{example}

As before, an interval-vertex $(c,d)$ will also be written in the form $(c,c^+)$ or $(d^-,d)$.
Since we have removed the connected components of the interior of $S_q$, there may be five (not mutually excluding) types of vertices, of the form
\begin{equation*}
(a_i^-, a_i),\quad
(b_j, b_j^+),\quad
(a_i, b_j),\quad
(\theta_i^-,\theta_i)
\qtq{and}
(\eta_j,\eta_j^+).
\end{equation*}
Here we may have $a_i^-=a_j$, $b_j$ or $\eta_j$; $\theta_i^-=a_j$, $b_j$ or $\eta_j$, and similarly $b_j^+=a_i$, $b_i$ or $\theta_i$; $\eta_j^+=a_i$, $b_i$ or $\theta_i$.

We say that a sequence $(c_i)$ is \emph{generated by an infinite path} $(I_i)$ in $\gg(q)$ if $(I_i)$ is a sequence of vertices in $\gg(q)$ such that
\begin{equation*}
I_1\xrightarrow{c_1}
I_2\xrightarrow{c_2}
\cdots.
\end{equation*}
We denote by $\gg_q'$ the set of sequences generated by $\gg(q)$, and we set
\begin{equation*}
\gg_q:=\set{(c_i)_q\ :\ (c_i)\in\gg_q'}.
\end{equation*}

It follows from the reflection properties $b_i=\frac{M}{q-1}-a_i$ and $\eta_j=\frac{M}{q-1}-\theta_{M+1-j}$ that $\gg_q'$ and $\gg_q$ are invariant for reflections:
\begin{equation}\label{35}
(c_i)\in \gg_q'\Longrightarrow\overline{(c_i)}\in \gg_q'
\qtq{and}
x\in\gg_q\Longrightarrow\frac{M}{q-1}-x\in\gg_q.
\end{equation}

The following lemma will be needed in the next section.

\begin{lemma}\label{l311}
Let $q\in \vv\setminus\uu$ and $q>\min\vv$.
Then $T_{\alpha_{N-1}}(a_{N-1}^+)>\eta_{\alpha_N^+}$.
\end{lemma}

\begin{remark}\label{r312}
The lemma fails for the generalized Golden Ratios.
Indeed, let $q=\min\vv$.
If $M=2m$, then $N=1$, so that the left hand side is undefined.
If $M=2m-1$, then $N=2$ and $\beta(q)=mm\ 0^{\infty}$, and a short computation shows that we have an equality:
\begin{equation*}
T_{\alpha_{N-1}}(a_{N-1}^+)=T_{m}(\theta_{m+1})=1=a_1=\eta_{\alpha_N^+}.
\end{equation*}
\end{remark}

\begin{proof}
Since $T_{\alpha_{N-1}}$ is an increasing map, by Lemma \ref{l34} (v) we have
\begin{equation*}
T_{\alpha_{N-1}}(a_{N-1}^+)>T_{\alpha_{N-1}}(a_{N-1})=a_N=\frac{\alpha_N^+}{q}.
\end{equation*}
Using Lemmas \ref{l35} (iv) and \ref{l37} (iii) hence we deduce that $T_{\alpha_{N-1}}(a_{N-1}^+)\ge \frac{\alpha_N}{q}+\frac{M}{q^2-q}$, and it remains to exclude the equality.

Assume on the contrary that $T_{\alpha_{N-1}}(a_{N-1}^+)=\frac{\alpha_N}{q}+\frac{M}{q^2-q}$, or equivalently that
\begin{equation*}
a_{N-1}^+=\frac{\alpha_{N-1}}{q}+\frac{\frac{\alpha_N}{q}+\frac{M}{q^2-q}}{q}
=(\alpha_{N-1}\alpha_N\ M^{\infty})_q
=\left(\alpha_{N-1}\alpha_N^+\overline{(\alpha_1\cdots\alpha_N)^{\infty}}\right)_q.
\end{equation*}
Since $(\alpha_N^+\overline{(\alpha_1\cdots\alpha_N)^{\infty}})_q=(\alpha_N M^{\infty})_q$, $\alpha_N^+\overline{(\alpha_1\cdots\alpha_N)^{\infty}}$ is a greedy expansion.
Hence $\alpha_{N-1}\alpha_N^+\overline{(\alpha_1\cdots\alpha_N)^{\infty}}$ is clearly also a greedy expansion if $\alpha_{N-1}=M$.

This is also true if $\alpha_{N-1}<M$ because we have $a_{N-1}^+<\frac{\alpha_{N-1}^+}{q}$. It is true, in fact
\begin{equation*}
\begin{split}
a_{N-1}^+<\frac{\alpha_{N-1}^+}{q}&\Longleftrightarrow a_{N-1}^+=\frac{\alpha_{N-1}}{q}+\frac{\frac{\alpha_N}{q}+\frac{M}{q^2-q}}{q}<\frac{\alpha_{N-1}}{q}+\frac{1}{q}\\
&\Longleftrightarrow \frac{\alpha_N}{q}+\frac{M}{q^2-q}<1\\
&\Longleftrightarrow\alpha_N^+\overline{(\alpha_1\cdots\alpha_N)^{\infty}}<(\alpha_1\cdots\alpha_N)^\infty\\
&\Longleftrightarrow\alpha_N^+<a_1 \qtq{or} \alpha_N^+=a_1,\
\overline{(\alpha_1\cdots\alpha_N)^\infty}<\alpha_2\cdots\alpha_N(\alpha_1\cdots\alpha_N)^\infty\\
\end{split}
\end{equation*}
We prove the final inequality as follows:
Since $q\in \vv$, we have $\overline{\alpha_1\cdots\alpha_N}\le\alpha_2\cdots\alpha_N\alpha_1\ $. We claim the inequality is strict, otherwise
$\overline{\alpha_1\cdots\alpha_N}=\alpha_2\cdots\alpha_N\alpha_1\ $ yields $\alpha_1=\overline{\alpha_1}$ if $N$ is odd, it is impossible since we have $\alpha_1>\overline{\alpha_1}$.  $\alpha_1\cdots\alpha_N=(\alpha_1\overline{\alpha_1})^n$ if $N=2n$ is even, it is impossible, since $\alpha_1\cdots\alpha_N$ as is the shortest period of $(\alpha_i)=(\alpha_1\cdots\alpha_N)^{\infty}$.

Hence the greedy expansion of $a_{N-1}^+$ has to start with the digit $\alpha_{N-1}$, and then to follow with the greedy expansion $\alpha_N^+\overline{(\alpha_1\cdots\alpha_N)^{\infty}}$,

so that it has the infinite greedy expansion
\begin{equation*}
b(a_{N-1}^+,q)=\alpha_N^+\overline{(\alpha_1\cdots\alpha_N)^{\infty}}.
\end{equation*}
We will arrive at a contradiction by showing that this expansion is different from the greedy expansions of the numbers $a_i$, $b_i$, $\theta_j$ and $\eta_j$.

This is clear for the number $\eta_{M+1}$ because
\begin{equation*}
\alpha_{N-1}\alpha_N^+\overline{(\alpha_1\cdots\alpha_N)^{\infty}}\ne M^{\infty},
\end{equation*}
and it is also clear for the numbers $a_i$ and  $\theta_j$ because the latter have finite greedy expansions by Lemma \ref{l34} (i).

If $q\in\vv\setminus\uuu$, then it is also true for the numbers $b_i$ and $\eta_1,\ldots, \eta_M$ because they also have finite greedy expansions by Lemma \ref{l37} (i) and (ii).

Henceforth we assume that $q\in\uuu\setminus\uu$.
It remains to exclude the equalities
\begin{equation}\label{36}
b(\eta_j,q)=\alpha_{N-1}\alpha_N^+\overline{(\alpha_1\cdots\alpha_N)^{\infty}}
\end{equation}
for $j=1,\ldots,M,$ and the equalities
\begin{equation}\label{37}
b(b_i,q)=\alpha_{N-1}\alpha_N^+\overline{(\alpha_1\cdots\alpha_N)^{\infty}}
\end{equation}
for $i=1,\ldots,N-1$.

Assume on the contrary that \eqref{36} holds for some $j\in\set{1,\ldots,M}$.
Then by Lemma \ref{l35} (ii) we obtain
\begin{equation*}
j\ \overline{(\alpha_1\cdots\alpha_N)^{\infty}}=\alpha_{N-1}\alpha_N^+\overline{(\alpha_1\cdots\alpha_N)^{\infty}};
\end{equation*}
hence $N=1$ and $\alpha_1^+=\overline{\alpha_1}$, contradicting Proposition \ref{p25} (iii).

Finally, assume on the contrary that \eqref{37} holds for some $i\in\set{1,\ldots,N-1}$.
Then by Lemma \ref{l35} (ii) we obtain that
\begin{equation*}
\overline{\alpha_i\cdots\alpha_N(\alpha_1\cdots\alpha_N)^{\infty}}=\alpha_{N-1}\alpha_N^+\overline{(\alpha_1\cdots\alpha_N)^{\infty}}.
\end{equation*}
For $i<N-1$ this implies
\begin{equation*}
(\alpha_1\cdots\alpha_N)^{\infty}
=(\alpha_{i+2}\cdots\alpha_N\alpha_1\cdots\alpha_{i+1})^{\infty},
\end{equation*}
contradicting the choice of $\alpha_1\cdots\alpha_N$ as the shortest period of $(\alpha_i)=(\alpha_1\cdots\alpha_N)^{\infty}$.

Finally, if $i=N-1$, then the supposed equality implies that $\overline{\alpha_{N-1}\alpha_N}=\alpha_{N-1}\alpha_N^+$.
But this is also impossible because
$\alpha_{N-1}\ne\overline{\alpha_{N-1}}$ if $M$ is odd, and  $\alpha_N^+\ne\overline{\alpha_N}$ if $M$ is even.
\end{proof}

We end this section by establishing some general properties of the graphs $\gg(q)$.

\begin{proposition}\label{p313}
Let $q\in\vv\setminus\uu$.
\begin{enumerate}[\upshape (i)]
\item If $b_1<\theta_1$, then $(b_1^-,b_1)=(0,\frac{M}{q-1}-1)$ is a vertex of $\gg(q)$, and the only edge arriving at this vertex is
\begin{equation*}
(0, b_1)\xrightarrow{0}(0,b_1).
\end{equation*}
If $\eta_{j-1}<b_1<\theta_j$ for some $j\in\set{2,\ldots,M}$, then $b_1^-=\eta_{j-1}$,
\begin{equation*}
(0,\theta_1),
(\eta_1,\theta_2),\ldots,
(\eta_{j-2},\theta_{j-1}),
(\eta_{j-1},b_1)
\end{equation*}
are the vertices of $\gg(q)$ lying in $(0,b_1)$, and the only edge arriving at any $J$ of these vertices is
\begin{equation*}
(0,\theta_1)\xrightarrow{0}J.
\end{equation*}

If $a_1>\eta_M$, then $(a_1, a_1^+)=(1,\eta_{M+1})=(1,M/(q-1))$ is a vertex of $\gg(q)$, and the only edge arriving at this vertex is
\begin{equation*}
(1,\eta_{M+1})\xrightarrow{M}(1,\eta_{M+1}).
\end{equation*}
If $\eta_{j-1}<a_1<\theta_j$ for some $j\le M$, then $a_1^+=\theta_j$,
\begin{equation*}
(a_1,\theta_j),
(\eta_j,\theta_{j+1}),\ldots,
(\eta_{M-1},\theta_M),
(\eta_M,\eta_{M+1})
\end{equation*}
are the vertices of $\gg(q)$ lying in $(a_1,\eta_{M+1})$, and the only edge arriving at any $J$ of these vertices is
\begin{equation*}
(\eta_M,\eta_{M+1})\xrightarrow{M}J.
\end{equation*}

\item  If there exists a path
\begin{equation*}
I_1\xrightarrow{\alpha_1}
I_2\xrightarrow{\alpha_2}
\cdots
\xrightarrow{\alpha_{N-1}}
I_N\xrightarrow{\alpha_N^+}
I_{N+1}
\end{equation*}
in $\gg(q)$, then $I_1\subset (a_1,\eta_{M+1})$.
%\item The only infinite paths
%\begin{equation*}
%I_1\xrightarrow{c_1}
%I_2\xrightarrow{c_2}
%\cdots
%\end{equation*}
%in $\gg(q)$ \emph{ending with} $0^{\infty}$ or $M^{\infty}$ are $0^{\infty}$ and $M^{\infty}$.
\end{enumerate}
\end{proposition}
\begin{proof}
(i) By symmetry it suffices to consider $b_1$.
All assertions follow from the preceding lemmas and from the relation
\begin{equation*}
T_0((0,\theta_1))=(0,1)=(0,a_1)\supset(0,b_1),
\end{equation*}
except the non-existence of other edges than those indicated.
This last property follows from Lemma \ref{l39} and from the relations
\begin{align*}
&T_0((\theta_1,\infty))=(1,\infty)=(a_1,\infty)\subset(b_1,\infty)
\intertext{and}
&T_{i-1}((\eta_{i-1},\infty))
=(b_1,\infty)\qtq{for}i=2,\ldots, M.
\end{align*}
\medskip

(ii) The composition map
\begin{equation*}
T:=T_{\alpha_N^+}\circ T_{\alpha_{N-1}}\circ\cdots\circ T_{\alpha_1}
\end{equation*}
is increasing, and $T(a_1)=0$ by Lemma \ref{l34} (v), so that $T(x)<0$ for all $x<a_1$.
If the vertex $I_1$ of the graph is not a subset of $(a_1,\eta_{M+1})$, then
$I_1\subset(-\infty, a_1)$.
In this case we have $T(I_1)\subset(-\infty,0)$, and hence $T(I_1)$ cannot contain any interval $I_{N+1}$ of the graph.
%\medskip
%
%(iii) Choose a large integer $k\ge 1$ such that $q^kb_1>\frac{M}{q-1}$.
%By symmetry it suffices to show that if a
% sequence $(c_i)\in\gg_q'$ contains the word $c_{n+1}\cdots c_{n+k}=0^k$ for some $n\ge 1$, then $c_n=0$.
%Setting
%\begin{equation*}
%T:=T_{c_{n+k}}\circ T_{c_{n+k-1}}\circ\cdots\circ T_{c_{n+1}}
%\end{equation*}
%we  have
%\begin{equation*}
%I_{n+k+1}\subset T(I_{n+1})=q^kI_{n+1}.
%\end{equation*}
%Hence $I_{n+1}\subset (0,b_1)$, for otherwise $\inf q^kI_{n+1}>\frac{M}{q-1}$, and thus $q^kI_{n+1}$ cannot contain any vertex $I_{n+k+1}$ of the graph.
%Since $I_{n+1}\subset (0,b_1)$, applying (ii) we conclude that $c_n=0$.
\end{proof}

\section{Proof of Theorem \ref{t11} for $q\in\uuu\setminus\uu$}\label{s4}

In this section we assume that $q\in\uuu\setminus\uu$, so that the graph $\gg(q)$ has $2N+M-1$ vertices.

\begin{lemma}\label{l41}\mbox{}
Let $q\in\uuu\setminus\uu$, and write
\begin{equation*}
\beta(q)=\alpha_1\cdots\alpha_N^+\ 0^{\infty}
\qtq{and}
\alpha(q)=(\alpha_1\cdots\alpha_N)^{\infty}
\end{equation*}
as before.
Then the graph $\gg(q)$ contains the following two cycles:
\begin{align*}
&
%\text{Cycle I}:\quad
(a_1^-, a_1)
\xrightarrow{\alpha_1}(a_2^-, a_2)\xrightarrow{\alpha_2}\cdots
\xrightarrow{\alpha_{N-1}}(a_N^-, a_N)
\xrightarrow{\alpha_N}(a_1^-, a_1),\\
&
%\text{Cycle II}:\quad
(b_1, b_1^+)
\xrightarrow{\overline{\alpha_1}}(b_2, b_2^+)
\xrightarrow{\overline{\alpha_2}}\cdots
\xrightarrow{\overline{\alpha_{N-1}}}(b_N, b_N^+)
\xrightarrow{\overline{\alpha_N}}(b_1, b_1^+),
\intertext{and the following two paths:}
&
%\text{Path I}:\quad
(a_1, a_1^+)
\xrightarrow{\alpha_1}(a_2, a_2^+)\xrightarrow{\alpha_2}\cdots
\xrightarrow{\alpha_{N-2}}(a_{N-1}, a_{N-1}^+)
\xrightarrow{\alpha_{N-1}}
(\eta_{\alpha_N^+},\eta_{\alpha_N^+}^+)
\xrightarrow{\alpha_N^+}
(b_1,b_1^+),\\
&
%\text{Path II}:\quad
(b_1^-, b_1)
\xrightarrow{\overline{\alpha_1}}(b_2^-, b_2)
\xrightarrow{\overline{\alpha_2}}\cdots
\xrightarrow{\overline{\alpha_{N-2}}}(b_{N-1}^-, b_{N-1})
\xrightarrow{\overline{\alpha_{N-1}}}(\theta_{\overline{\alpha_N}}^-,\theta_{\overline{\alpha_N}})
\xrightarrow{\overline{\alpha_N^+}}
(a_1^-,a_1).
\end{align*}
\end{lemma}

\begin{proof}
The existence of the indicated edges follows from Lemma \ref{l34} (v), except for the last edges
\begin{equation*}
(a_{N-1}, a_{N-1}^+)
\xrightarrow{\alpha_{N-1}}
(\eta_{\alpha_N^+},\eta_{\alpha_N^+}^+)
\xrightarrow{\alpha_N^+}
(b_1,b_1^+)
\end{equation*}
and
\begin{equation*}
(b_{N-1}^-, b_{N-1})
\xrightarrow{\overline{\alpha_{N-1}}}(\theta_{\overline{\alpha_N}}^-,\theta_{\overline{\alpha_N}})
\xrightarrow{\overline{\alpha_N^+}}
(a_1^-,a_1)
\end{equation*}
of the two paths.
The existence of the first edge follows from the relations $T_{\alpha_{N-1}}(a_{N-1})=a_N=\theta_{\alpha_N^+}$ and $T_{\alpha_{N-1}}(a_{N-1}^+)>
\eta_{\alpha_N^+}$ (see Lemma \ref{l311}), because $(\theta_{\alpha_N^+},\eta_{\alpha_N^+})$ is a switch interval.
The existence of the fourth edge follows from the equality
\begin{equation*}
T_{\overline{\alpha_N^+}}(\theta_{\overline{\alpha_N}})
=\overline{\alpha_N}-\overline{\alpha_N^+}
=1=a_1.
\end{equation*}
The existence of the remaining two edges follows by reflection.
\end{proof}

Now we are ready to determine the numbers and sequences generated by the graph $\gg(q)$:

\begin{proof}[Proof of Theorem \ref{t11} for $q\in\uuu\setminus\uu$]
In view of Proposition  \ref{p210} we prove the equivalent relations
\begin{equation*}
\gg_q=\overline{\uu_q}
\qtq{and}
\gg_q'=\set{a(x,q)\ :\ x\in \overline{\uu_q}}.
\end{equation*}

First we show that if $(c_i)\in\gg_q'$,
then $(c_i)_q\in\overline{\uu_q}$ and
$(c_i)=a(x_1,q)$.
Setting $x_i:=(c_ic_{i+1}\cdots)_q$ for $i=1,2,\ldots$ as in Lemma \ref{l38}, none of the points $x_i$ belongs to the interior of $S_q$ by the definition of the graph.
If $x_i\notin S_q$ for all $i\ge 1$, i.e., if we always avoid the switch region, then $x_1\in\uu_q$ by Lemma \ref{l31}.
Furthermore,  $(c_i)$ being the only expansion of $x_1$, $(c_i)=a(x_1,q)$.

Otherwise there exists a first $k$ such that $x_k\in S_q$, i.e., $x_k\in\set{\theta_j, \eta_j\ :\ j=1,\ldots,N}$; we will prove that $x_1\in \overline{\uu_q}$ and $(c_i)=a(x_1,q)$.

Assume first that $x_k=\theta_j$ for some $j$.
Then $\theta_j\in \overline{\uu_q}$ by Lemma \ref{l35} (i).

Since $\theta_j\in\overline{\uu_q}$ and $\theta_j$ has a finite greedy expansion by definition, it has a unique infinite expansion by Proposition \ref{p210} (vi).
Since all expansions of $x_1$ start with $c_1\cdots c_{k-1}$ by the minimality of $k$, all expansions of $x_1$ are obtained by adding to the expansions of $\theta_j$ the same prefix  $c_1\cdots c_{k-1}$.
Hence $x_1$ also has a unique infinite expansion, and therefore it belongs to $\overline{\uu_q}$  by Proposition \ref{p210} (vi).
Since both $a(x_1,q)$ and $(c_i)$ are infinite expansions of $x_1$ (they are even doubly infinite by Proposition \ref{p210} (i) and Lemma \ref{l41}), we conclude that $a(x_1,q)=(c_i)$.

Now assume that $x_k=\eta_j$ for some $j$.
Since $\overline{(c_i)}\in\gg_q'$ by \eqref{35}, repeating the preceding arguments we obtain that $\overline{x_1}\in\overline{\uu_q}$ and $a(\overline{x_1},q)=\overline{(c_i)}$.
Since $\overline{\uu_q}$ is invariant for reflections, hence we infer that $x_1\in \overline{\uu_q}$, and that $(c_i)$ is a doubly infinite expansion of $x_1$.
Since $\overline{\uu_q}=\vv_q$, this is the only doubly infinite expansion of $x_1$.
Since $a(x_1,q)$ is also doubly infinite, we conclude again that $(c_i)=a(x_1,q)$.
\medskip

It remains to prove that $a(x_1,q)\in\gg_q'$ for every $x_1\in\overline{\uu_q}$.
If $x_1\in\uu_q$, then this follows from Lemmas \ref{l31} and \ref{l38} (i) because $x_i\notin S_q$ and therefore $I_i\cap S_q=\varnothing$ for every $i$.

Finally, let $x_1\in\overline{\uu_q}\setminus\uu_q$.
If $x_1$ has a finite greedy expansion $c_1\cdots c_k^+0^{\infty}$, then  all other expansions of $x_1$ start with $c_1\cdots c_k$ by \cite[Proposition 4.2]{K2012}.
Setting $x_n:=(c_n\cdots c_k^+0^{\infty})_q$ for $n=2,3\ldots, k$, we have
\begin{equation*}
x_1\xmapsto{c_1}
x_2\xmapsto{c_2}
\cdots
\xmapsto{c_{k-2}}
x_{k-1}\xmapsto{c_{k-1}}
x_k=\frac{c_k^+}{q}=\theta_{c_k^+}.
\end{equation*}
We infer from $x_1\in\overline{\uu_q}\setminus\uu_q$ by using the lexicographic characterizations that $x_2,\ldots, x_k\in\overline{\uu_q}\setminus\uu_q$.
Approximating them with elements of $\uu_q$ hence we obtain that $x_1,\ldots, x_k$ are not interior points of $S_q$.
The points $x_1,\ldots, x_{k-1}$ do not belong to the boundary of $S_q$ either because their greedy expansions are $c_n\cdots c_k^+0^{\infty}$ for
$n=1,\ldots, k-1$,  and they are different from the greedy expansions of $\theta_j$ and $\eta_j$ for $j=1,\ldots, M$.
We have thus $\set{x_1,\ldots, x_{k-1}}\cap S_q=\varnothing$.
This implies that there exist interval-vertices $I_1,\ldots, I_{k-1}$ of $\gg(q)$ such that $x_n\in\overline{I_n}$ for $n=1,\ldots, k-1$, and $\gg(q)$ contains the path
\begin{equation*}
I_1
\xrightarrow{c_1}
I_2
\xrightarrow{c_2}
\cdots
\xrightarrow{c_{k-2}}
I_{k-1}
\xrightarrow{c_{k-1}}
(\theta_{c_k^+}^-,\theta_{c_k^+})
\xrightarrow{c_k}
(a_1^-, a_1).
\end{equation*}
Since $\gg(q)$ also contains the cycle
\begin{equation*}
(a_1^-, a_1)
\xrightarrow{\alpha_1}(a_2^-, a_2)\xrightarrow{\alpha_2}\cdots
\xrightarrow{\alpha_{N-1}}(a_N^-, a_N)
\xrightarrow{\alpha_N}(a_1^-, a_1)
\end{equation*}
by Lemma \ref{l41}, hence we infer that $\gg_q'$ contains the sequence
\begin{equation*}
c_1\cdots c_{k-1}c_k(\alpha_1\cdots\alpha_N)^{\infty}.
\end{equation*}
Since
\begin{equation*}
\left(c_1\cdots c_{k-1}c_k(\alpha_1\cdots\alpha_N)^{\infty}\right)_q
=\left(c_1\cdots c_{k-1}c_k^+\ 0^{\infty}\right)_q
=x_1,
\end{equation*}
we conclude that $x_1\in\gg_q$ and
\begin{equation*}
a(x_1,q)=c_1\cdots c_{k-1}\alpha_N(\alpha_1\cdots\alpha_N)^{\infty}\in\gg_q'.
\end{equation*}

If $x_1$ does not have a finite greedy expansion, then $y_1:=\frac{1}{q-1}-x_1\in\overline{\uu_q}\setminus\uu_q$ has a finite greedy expansion, so that $a(y_1,q)\in\gg_q'$ by the preceding paragraph.
Then we also have $\overline{a(y_1,q)}\in\gg_q'$ by the symmetry of the graph $\gg(q)$.
Since $x_1\in\overline{\uu_q}\setminus\uu_q$ has only one doubly infinite expansion, we conclude that $a(x_1,q)=\overline{a(y_1,q)}\in\gg_q'$.
\end{proof}

Let us consider the subgraph $\tilde\gg(q)$ obtained from $\gg(q)$ by keeping only the vertices that are subintervals of $(b_1,a_1)$.
We  denote by $\tilde\gg_q'$ the set of sequences generated by $\tilde\gg(q)$, and we set
\begin{equation*}
\tilde\gg_q:=\set{(c_i)_q\ :\ (c_i)\in \tilde\gg_q'}.
\end{equation*}
Let us also introduce the following sets:
\begin{equation}\label{41}
\begin{split}
&\tilde\vv_q:=
\set{x\in\vv_q\ :\ \overline{\alpha(q)}\le(a_{n+i}(x,q))\le\alpha(q)\qtq{for all}n\ge 0};\\
&\tilde\vv_q':=
\set{a(x,q)\ :\ x\in \tilde\vv_q}.
\end{split}
\end{equation}

\begin{lemma}\label{l42}
We have $\vv_{q}\cap [b_1, a_1]=\tilde \vv_{q}$.
\end{lemma}

\begin{proof}
If $q=M+1$, then $\alpha(q)=M^{\infty}$ and hence $\vv_{q}=[b_1, a_1]=\tilde \vv_{q}=[0,1]$.
Henceforth we assume that $1<q<M+1$.
By definition we have $x\in\vv_q$ if and only if
\begin{align*}
&(a_{n+i}(x,q))\le\alpha(q)\qtq{whenever}a_n(x,q)<M,
\intertext{and}
&\overline{\alpha(q)}\le(a_{n+i}(x,q))\qtq{whenever}a_n(x,q)>0.
\end{align*}
Since  $q<M+1$, we have $0<b_1<a_1<M/(q-1)$, and we have both $a_n(x,q)<M$ and $a_n(x,q)>0$ infinitely many times for each if $0<x<\frac{M}{q-1}$.

Since by Proposition \ref{p21} (ii) we have $x\in [b_1, a_1]$  if and only if
\begin{equation*}
\overline{\alpha(q)}\le a(x,q)\le \alpha(q),
\end{equation*}
we have $\tilde\vv_q\subseteq \vv_q\cap [b_1, a_1]$.
For the converse relation we show that if $x\in\vv_q\cap [b_1, a_1]$, then in fact \begin{equation*}
\overline{\alpha(q)}\le (a_{n+i}(x,q))\le \alpha(q)\qtq{for all}n\ge 0.
\end{equation*}

Indeed, the second inequality holds for $n=0$ and for infinitely many indices $n$ satisfying $a_n(x,q)<M$.
Therefore, if $a_n(x,q)=M$ for some $n$, then, writing $c_i$ instead of $a_i(x,q)$ for brevity, there exist two integers $k$ and $\ell$ such that $0\le k<n\le\ell$,
\begin{equation*}
(k=0\text{ or }c_k<M),\quad
c_{k+1}=c_{k+2}=\cdots=c_{\ell}=M\qtq{and}
c_{\ell+1}<M.
\end{equation*}
Then
\begin{equation*}
c_{n+1}\cdots c_{\ell+1}<M^{\ell-n+1}=c_{k+1}\cdots c_{\ell+k-n+1},
\end{equation*}
and therefore
\begin{equation*}
c_{n+1}c_{n+2}\cdots<c_{k+1}c_{k+2}\cdots
\le\alpha(q).
\end{equation*}
The proof of the second inequality, is similar, by setting $c_i:=\overline{a_i(x,q)}$ instead of $a_i(x,q)$.
\end{proof}

\begin{corollary}\label{c43}
Let $q\in \uuu\setminus\uu$. Then $\tilde \gg_{q}=\tilde \vv_{q}$ and $\tilde \gg'_{q}=\tilde \vv'_{q}$.
\end{corollary}

\begin{proof}
The equality $\tilde \gg_{q}=\tilde \vv_{q}$ follows from the relations
\begin{equation*}
\tilde \gg_{q}=\gg_{q}\cap [b_1, a_1]=\vv_{q}\cap [b_1, a_1]=\tilde \vv_{q}.
\end{equation*}
Here the first equality follows from the definition of $\tilde \gg_q$, the second one  from Theorem \ref{t11} (i) stating that $\gg_q=\vv_q$, and the last one  from Lemma  \ref{l42}.
The equality $\tilde \gg'_{q}=\tilde \vv'_{q}$ hence follows by recalling from
Proposition \ref{p210} (vii) that the only doubly infinite expansion of a number $x\in\vv_q$ is its quasi-greedy expansion.
\end{proof}

\section{Proof of Theorem \ref{t13}}\label{s5}

This section is devoted to the proof of the isomorphism of the graphs $\gg(q_0)$ and $\gg(q_1)$.
We write
\begin{equation*}
\beta(q_0)
=\alpha_1\cdots\alpha_n^+\ 0^{\infty}
\qtq{and}
\beta(q_1)
=\alpha_1\cdots\alpha_n^+\ \overline{\alpha_1\cdots\alpha_n}\ 0^{\infty}.
\end{equation*}

By Lemmas \ref{l34}, \ref{l35} and \ref{l37}  the numbers $a_i, b_i, \theta_j, \eta_j$ for $q_0$ and the similar numbers $\tilde a_i, \tilde \theta_j, \tilde\eta_j$ for $q_1$ have the following greedy expansions in bases $q_0$ and $q_1$, respectively:
\begin{equation*}
b(\theta_0,q_0)=b(\tilde\theta_0,q_1)=0^{\infty},\quad
b(\eta_{M+1},q_0)=b(\tilde\eta_{M+1},q_1)=M^{\infty},
\end{equation*}
and, for $1\le i\le n$ and $1\le j\le M$:
\begin{align*}
b(a_i,q_0)
&:=\alpha_i\cdots\alpha_n^+\ 0^{\infty},\\
b(b_i,q_0)
&:=\overline{\alpha_i\cdots\alpha_n\ (\alpha_1\cdots\alpha_n)^{\infty}},\\
b(\theta_j,q_0)
&:=j0^\infty,\\
b(\eta_j,q_0)
&:=j\overline{(\alpha_1\cdots\alpha_n)^{\infty}},
\intertext{and}
b(\tilde a_i,q_1)
&:=\alpha_i\cdots\alpha_n^+\
\overline{\alpha_1\cdots\alpha_n}\ 0^{\infty},\\
b(\tilde a_{n+i},q_1)
&:=\overline{\alpha_i\cdots\alpha_n}\ 0^{\infty},\\
b(\tilde \theta_j,q_1)
&:=j0^\infty,\\
b(\tilde\eta_j,q_1)
&:=j\overline{\alpha_1\cdots\alpha_n}0^\infty.
\end{align*}
Observe the equalities
\begin{equation*}
a_n=\theta_{\alpha_n^+},\quad
b_n=\eta_{\overline{\alpha_n}},\quad
\tilde a_{2n}=\tilde\theta_{\overline{\alpha_n}}\qtq{and}
\tilde a_n=\tilde\eta_{\alpha_n^+}.
\end{equation*}

We admit temporarily the following technical result:

\begin{lemma}\label{l51}
Let us introduce the sets
\begin{equation*}
A:=\set{a_i, b_i\ :\ i=1,\ldots,n-1}
\cup\set{\theta_0,\ldots, \theta_M},
\cup\set{\eta_1,\ldots, \eta_{M+1}}
\end{equation*}
and
\begin{equation*}
\tilde A:=\set{\tilde a_i, \tilde a_{n+i}\ :\ i=1,\ldots, n-1}
\cup\set{\tilde \theta_0,\ldots, \tilde\theta_M}
\cup\set{\tilde \eta_1,\ldots, \tilde\eta_{M+1}}
\end{equation*}
and a bijection $f:A\to \tilde A$ by the following formula:
\begin{align*}
&f(a_i):=\tilde a_i\qtq{and}
f(b_i):=\tilde a_{n+i}\qtq{for}i=1,\ldots,n-1,\\
&f(\theta_j):=\tilde \theta_j\qtq{for}j=0,\ldots,M,\\
&f(\eta_j):=\tilde \eta_j\qtq{for}j=1,\ldots,M+1.
\end{align*}
Then the map $f$ is increasing.
\end{lemma}

\begin{proof}[Proof of Theorem \ref{t13}]
Using the function $f$ of Lemma \ref{l51}, the map
\begin{equation*}
F((x_i,y_j)):=(f(x_i),f(y_j))
\end{equation*}
defines a bijection $F$ between the vertices $I_i$ of $\gg(q_0)$ and the vertices $J_i$ of $\gg(q_1)$ such that
\begin{equation*}
I_i\xrightarrow{k}I_j
\Longleftrightarrow
F(I_i)\xrightarrow{k}F(I_j).
\end{equation*}
Hence $F$ is an isomorphism between the graphs $\gg(q_0)$ and $\gg(q_1)$.
\end{proof}

\begin{proof}[Proof of Lemma \ref{l51}]
It follows Lemma \ref{l35} (iii)  that  all numbers in $\tilde A$  are distinct, and all numbers in $A$ are distinct, so that both sets have $2n+2M$ elements.
The restriction of $f$ onto
\begin{equation*}
\set{\theta_0,\ldots, \theta_M}
\cup\set{\eta_1,\ldots, \eta_{M+1}}
\end{equation*}
is increasing by the relations
\begin{align*}
&\theta_0<\theta_1<\eta_1
<\theta_2<\eta_2
<\cdots<\eta_{M-1}
<\theta_M<\eta_M<\eta_{M+1}
\intertext{and}
&\tilde \theta_0<\tilde \theta_1<\tilde \eta_1
<\tilde \theta_2<\tilde \eta_2
<\cdots<\tilde \eta_{M-1}
<\tilde \theta_M<\tilde \eta_M<\tilde \eta_{M+1}.
\end{align*}
Therefore the lemma will follow by proving the following equivalences for all $i,j\in\set{1,\ldots,n-1}$ and $k\in\set{1,\ldots,M}$:
\begin{align*}
&\tilde a_i<\tilde a_j\Longleftrightarrow a_i<a_j,\\
&\tilde a_{n+i}<\tilde a_j\Longleftrightarrow b_i<a_j,\\
&\tilde a_i<\tilde a_{n+j}\Longleftrightarrow a_i<b_j,\\
&\tilde a_{n+i}<\tilde a_{n+j}\Longleftrightarrow b_i<b_j,\\
&\tilde a_i<\tilde\theta_k\Longleftrightarrow a_i<\theta_k
\qtq{and}
\tilde a_i<\tilde\eta_k\Longleftrightarrow a_i<\eta_k,\\
&\tilde a_{n+i}<\tilde\theta_k\Longleftrightarrow b_i<\theta_k
\qtq{and}
\tilde a_{n+i}<\tilde\eta_k\Longleftrightarrow b_i<\eta_k,\\
&\tilde a_i>\tilde\theta_k\Longleftrightarrow a_i>\theta_k
\qtq{and}
\tilde a_i>\tilde\eta_k\Longleftrightarrow a_i>\eta_k,\\
&\tilde a_{n+i}>\tilde\theta_k\Longleftrightarrow b_i>\theta_k
\qtq{and}
\tilde a_{n+i}>\tilde\eta_k\Longleftrightarrow b_i>\eta_k.
\end{align*}

Since $f$ is a bijection between the sets $\set{a_i\ :\ i<n}$ and $\set{\tilde a_i\ :\ i<n}$, it suffices to prove the one-sided implications $a_i<a_j\Longrightarrow\tilde a_i<\tilde a_j$ instead of the equivalences of the first line.
A similar remark applies to the following seven lines.

Next, since $\tilde\theta_k^+=\tilde\eta_k$ and $\theta_k^+=\eta_k$, the last two relations in each of the last four lines are equivalent.
Furthermore, the fifth and seventh lines are equivalent because  $\tilde a_i\ne \tilde\theta_k$ and $a_i\ne\theta_k$ for all $i\in\set{1,\ldots,n-1}$ and $k\in\set{1,\ldots,M}$.
Similarly, the sixth and eighth lines are also equivalent.

Therefore it remains to establish the following six implications:

\begin{align}
& a_i<a_j \Longrightarrow \tilde a_i<\tilde a_j,\label{51}\\
&b_i<a_j\Longrightarrow \tilde a_{n+i}<\tilde a_j,\label{52}\\
& a_i<b_j \Longrightarrow \tilde a_i<\tilde a_{n+j},\label{53}\\
&  \tilde a_{n+i}<\tilde a_{n+j} \Longrightarrow b_i<b_j\label{54}\\
&a_i<\theta_k\Longrightarrow \tilde a_i<\tilde\theta_k,\label{55}\\
&\tilde a_{n+i}<\tilde\theta_k\Longrightarrow b_i<\theta_k.\label{56}
\end{align}

Here \eqref{51} is equivalent to the implication
\begin{equation}\label{57}
\alpha_i\cdots\alpha_n^+\ 0^{\infty}
<
\alpha_j\cdots\alpha_n^+\ 0^{\infty}
\Longrightarrow
\alpha_i\cdots\alpha_n^+\
\overline{\alpha_1\cdots\alpha_n}\ 0^{\infty}
<
\alpha_j\cdots\alpha_n^+\
\overline{\alpha_1\cdots\alpha_n}\ 0^{\infty}.
\end{equation}
If $i>j$, then the left inequality in \eqref{57} means that $$\alpha_i\cdots\alpha_n^+\le \alpha_j\cdots\alpha_{n+j-i}.$$
If the inequality is strict, we obtain the right inequality in \eqref{57}.
If we have an equality, then using Proposition \ref{p25} (v) we get
\begin{equation*}
\overline{\alpha_{n+j-i}\cdots\alpha_n^+}<\alpha_1\cdots\alpha_{i-j+1}\qtq{or}
\overline{\alpha_1\cdots\alpha_{i-j+1}}<\alpha_{n+j-i}\cdots\alpha_n^+;
\end{equation*}
together with the equality this yields the desired result.

If $i<j$,  then the left inequality in \eqref{57} means that
\begin{equation*}
\alpha_i\cdots\alpha_{n+i-j}< \alpha_j\cdots\alpha_{n}^+,
\end{equation*}
and this implies the right inequality in \eqref{57}.
\medskip

The relation  \eqref{52} is equivalent to
\begin{equation*}
\overline{\alpha_i\cdots\alpha_n\ (\alpha_1\cdots\alpha_n)^{\infty}}
<\alpha_j\cdots\alpha_n^+\ 0^{\infty}
\Longrightarrow
\overline{\alpha_i\cdots\alpha_n}\ 0^{\infty}
<
\alpha_j\cdots\alpha_n^+\
\overline{\alpha_1\cdots\alpha_n}\ 0^{\infty}.
\end{equation*}
If $i\le j$, then the first inequality implies
\begin{equation*}
\overline{\alpha_i\cdots\alpha_{i+n-j}}<\alpha_j\cdots\alpha_n^+,
\end{equation*}
and this implies the second inequality.
If $i>j$, then the first inequality implies
\begin{equation*}
\overline{\alpha_i\cdots\alpha_n\alpha_1\cdots\alpha_{i-j}}<\alpha_j\cdots\alpha_n^+,
\end{equation*}
and therefore
\begin{equation*}
\overline{\alpha_i\cdots\alpha_n}\ 0^{\infty}
\le
\overline{\alpha_i\cdots\alpha_n\alpha_1\cdots\alpha_{i-j}}\ 0^{\infty}
<\alpha_j\cdots\alpha_n^+\ 0^{\infty}
\le\alpha_j\cdots\alpha_n^+\
\overline{\alpha_1\cdots\alpha_n}\ 0^{\infty}.
\end{equation*}
\medskip

The relation \eqref{53} is equivalent to
\begin{equation}\label{58}
\alpha_i\cdots\alpha_n^+\ 0^{\infty}
<
\overline{\alpha_j\cdots\alpha_n\ (\alpha_1\cdots\alpha_n)^{\infty}}
\Longrightarrow
\alpha_i\cdots\alpha_n^+\
\overline{\alpha_1\cdots\alpha_n}\ 0^{\infty}
<
\overline{\alpha_j\cdots\alpha_n}\ 0^{\infty}.
\end{equation}
If $i\ge j$, then the first inequality of \eqref{58} implies that
\begin{equation*}
\alpha_i\cdots\alpha_n^+\le\overline{\alpha_j\cdots\alpha_{n+j-i}},
\end{equation*}
If this  inequality is strict, then the  right inequality in \eqref{58} follows. If we have an equality, then, since $\beta(q_0)=\alpha_1\cdots\alpha_n^+$, we have
\begin{equation*}
\alpha_{n+j-i+1}\cdots\alpha_{n}<\alpha_{n+j-i+1}\cdots\alpha_{n}^+\le \alpha_{1}\cdots\alpha_{i-j},
\end{equation*}
i.e.,
\begin{equation*}
\overline{\alpha_{1}\cdots\alpha_{i-j}}<\overline{\alpha_{n+j-i+1}\cdots\alpha_{n}}.
\end{equation*}
Together with the equality this yields the desired result.

If $i<j$, then the first inequality of \eqref{58} implies that
\begin{equation*}
\alpha_i\cdots\alpha_{i+n-j}\le\overline{\alpha_j\cdots\alpha_n}.
\end{equation*}
If this inequality is strict, then the second inequality of \eqref{58} follows again.
We conclude the proof in this case by showing that we cannot have an equality.
Indeed, in case of equality we would have the following implications, where the first one follows from the first inequality of \eqref{58}:
\begin{align*}
\alpha_{i+n-j+1}\cdots\alpha_n^+\ 0^{\infty}
<\overline{(\alpha_1\cdots\alpha_n)^{\infty}}
&\Longrightarrow
\alpha_{i+n-j+1}\cdots\alpha_n^+
\le\overline{\alpha_1\cdots\alpha_{j-i}}\\
&\Longrightarrow
\alpha_{i+n-j+1}\cdots\alpha_n
<\overline{\alpha_1\cdots\alpha_{j-i}}\\
&\Longrightarrow
\overline{\alpha_{i+n-j+1}\cdots\alpha_n}
>\alpha_1\cdots\alpha_{j-i}.
\end{align*}
However, the last inequality would contradict Proposition \ref{p25} (v).
\medskip

The relation \eqref{54} is equivalent to
\begin{equation}\label{59}
\overline{\alpha_i\cdots\alpha_n(\alpha_1\cdots\alpha_n)^{\infty}}<\overline{\alpha_j\cdots\alpha_n(\alpha_1\cdots\alpha_n)^{\infty}}
\Longrightarrow
\overline{\alpha_i\cdots\alpha_n}\ 0^{\infty}<\overline{\alpha_j\cdots\alpha_n}\ 0^{\infty}
\end{equation}

If $i>j$, then the left inequality in \eqref{59} implies that
\begin{equation*}
\overline{\alpha_i\cdots\alpha_n}\leq \overline{\alpha_j\cdots\alpha_{j-i+n}},
\end{equation*}
and hence the right inequality of \eqref{59} follows.

If $i<j$, then the first inequality of \eqref{59} implies that
\begin{equation*}
 \overline{\alpha_i\cdots\alpha_{i-j+n}}\leq \overline{\alpha_j\cdots\alpha_{n}},
\end{equation*}
and it remains to exclude the equality.
In case of equality the first inequality of \eqref{59} would also imply the inequality
\begin{equation*}
\overline{\alpha_{i-j+n+1}\cdots\alpha_n(\alpha_1\cdots\alpha_n)^{\infty}}<\overline{(\alpha_1\cdots\alpha_n)^{\infty}},
\end{equation*}
contradicting Proposition \ref{p23} (ii).
\medskip

The implication \eqref{55} is equivalent to
\begin{equation*}
\alpha_i\cdots\alpha_n^+0^{\infty}<k0^{\infty} \Longrightarrow \alpha_i\cdots\alpha_n^+
\overline{\alpha_1\cdots\alpha_n}0^{\infty}<k0^{\infty},
\end{equation*}
and this is obvious because both inequalities are equivalent to $\alpha_i<k$.
\medskip

Finally, the implication \eqref{56} is equivalent to
\begin{equation*}
\overline{\alpha_i\cdots\alpha_n(\alpha_1\cdots\alpha_n)^{\infty}}<k0^{\infty}\Longrightarrow
\overline{\alpha_i\cdots\alpha_n}0^{\infty}<k0^{\infty},
\end{equation*}
and this is obvious because both inequalities are equivalent to $\overline{\alpha_i}<k$.
\end{proof}

\section{Proof of Theorem \ref{t14}}\label{s6}

We have proved in the preceding section that the graphs $\gg(q_0)$ and $\gg(q_1)$ are isomorphic.
We will see that the graphs $\gg(q_m)$ and $\gg(q_{m+1})$ are not isomorphic if $m\ge 1$.
For this we need some particular properties of the graphs $\gg(q_{m+1})$ for $m\ge 1$.

Fix an integer $m\ge 1$, write
\begin{align*}
\beta(q_m)
=\alpha_1\cdots\alpha_n^+\ \overline{\alpha_1\cdots\alpha_n}\ 0^{\infty}
\end{align*}
and
\begin{align*}
\beta(q_{m+1})
=\alpha_1\cdots\alpha_n^+\ \overline{\alpha_1\cdots\alpha_n\
\alpha_1\cdots\alpha_n^+}\
\alpha_1\cdots\alpha_n^+\ 0^{\infty},
\end{align*}
and introduce the usual numbers $a_i,\theta_j,\eta_j$ for $q_m$ and $\tilde a_i, \tilde\theta_j, \tilde\eta_j$ for $q_{m+1}$.
It follows from the results of Section \ref{s3} that $\gg(q_m)$ and $\gg(q_{m+1})$ have respectively $2n+M-1$ and $4n+M-1$ vertices, respectively, and the greedy expansions of the endpoints of the interval-vertices are given by
\begin{equation*}
b(\theta_0,q_m)=b(\tilde\theta_0,q_{m+1})=0^{\infty},\quad
b(\eta_{M+1},q_m)=b(\tilde\eta_{M+1},q_{m+1})=M^{\infty},
\end{equation*}
and the following formulas where $i\in\set{1,\ldots,n}$ and $j\in\set{1,\ldots,M}$:
\begin{equation}\label{61}
\begin{split}
&b(a_i,q_m)=\alpha_i\cdots\alpha_n^+\ \overline{\alpha_1\cdots\alpha_n}\ 0^{\infty},\\
&b(a_{n+i},q_m)=\overline{\alpha_i\cdots\alpha_n}\ 0^{\infty},\\
&b(\theta_j,q_m)=j\ 0^{\infty},\\
&b(\eta_j,q_m)=j\ \overline{\alpha_1\cdots\alpha_n}\ 0^{\infty}
\end{split}
\end{equation}
for $q_m$, and
\begin{equation}\label{62}
\begin{split}
&b(\tilde a_i,q_{m+1})=
\alpha_i\cdots\alpha_n^+\ \overline{\alpha_1\cdots\alpha_n\
\alpha_1\cdots\alpha_n^+}\
\alpha_1\cdots\alpha_n^+\ 0^{\infty},\\
&b(\tilde a_{n+i},q_{m+1})=\overline{\alpha_i\cdots\alpha_n\
\alpha_1\cdots\alpha_n^+}\
\alpha_1\cdots\alpha_n^+\ 0^{\infty},\\
&b(\tilde a_{2n+i},q_{m+1})=
\overline{\alpha_i\cdots\alpha_n^+}\
\alpha_1\cdots\alpha_n^+\ 0^{\infty},\\
&b(\tilde a_{3n+i},q_{m+1})=
\alpha_i\cdots\alpha_n^+\ 0^{\infty},\\
&b(\tilde \theta_j,q_{m+1})=j\ 0^{\infty},\\
&b(\tilde \eta_j,q_{m+1})=j\ \overline{\alpha_1\cdots\alpha_n^+}
\ \alpha_1\cdots\alpha_n^+
\ 0^{\infty}
\end{split}
\end{equation}
for $q_{m+1}$.

It follows that
\begin{equation*}
a_{2n}=\theta_{\overline{\alpha_n}},\quad
a_n=\eta_{\alpha_n^+},\quad
\tilde a_{4n}=\tilde\theta_{\alpha_n^+}
\qtq{and}
\tilde a_{2n}=\tilde\eta_{\overline{\alpha_n}}.
\end{equation*}
We will constantly use these formulas without reference in this section.

\begin{lemma}\label{l61}
We consider the graph $\gg(q_{m+1})$ and the corresponding points $\tilde a_i, \tilde\theta_j, \tilde\eta_j$ for an integer $m\ge 1$.
\begin{enumerate}[\upshape (i)]
\item $(\tilde \eta_{\alpha_n^+}, \tilde a_n)$ and $(\tilde a_{3n}, \tilde \theta_{\overline{a_n}})$ are vertices of the graph $\gg(q_{m+1})$.
\item $\gg(q_{m+1})$ contains the following cycle of $2n$ vertices, denoted by $\cc_{m+1}$:
\begin{equation*}
\begin{split}
(\tilde a_{3n+1},\tilde a_1)
&\xrightarrow{\alpha_1}
\cdots
\xrightarrow{\alpha_{n-2}}
(\tilde a_{4n-1},\tilde a_{n-1})
\xrightarrow{\alpha_{n-1}}
(\tilde \eta_{\alpha_n^+},\tilde a_n)
\xrightarrow{\alpha_n^+}\\
(\tilde a_{2n+1},\tilde a_{n+1})
&\xrightarrow{\overline{\alpha_1}}
\cdots
\xrightarrow{\overline{\alpha_{n-2}}}
(\tilde a_{3n-1},\tilde a_{2n-1})
\xrightarrow{\overline{\alpha_{n-1}}}
(\tilde a_{3n},\tilde \theta_{\overline{\alpha_n}})
\xrightarrow{\overline{\alpha_n^+}}
(\tilde a_{3n+1},\tilde a_1).
\end{split}
\end{equation*}
\item There is no any other outgoing edge in $\gg(q_{m+1})$ from the vertices of $\cc_{m+1}$.
\item The cycle contains no consecutive intervals.
\item $\gg(q_{m+1})$ contains the  two paths
\begin{align*}
&(\tilde a_{2n-1},\tilde a_{2n-1}^+)
\xrightarrow{\overline{\alpha_{n-1}}}
(\tilde a_{2n},\tilde a_{2n}^+)
=(\tilde \eta_{\overline{\alpha_n}},\tilde \eta_{\overline{\alpha_n}}^+)
\xrightarrow{\overline{\alpha_{n}}}(\tilde a_{2n+1}, \tilde a_{2n+1}^+)
\intertext{and}
&(\tilde a_{4n-1}^-,\tilde a_{4n-1})
\xrightarrow{\alpha_{n-1}}
(\tilde a_{4n}^-,\tilde a_{4n})
=(\tilde \theta_{\alpha_n^+}^-,\tilde \theta_{\alpha_n^+})
\xrightarrow{\alpha_n}(\tilde a_1^-, \tilde a_1).
\end{align*}
\end{enumerate}
\end{lemma}

\begin{proof}
(i) We only consider $(\tilde\eta_{\alpha_n^+},\tilde a_n)$; the other case hence will follow by reflection because
\begin{equation*}
\tilde a_{3n}=\frac{M}{q-1}-\tilde a_n
\qtq{and}
\tilde \theta_{\overline{a_n}}=\frac{M}{q-1}-\tilde \eta_{\alpha_n^+}.
\end{equation*}

The inequality $\tilde \eta_{\alpha_n^+}<\tilde a_n$ follows from the corresponding lexicographic inequality between their greedy expansions
\begin{equation*}
b(\eta_{\alpha_n^+},q_{m+1})
=\alpha_n^+\
\overline{\alpha_1\cdots\alpha_n^+}\ \alpha_1\cdots\alpha_n^+\ 0^\infty
\end{equation*}
and
\begin{equation*}
b(\tilde a_n,q_{m+1})
=\alpha_n^+\
\overline{\alpha_1\cdots\alpha_n\
\alpha_1\cdots\alpha_n^+}\ \alpha_1\cdots\alpha_n^+\ 0^\infty.
\end{equation*}
It remains to prove that there is no further point $\tilde a_j$ between $\eta_{\alpha_n^+}$ and $\tilde a_n$.

Applying Proposition \ref{p25} (v) for $\alpha(q_m)=(\alpha_1\cdots\alpha_{n}^+\ \overline{\alpha_1\cdots\alpha_{n}^+})^\infty$ with $N=n$ and $i=n-j$ we obtain that
\begin{equation}\label{63}\overline{\alpha_{n-j+1}\cdots\alpha_n^+}<\alpha_1\cdots\alpha_{j},\quad j=1,\ldots, n.
\end{equation}

Similarly, applying Proposition \ref{p25}  (v) for
\begin{equation*}
\alpha(q_{m+1})=(\alpha_1\cdots\alpha_n^+\ \overline{\alpha_1\cdots\alpha_n\
\alpha_1\cdots\alpha_n^+}\
\alpha_1\cdots\alpha_n)^\infty
\end{equation*}
with $N=2n$ and with $i=j$ and $i=n+j$, respectively, we obtain that
\begin{equation}\label{64}
\alpha_{j+1}\cdots\alpha_n^+\overline{\alpha_1\cdots \alpha_n}>\overline{\alpha_{1}\cdots\alpha_{n}^+}\alpha_1\cdots\alpha_{n-j},\quad j=0, 1,\ldots, n,
\end{equation}
and
\begin{equation}\label{65}
\alpha_{j+1}\cdots\alpha_n<\alpha_{1}\cdots\alpha_{n-j},\quad j=1,\ldots, n-1.
\end{equation}

Assume on the contrary that there exists a point $\tilde a_j\in (\eta_{\alpha_n^+},\tilde a_n)$.
Then $j\ne n$.
Writing $b(\tilde a_j,q_{m+1})=c_1\cdots c_K^+0^\infty$ we have the lexicographic relations
\begin{equation}\label{66}
\alpha_{n}^+\ \overline{\alpha_1\cdots\alpha_{n}^+}\ \alpha_1\cdots\alpha_{n}^+0^\infty
<c_1\cdots c_K^+0^\infty
<\alpha_{n}^+\
\overline{\alpha_1\cdots\alpha_{n}\ \alpha_1\cdots\alpha_{n}^+}\ \alpha_1\cdots\alpha_{n}^+0^\infty.
\end{equation}
Hence $K\ge n+1$,  $c_1\cdots c_n=\alpha_{n}^+\overline{\alpha_1\cdots\alpha_{n-1}}$, and $c_{n+1}=\overline{\alpha_{n}^+}$ or $c_{n+1}=\overline{\alpha_{n}}$.
Furthermore, using the greedy expansion of $\tilde a_j$, in case $c_{n+1}=\overline{\alpha_{n}^+}$ we have $K\ge 2n+2$, hence $j<2n$, and
\begin{multline}\label{67}
\alpha_{n}^+\ \overline{\alpha_1\cdots\alpha_{n}^+}\ \alpha_1\cdots\alpha_{n}^+c_{2n+2}\cdots c_K^+\\
=
\begin{cases}
\alpha_j\cdots\alpha_n^+\ \overline{\alpha_1\cdots\alpha_n\ \alpha_1\cdots\alpha_n^+}\
\alpha_1\cdots\alpha_n^+&\text{if }0<j<n,\\
\overline{\alpha_{j-n}\cdots\alpha_n\ \alpha_1\cdots\alpha_n^+}\
\alpha_1\cdots\alpha_n^+&\text{if }n<j<2n,
\end{cases}
\end{multline}
while in case $c_{n+1}=\overline{\alpha_{n}}$  we have $K\ge n+1$, hence $j\le 3n$, and
\begin{multline*}
\alpha_{n}^+\ \overline{\alpha_1\cdots\alpha_{n}}\ c_{n+2}\cdots c_K^+\\
=
\begin{cases}
\alpha_j\cdots\alpha_n^+\ \overline{\alpha_1\cdots\alpha_n\
\alpha_1\cdots\alpha_n^+}\
\alpha_1\cdots\alpha_n^+
&\text{if }0<j<n,\\
\overline{\alpha_{j-n}\cdots\alpha_n\
\alpha_1\cdots\alpha_n^+}\
\alpha_1\cdots\alpha_n^+
&\text{if }n<j\le 2n,\\
\overline{\alpha_{j-2n}\cdots\alpha_n^+}\
\alpha_1\cdots\alpha_n^+
&\text{if }2n<j\le 3n.
\end{cases}
\end{multline*}

If $c_{n+1}=\overline{\alpha_{n}^+}$, then we infer from \eqref{66} the equalities
\begin{equation*}
\begin{cases}
\overline{\alpha_{1}\cdots\alpha_{n}^+}\alpha_1\cdots\alpha_{n-j}=\alpha_{j+1}\cdots\alpha_n^+\overline{\alpha_1\cdots \alpha_n}&\qtq{if}0<j<n,\\
\overline{\alpha_{1}\cdots\alpha_{n-\ell}}=\overline{\alpha_{\ell+1}\cdots\alpha_n}&\qtq{if}n<j=n+\ell<2n,
\end{cases}
\end{equation*}
contradicting \eqref{64} and \eqref{65}, respectively.

If $c_{n+1}=\overline{\alpha_{n}}$, then we infer from \eqref{67} the relations
\begin{equation*}
\begin{cases}
\overline{\alpha_{1}\cdots\alpha_{n-j}}=\alpha_{j+1}\cdots\alpha_n^+
&\qtq{if}0<j<n,\\
\overline{\alpha_{1}\cdots\alpha_{n-\ell}}=\overline{\alpha_{\ell+1}\cdots\alpha_n}
&\qtq{if}n<j=n+\ell<2n,\\
\overline{\alpha_{1}\cdots\alpha_{n}}=\overline{\alpha_{1}\cdots\alpha_n^+}&\qtq{if}j=2n,\\
c_{n+2}\cdots c_{2n-\ell+1}=\alpha_{\ell+1}\cdots\alpha_{n}^+
&\qtq{if}2n<j=2n+\ell\le 3n.
\end{cases}
\end{equation*}
They first two equalities contradict \eqref{63} and \eqref{65}, respectively.
The third equality is impossible because $\alpha_n\ne\alpha_n^+$.
Combining the last equality with the second inequality of \eqref{66} we obtain
\begin{equation*}
\alpha_{\ell+1}\cdots\alpha_{n}^+\leq \overline{\alpha_1\cdots\alpha_{n-\ell}},
\end{equation*}
contradicting  \eqref{63} again.

\medskip

(ii) The explicit formulas of the greedy expansions of the numbers $\tilde a_i$ show that none of the indicated intervals is degenerate and that they are disjoint from the switch region $S_{q_{m+1}}$.
Furthermore using Lemma \ref{l34} (v), the relations
\begin{equation*}
\tilde \eta_{\alpha_n^+} \xmapsto{\alpha_n^+} \tilde a_{2n+1}, \qtq{and}
\tilde \theta_{\overline{\alpha_n}} \xmapsto{\overline{\alpha_n^+}}\tilde a_1
\end{equation*}
and that
\begin{itemize}
\item $b(\tilde a_{3n+i},q_{m+1})$ and $b(\tilde a_{i},q_{m+1})$ start with the same digit $\alpha_i$ for $0<i<n$,
\item $b(\tilde \eta_{\alpha_n^+},q_{m+1})$ and $b(\tilde a_n,q_{m+1})$ start with the same digit $\alpha_n^+$,
\item $b(\tilde a_{2n+i},q_{m+1})$ and $b(\tilde a_{n+i},q_{m+1})$ start with the same digit $\overline{\alpha_i}$ for $0<i<n$,
\item $b(\tilde a_{3n},q_{m+1})$ and $a(\tilde \theta_{\overline{\alpha_n}},q_{m+1})$ start with the same digit
$\overline{\alpha_n^+}$,
\end{itemize}
we obtain that each interval is the image of the preceding one by the map $T_{\alpha_i}$ indicated on the arrow between them, except two:
\begin{equation}\label{68}
T_{\alpha_{n-1}}\left((\tilde a_{4n-1},\tilde a_{n-1})\right)=(\tilde \theta_{\alpha_n^+},\tilde a_n)
\qtq{and}
T_{\overline{\alpha_{n-1}}}\left((\tilde a_{3n-1},\tilde a_{2n-1})\right)=(\tilde a_{3n},\tilde \eta_{\overline{\alpha_n}}).
\end{equation}
We claim that these intervals are vertices of $\gg(q_{m+1})$.
For this we need to show that none of them contains another point $\tilde a_j$ in its interior.

This is true for the intervals $(\tilde \eta_{\alpha_n^+},\tilde a_n)$ and $(\tilde a_{3n},\tilde \theta_{\overline{\alpha_n}})$ by (i).
Next we show that none of the intervals $(\tilde a_{3n+i},\tilde a_i)$, $0<i<n$ contains another point $\tilde a_j$ in its interior.

Assume on the contrary that $\tilde a_{3n+i}<\tilde a_j<\tilde a_i$ for some $0<i<n$ and $0<j\le 4n$.
Then
\begin{equation*}
(T_{\alpha_{n-1}}\circ\cdots\circ T_{\alpha_i})(\tilde a_j)
\in (\tilde a_{4n},\tilde a_n)
\end{equation*}
and hence
\begin{equation*}
(T_{\alpha_{n-1}}\circ\cdots\circ T_{\alpha_i})(\tilde a_j)=\tilde \eta_{\alpha_n^+}.
\end{equation*}
Using the relations (we use here the notation $\alpha(q_{m+1})=\alpha_1(q_{m+1})\alpha_2(q_{m+1})\cdots$)
\begin{equation*}
\tilde a_1\xmapsto{\alpha_1(q_{m+1})}
\cdots\xmapsto{\alpha_{4n-1}(q_{m+1})}
\tilde a_{4n}\xmapsto{\alpha_{4n}(q_{m+1})}\tilde a_1
\end{equation*}
and the equalities
\begin{equation*}
\alpha_{3n+\ell}(q_{m+1})
=\alpha_{\ell}(q_{m+1})
\qtq{for}\ell=1,\ldots, n-1,
\end{equation*}
hence we infer that there exists an index $k$ such that
\begin{equation*}
\tilde a_{4n-1}<\tilde a_k<\tilde a_{n-1}
\qtq{and}
\tilde a_{4n}<T_{\alpha_{4n-1}}(\tilde a_k)<\tilde a_n.
\end{equation*}
By (i) this implies that $T_{\alpha_{4n-1}}(\tilde a_k)=\tilde \eta_{\alpha_n^+}$.
Since $\alpha_{4n-1}=\alpha_{n-1}$, hence we infer  that
\begin{equation*}
\alpha_{n-1}\alpha_n^+
\ \overline{\alpha_1\cdots\alpha_n^+}
\ \alpha_1\cdots\alpha_n^+\ 0^{\infty}
\end{equation*}
is an expansion of $\tilde a_k$.
This expansion is greedy because
\begin{equation*}
\tilde a_k<\tilde a_{n-1}<\frac{\alpha_{n-1}^+}{q_{m+1}}
\end{equation*}
(recall that the first digit of $b(\tilde a_{n-1},q_{m+1})$ is $\alpha_{n-1}$), and because
\begin{equation*}
\alpha_{n}^+\
\overline{\alpha_1\cdots\alpha_{n}^+}\ \alpha_1\cdots\alpha_{n}^+\ 0^\infty=b(\tilde \eta_{\alpha_n^+},q_{m+1})
\end{equation*}
is a greedy expansion.
Thus we have
\begin{equation*}
b(\tilde a_k,q_{m+1})=
\alpha_{n-1}\alpha_n^+
\ \overline{\alpha_1\cdots\alpha_n^+}
\ \alpha_1\cdots\alpha_n^+\ 0^{\infty}.
\end{equation*}
Comparing with the formulas of $b(\tilde a_i,q_{m+1})$ at the beginning of the section, hence we infer that $k=2n-1$ and
$\overline{\alpha_{n-1}\alpha_n}=\alpha_{n-1}\alpha_n^+$.
But this is impossible because $\overline{\alpha_{n-1}}\ne \alpha_{n-1}$ if $M$ is odd, and $\overline{\alpha_n}\ne \alpha_n^+$ if $M$ is even.
This contradiction proves our initial claim: the intervals $(\tilde a_{3n+i},\tilde a_i)$ for $0<i<n$ are vertices of $\gg(q_{m+1})$.

By reflection and symmetry we obtain that the intervals
\begin{align*}
(\tilde a_{2n+i},\tilde a_{n+i})
&=\left(\frac{M}{q_{m+1}-1}-\tilde b_{2n+i},\frac{M}{q_{m+1}-1}-\tilde b_{n+i}\right)\\
&=\left(\frac{M}{q_{m+1}-1}-\tilde a_{i},\frac{M}{q_{m+1}-1}-\tilde a_{3n+i}\right)
\end{align*}
for $0<i<n$ are vertices of $\gg(q_{m+1})$.
\medskip

(iii) This  follows from the proof of (ii) showing that for any relation $I\xmapsto{k}J$ in the cycle, $T_k(I)$ contains no other interval-vertex of the graph than $J$.
More precisely, we have $T_k(I)=J$ for all but two vertices: those indicated in \eqref{68}.
Furthermore,
\begin{equation*}
[\tilde \theta_{\alpha_n^+},\tilde a_n]
=[\tilde \theta_{\alpha_n^+},\tilde \eta_{\alpha_n^+}]
\cup[\tilde \eta_{\alpha_n^+},\tilde a_n]
\qtq{and}
[\tilde a_{3n},\tilde \eta_{\overline{\alpha_n}}]
=[\tilde a_{3n},\tilde \theta_{\overline{\alpha_n}}]
\cup [\theta_{\overline{\alpha_n}},\tilde \eta_{\overline{\alpha_n}}],
\end{equation*}
and the switch intervals $[\tilde \theta_{\alpha_n^+},\tilde \eta_{\alpha_n^+}]$ and $[\theta_{\overline{\alpha_n}},\tilde \eta_{\overline{\alpha_n}}]$ are not vertices of $\gg(q_{m+1})$.
\medskip

(iv) This follows by observing that no two intervals of the cycle have any common endpoint.
\medskip

(v) The existence of both paths follows from Lemma \ref{l34} (v).
\end{proof}

Now we clarify the relationship between the graphs $\gg(q_m)$ and $\gg(q_{m+1})$ for some $m\ge 1$.
Let us introduce the corresponding points
\begin{equation*}
a_1,\ldots, a_{2n},
\quad \theta_0,\ldots,\theta_M,
\quad \eta_1,\ldots,\eta_{M+1}
\end{equation*}
associated with $q_m$, and
\begin{equation*}
\tilde a_1,\ldots, \tilde a_{4n},
\quad \tilde \theta_0,\ldots,\tilde \theta_M,
\quad \tilde \eta_1,\ldots,\tilde \eta_{M+1}
\end{equation*}
associated with$q_{m+1}$.
We recall that
\begin{equation*}
\theta_0=\tilde\theta_0=0,
\quad \eta_{M+1}=\frac{M}{q_m-1}.
\quad \tilde \eta_{M+1}=\frac{M}{q_{m+1}-1}
\end{equation*}
and
\begin{equation*}
a_n=\eta_{\alpha_n^+},
\quad a_{2n}=\theta_{\overline{\alpha_n}},
\quad \tilde a_{2n}=\tilde \eta_{\overline{\alpha_n}},
\quad \tilde a_{4n}=\tilde \theta_{\alpha_n^+}.
\end{equation*}
We will use the notation
\begin{equation*}
\tilde z:=
\begin{cases}
\tilde a_i&\text{if $z=a_i$ for some $i$},\\
\tilde \theta_j&\text{if $z=\theta_j$ for some $j$},\\
\tilde \eta_j&\text{if $z=\eta_j$ for some $j$}.
\end{cases}
\end{equation*}
Observe that if $(w,z)$ runs over the vertices of $\gg(q_m)$, then $w$ runs over the points
\begin{equation}\label{69}
a_1,\ldots, a_{n-1},\quad
a_{n+1},\ldots, a_{2n-1},\quad
\theta_0,\quad
\eta_1,\ldots, \eta_M,
\end{equation}
and $z$ runs over the points
\begin{equation}\label{610}
a_1,\ldots, a_{n-1},\quad
a_{n+1},\ldots, a_{2n-1},\quad
\theta_1,\ldots, \theta_M\qtq{and}\eta_{M+1}.
\end{equation}

Let us introduce the following function $f=f_m$:
\begin{equation}\label{611}
\begin{cases}
&f(a_i):=\tilde a_i
\qtq{for}i=1,\ldots, n-1,\\
&f(a_i):=\tilde a_i
\qtq{for}i=n+1,\ldots, 2n-1,\\
&f(\theta_j):=\tilde\theta_j\qtq{for}j=0,1,\ldots,M, \\
&f(\eta_j):=\tilde \eta_j\qtq{for}j\in\set{1,\ldots,M+1}\setminus\set{\alpha_n^+},\\
&f(\eta_{\alpha_n^+}):=\tilde a_n.
\end{cases}
\end{equation}
Note that $a_n=\eta_{\alpha_n^+}$ and $a_{2n}=\theta_{\overline{\alpha_n}}$, so that
\begin{equation*}
f(a_n)=\tilde a_n\qtq{and}
f(a_{2n})=\tilde \theta_{\overline{\alpha_n}}.
\end{equation*}

\begin{lemma}\label{l62}
Let $I\xrightarrow{k}J$ in $\gg(q_m)$, and let $z$ be one of the endpoints of $I$.
\begin{enumerate}[\upshape (i)]
\item If $z=a_{2n-1}$, then $(f(T_k(z)),T_k(f(z)))=(\tilde\theta_{\overline{\alpha_n}}
,\tilde\eta_{\overline{\alpha_n}})$ is a switch interval.
\smallskip
\item If $z\in \set{\eta_1,\ldots \eta_M}\setminus\set{\eta_{\alpha_n^+}}$, then
$(T_k(f(z)), f(T_k(z)))=(\tilde a_{2n+1},\tilde a_{n+1})$ is a vertex of $\gg(q_{m+1})$.
\smallskip
\item Otherwise we have
$f(T_k(z))=T_k(f(z))$.
\end{enumerate}
\end{lemma}

\begin{proof}
We apply Lemmas \ref{l33}, \ref{l34}, \ref{l39} and we use the relations \eqref{61}, \eqref{62}.
We write $T:=T_k$ for brevity.
For $z=a_{2n-1}$ we have
\begin{align*}
&f(T(a_{2n-1}))=f(a_{2n})
=f(\theta_{\overline{\alpha_n}})
=\tilde \theta_{\overline{\alpha_n}}
\intertext{and}
&T(f(a_{2n-1}))=T(\tilde a_{2n-1})
=\tilde a_{2n}=\tilde\eta_{\overline{\alpha_n}}.
\end{align*}
If $z=\eta_j$ for some $j\in\set{1,\ldots, M}\setminus\set{\alpha_n^+}$, then
\begin{equation*}
f(T(\eta_j))=f(a_{n+1})=\tilde a_{n+1}\qtq{and}
T(f(\eta_j))=T(\tilde\eta_j)=\tilde a_{2n+1}.
\end{equation*}
If $z=a_i$ for some $i\in\set{1,\ldots,2n-2}$ (this include the case $z=\eta_{\alpha_n^+}$), then
\begin{equation*}
f(T(a_i))=f(a_{i+1})=\tilde a_{i+1}
\qtq{and}
T(f(a_i))=T(\tilde a_i)=\tilde a_{i+1}.
\end{equation*}
If $z=\theta_j$ for some  $j\in\set{1,\ldots,M}$ (this include the case $z=a_{2n}$), then
\begin{equation*}
f(T(\theta_j))=f(a_1)=\tilde a_1
\qtq{and}
T(f(\theta_j))=T(\tilde\theta_j)=\tilde a_1.
\end{equation*}
Finally, we have
\begin{equation*}
f(T(\theta_0))=f(\theta_0)=\tilde \theta_0
\qtq{and}
T(f(\theta_0))=T(\tilde\theta_0)=\tilde\theta_0,
\end{equation*}
and
\begin{equation*}
f(T(\eta_{M+1}))=f(\eta_{M+1})=\tilde\eta_{M+1}
\qtq{and}
T(f(\eta_{M+1}))=T(\tilde\eta_{M+1})=\tilde\eta_{M+1}.
\end{equation*}
The lemma follows from these relations and from Lemma \ref{l61} (ii).
\end{proof}

We prove a variant of Lemma \ref{l51}; we recall that $a_n=\eta_{\alpha_n^+}$ and $a_{2n}=\theta_{\overline{\alpha_n}}$.

\begin{lemma}\label{l63}
We consider the graphs $\gg(q_m)$, $\gg(q_{m+1})$, and the function $f=f_m$.
\begin{enumerate}[\upshape (i)]
\item The function $f$ is increasing.
\item Let $(w,z)$ be a vertex  of $\gg(q_m)$.
\begin{enumerate}[\upshape (a)]
\item If $z=a_i$ for some $i\in\set{1,\ldots,n-1,n+1,\ldots,2n}$,
then $(f(w),f(z))$ contains two vertices of $\gg(q_{m+1})$:
\begin{align*}
&(f(w),\tilde a_{3n+i})\qtq{and}
(\tilde a_{3n+i},\tilde a_i)\qtq{if}i=1,\ldots,n-1,\\
&(f(w),\tilde a_{n+i})\qtq{and}
(\tilde a_{n+i},\tilde a_i)\qtq{if}i=n+1,\ldots,2n.
\end{align*}
\item If $z\in\set{\theta_1,\ldots,\theta_M,\eta_{M+1}}\setminus\set{a_{2n}}$,
then $(f(w),f(z))$ is a vertex of $\gg(q_{m+1})$.
\item $(\tilde\eta_{\alpha_n^+},\tilde a_n)$ is a vertex of $\gg(q_{m+1})$.
\end{enumerate}
Moreover, this is the complete list of the vertices of $\gg(q_{m+1})$.
\item The map
\begin{equation*}
F((x,x^+)):=(f(x),f(x)^+)
\end{equation*}
defines an isomorphism between $\gg(q_m)$ and a subgraph $\hat\gg(q_m)$ of $\gg(q_{m+1})$.
\item The vertices of $\hat\gg(q_m)$ and $\cc_{m+1}$ form a partition of the vertices of  $\gg(q_{m+1})$.
\item If $m\ge 2$, then there exists a path in $\gg(q_{m+1})$ from the image $\hat\cc_m$ of $\cc_m$ to  $\cc_{m+1}$.
\end{enumerate}
\end{lemma}

\begin{proof}
(i) By Lemma \ref{l61} (i) the interval $(\tilde\eta_{\alpha_n^+},\tilde a_n)$ is a vertex of $\gg(q_{m+1})$, and therefore it does not contain any other point $\tilde a_i$ for $0<i<n$ or $n<i<2n$, $\tilde\theta_j$ for $0\le j\le M$, or $\tilde\eta_j$ for $1\le j\le M+1$.Therefore it suffices to show that the modified function $g$, defined the following formulas, is increasing:

\begin{equation*}
\begin{cases}
&g(a_i):=\tilde a_i
\qtq{for}i=1,\ldots, n-1,\\
&g(a_{n+i}):=\tilde a_{n+i}
\qtq{for}i=1,\ldots, n-1,\\
&g(\theta_j):=\tilde\theta_j\qtq{for}j=0,1,\ldots,M, \\
&g(\eta_j):=\tilde \eta_j\qtq{for}j=1,\ldots,M+1.
\end{cases}
\end{equation*}
Repeating the proof of Lemma \ref{l51}, and recalling that $b_i=a_{n+i}$ for $i=1,\ldots,n-1,$ it suffices to check the following implications (cf. \eqref{51}--\eqref{56}) for $i=1,\ldots, n-1$ and $j=1,\ldots, M$:

\begin{align}
& a_i<a_j \Longrightarrow \tilde a_i<\tilde a_j,\label{612}\\
&a_{n+i}<a_j\Longrightarrow \tilde a_{n+i}<\tilde a_j,\label{613}\\
& a_i<a_{n+j} \Longrightarrow \tilde a_i<\tilde a_{n+j},\label{614}\\
& a_{n+i}<a_{n+j}\Longrightarrow\tilde a_{n+i}<\tilde a_{n+j}, \label{615}\\
&a_i<\theta_k\Longrightarrow \tilde a_i<\tilde\theta_k,\label{616}\\
&\tilde a_{n+i}<\tilde\theta_k\Longrightarrow a_{n+i}<\theta_k.\label{617}
\end{align}
The last two implications hold because both inequalities in \eqref{616} are equivalent to $\alpha_i<k$, and both inequalities in \eqref{617} are equivalent to $\overline{\alpha_i}<k$.
We prove them by adapting the proofs of \eqref{51}--\eqref{54}.

The implication \eqref{612} is equivalent to the implication
\begin{equation}\label{618}
\begin{split}
&\alpha_i\cdots\alpha_n^+\overline{\alpha_1\cdots\alpha_n}\ 0^{\infty}
<
\alpha_j\cdots\alpha_n^+\overline{\alpha_1\cdots\alpha_n}\ 0^{\infty}\\
&\Longrightarrow
\alpha_i\cdots\alpha_n^+\
 \overline{\alpha_1\cdots\alpha_n\
\alpha_1\cdots\alpha_n^+}\
\alpha_1\cdots\alpha_n^+\ 0^{\infty}\\
&\qquad\qquad <
\alpha_j\cdots\alpha_n^+\
 \overline{\alpha_1\cdots\alpha_n\
\alpha_1\cdots\alpha_n^+}\
\alpha_1\cdots\alpha_n^+\ 0^{\infty}.
\end{split}
\end{equation}
If $i>j$, then the left inequality in \eqref{618} implies that $$\alpha_i\cdots\alpha_n^+\le \alpha_j\cdots\alpha_{n+j-i}.$$
If the inequality is strict, we obtain the right inequality in \eqref{618}.
If we have an equality, then using Proposition \ref{p25} (v) we get
\begin{equation}\label{619}
\overline{\alpha_{n+j-i+1}\cdots\alpha_n^+}<\alpha_1\cdots\alpha_{i-j}\qtq{or}
\overline{\alpha_1\cdots\alpha_{i-j}}<\alpha_{n+j-i+1}\cdots\alpha_n^+;
\end{equation}
together with the equality this yields the desired result.

If $i<j$,  then the left inequality in \eqref{618} means that
\begin{equation*}
\alpha_i\cdots\alpha_{n}^+\overline{\alpha_1\cdots\alpha_{n+i-j}}< \alpha_j\cdots\alpha_{n}^+\overline{\alpha_1\cdots\alpha_{n}},
\end{equation*}
and this implies the right inequality in \eqref{618}.
\medskip

The relation  \eqref{613} is equivalent to
\begin{equation}\label{620}
\begin{split}
&\overline{\alpha_i\cdots\alpha_n}0^\infty
<\alpha_j\cdots\alpha_n^+\overline{\alpha_1\cdots\alpha_n}\ 0^{\infty}
\Longrightarrow\\
&\overline{\alpha_i\cdots\alpha_n\
\alpha_1\cdots\alpha_n^+}\
\alpha_1\cdots\alpha_n^+\ 0^{\infty}
<
\alpha_j\cdots\alpha_n^+\
 \overline{\alpha_1\cdots\alpha_n\
\alpha_1\cdots\alpha_n^+}\
\alpha_1\cdots\alpha_n^+\ 0^{\infty}.
\end{split}
\end{equation}
If $i\ge j$, then the first inequality of \eqref{620} implies that
\begin{equation*}
\overline{\alpha_i\cdots\alpha_n}\le\alpha_j\cdots\alpha_{n+j-i},
\end{equation*}
If this  inequality is strict, then the  right inequality in \eqref{620} follows. If we have an equality, then, \eqref{619} together with the equality  yields the desired result.

If $i<j$, then the first inequality of \eqref{620} implies that
\begin{equation*}
\overline{\alpha_i\cdots\alpha_{i+n-j}}\le\alpha_j\cdots\alpha_n^+.
\end{equation*}
If this inequality is strict, then the second inequality of \eqref{620} follows again.
We conclude the proof in this case by showing that we cannot have an equality. Indeed, in case of equality we would have the following implications
\begin{equation*}
\overline{\alpha_{1}\cdots\alpha_{j-i}}\geq\overline{\alpha_{n+i-j+1}\cdots\alpha_{n}}.
\end{equation*}
This is impossible because by $\beta(q_0)=\alpha_1\cdots\alpha_n^+\ 0^{\infty}$ we have
\begin{equation*}
\alpha_{n+i-j+1}\cdots\alpha_{n}<\alpha_{n+i-j+1}\cdots\alpha_{n}^+\le \alpha_{1}\cdots\alpha_{j-i},
\end{equation*}
and hence
\begin{equation*}
\overline{\alpha_{1}\cdots\alpha_{j-i}}<\overline{\alpha_{n+i-j+1}\cdots\alpha_{n}}.
\end{equation*}

The relation \eqref{614} is equivalent to
\begin{equation}\label{621}
\begin{split}
&\alpha_i\cdots\alpha_n^+\overline{\alpha_1\cdots\alpha_n}\ 0^{\infty}
<
\overline{\alpha_j\cdots\alpha_n}\ 0^{\infty}
\Longrightarrow\\
&\alpha_i\cdots\alpha_n^+\
 \overline{\alpha_1\cdots\alpha_n\
\alpha_1\cdots\alpha_n^+}\
\alpha_1\cdots\alpha_n^+\ 0^{\infty}
<
 \overline{\alpha_j\cdots\alpha_n\
\alpha_1\cdots\alpha_n^+}\
\alpha_1\cdots\alpha_n^+\ 0^{\infty}.
\end{split}
\end{equation}
The first inequality of \eqref{621} implies that
\begin{equation*}
\alpha_i\cdots\alpha_n^+\overline{\alpha_1\cdots\alpha_{i-1}}<\overline{\alpha_j\cdots\alpha_{n}}\ 0^{j-1}.
\end{equation*}
 This yields the desired result.
\medskip

The relation \eqref{615} is equivalent to
\begin{equation}\label{622}
\begin{split}
&\overline{\alpha_i\cdots\alpha_n}\ 0^{\infty}
<
\overline{\alpha_j\cdots\alpha_n}\ 0^{\infty}
\Longrightarrow\\
&\overline{\alpha_i\cdots\alpha_n\
\alpha_1\cdots\alpha_n^+}\
\alpha_1\cdots\alpha_n^+\ 0^{\infty}
<
 \overline{\alpha_j\cdots\alpha_n\
\alpha_1\cdots\alpha_n^+}\
\alpha_1\cdots\alpha_n^+\ 0^{\infty}.
\end{split}
\end{equation}

If $i>j$, then the left inequality in \eqref{622} implies that
\begin{equation*}
\overline{\alpha_i\cdots\alpha_n}\leq \overline{\alpha_j\cdots\alpha_{j-i+n}},
\end{equation*}
this together with \eqref{64} yields the right inequality of \eqref{622}.

If $i<j$, then the first inequality of \eqref{622} implies that
\begin{equation*}
 \overline{\alpha_i\cdots\alpha_{i-j+n}}<\overline{\alpha_j\cdots\alpha_{n}}.
\end{equation*}
We obtain  the right inequality of \eqref{622}.
\medskip

(ii) First we consider the case (a).
If $w=\eta_{\alpha_n^+}$, then $f(w)=\tilde a_n$, and it follows from Lemma \ref{l61} (ii) that
\begin{align*}
&\tilde w<\tilde a_n<\tilde a_{3n+i}<\tilde a_i=\tilde z
\qtq{if}0<i<n,\\
&\tilde w<\tilde a_n<\tilde a_{n+i}<\tilde a_i=\tilde z
\qtq{if}n<i\le 2n.
\end{align*}
Hence $(f(w),f(z))$ contains at least two vertices of $\gg(q_{m+1})$ by the construction of $\gg(q_m)$ and $\gg(q_{m+1})$.

If $w\ne\eta_{\alpha_n^+}$, then $f(w)=\tilde w$, and Lemma \ref{l61} (ii) implies that
\begin{align*}
&\tilde w<\tilde a_{3n+i}<\tilde a_i=\tilde z
\qtq{if}0<i<n,\\
&\tilde w<\tilde a_{n+i}<\tilde a_i=\tilde z
\qtq{if}n<i\le 2n.
\end{align*}
Hence $(f(w),f(z))$ contains at least two vertices of $\gg(q_{m+1})$ again.

In cases (b) and (c) we have $f(w)<f(z)$, and $(f(w),f(z))$ is not a switch region.
Therefore $(f(w),f(z))$  contains at least one vertex of $\gg(q_{m+1})$.

We have indicated altogether at least $2(2n-2)+M+1=4n+M-1$ vertices of $\gg(q_{m+1})$: their right endpoints run over
\begin{equation*}
\tilde a_1,\ldots, \tilde a_{2n-1},
\tilde a_{2n+1},\ldots, \tilde a_{4n-1},
\theta_1,\ldots,\theta_M,\eta_{M+1}
\end{equation*}
(we recall that $\tilde a_{2n}=\tilde \theta_{\overline{\alpha_n}}$).

Since $\gg(q_{m+1})$ has exactly $4n+M-1$ vertices, we conclude that $(f(w),f(z))$ contains exactly two vertices of $\gg(q_{m+1})$ in case (a), and $(f(w),f(z))$ is a vertex of $\gg(q_{m+1})$ in cases (b) and (c).
\medskip

(iii) We have to show the equivalence
\begin{equation*}
(x,x^+)\xrightarrow{k}(y,y^+)
\Longleftrightarrow
(f(x),f(x)^+)\xrightarrow{k}(f(y),f(y)^+)\end{equation*}
where $x, y$ run over the points \eqref{69}.

If
\begin{equation*}
(x,x^+)\xrightarrow{k}(y,y^+)
\Longleftrightarrow
(f(x),f(x)^+)\xrightarrow{j}(f(y),f(y)^+)\end{equation*}
for some labels $k,j$, then $k=j$ because \begin{equation*}
k=b_1(x,q_m)=b_1(f(x),q_{m+1})=j.
\end{equation*}
Henceforth we  write $T:=T_k$ for brevity.
We have to show that
\begin{equation*}
T(x)\le y\qtq{and}y^+\le T(x^+)
\Longleftrightarrow
T(f(x))\le f(y)\qtq{and}f(y)^+\le T(f(x)^+).
\end{equation*}
This is equivalent to
\begin{equation*}
T(x)\le y<T(x^+)
\Longleftrightarrow
T(f(x))\le f(y)<T(f(x)^+),
\end{equation*}
and then, since $f$ is increasing by (i), to
\begin{equation*}
f(T(x))\le f(y)<f(T(x^+))
\Longleftrightarrow
T(f(x))\le f(y)<T(f(x)^+).
\end{equation*}

First we prove that
\begin{equation}\label{623}
f(T(x))\le f(y)
\Longleftrightarrow
T(f(x))\le f(y).
\end{equation}
If $x\in\set{\eta_1, \ldots, \eta_M}\setminus\set{\eta_{\alpha_n^+}}$, then  \eqref{623} takes the form
\begin{equation*}
\tilde a_{n+1}\le f(y)
\Longleftrightarrow
\tilde a_{2n+1}\le f(y).
\end{equation*}
Since $(\tilde a_{2n+1},\tilde a_{n+1})$ is a vertex of  $\gg(q_{m+1})$ by  Lemma \ref{l62}, we have
\begin{equation*}
\tilde a_{n+1}\le f(y)
\Longleftrightarrow
\tilde a_{2n+1}<f(y).
\end{equation*}
It remains to observe that, since $(y, y^+)$ is a vertex of $\gg(q_m)$, $f(y)$ differs from $\tilde a_{2n+1}=T(f(x))$.

If $x=a_{2n-1}$, then
\eqref{623} takes the form
\begin{equation*}
\tilde\theta_{\overline{\alpha_n}}\le f(y)
\Longleftrightarrow
\tilde\eta_{\overline{\alpha_n}}\le f(y).
\end{equation*}
Since $(\tilde\theta_{\overline{\alpha_n}},\tilde\eta_{\overline{\alpha_n}})$ is a switch interval by  Lemma \ref{l62}, we have
\begin{equation*}
\tilde\theta_{\overline{\alpha_n}}<f(y)
\Longleftrightarrow
\tilde\eta_{\overline{\alpha_n}}\le f(y).
\end{equation*}
It remains to show that $f(y)$ differs from $\tilde\theta_{\overline{\alpha_n}}$.
This follows from \eqref{69} and \eqref{611} because
\begin{equation*}
\alpha_n^+\le M\Longrightarrow
\alpha_n<M\Longrightarrow
\overline{\alpha_n}>0.
\end{equation*}
In the remaining cases \eqref{623} is obvious  because $f(T(x))=T(f(x))$ by  Lemma \ref{l62}.

We finish the proof of (iii) by showing the equivalence
\begin{equation}\label{624}
f(y)<f(T(x^+))
\Longleftrightarrow
f(y)<T(f(x)^+)
\end{equation}
for the  points $x,y$ listed in \eqref{69}; then $x^+$ is one of the points listed in \eqref{610}.

If $x^+\in\set{\theta_1,\ldots,\theta_M,\eta_{M+1}}\setminus\set{\theta_{\overline{\alpha_n}}}$, then \eqref{624} holds because $T(f(x)^+)=f(T(x^+))$.
Indeed, $(f(x),f(x^+))$ is a vertex of $\gg(q_{m+1})$ by (ii), and hence $f(x)^+=f(x^+)$.
Since $T(f(x^+))=f(T(x^+))$ by Lemma \ref{l62}, we conclude that
\begin{equation*}
T(f(x)^+)=T(f(x^+))=f(T(x^+)).
\end{equation*}

It remains to consider the case where $x^+=a_i$ for some
\begin{equation*}
i\in\set{1,\ldots,n-1,n+1,\ldots,2n};
\end{equation*}
we recall that $a_{2n}=\theta_{\overline{\alpha_n}}$.
Applying (ii) we obtain that
\begin{equation*}
(T(f(x)^+),f(T(x^+)))=
\begin{cases}
(\tilde a_{3n+i+1},\tilde a_{i+1})
&\text{if }i=1,\ldots, n-1,\\
(\tilde a_{n+i+1},\tilde a_{i+1})
&\text{if }i=n+1,\ldots, 2n-1,\\
(\tilde a_{3n+1},\tilde a_1)
&\text{if }i=2n.
\end{cases}
\end{equation*}
It remains to show that
$f(y)$ cannot belong to $[T(f(x)^+),f(T(x^+)))$.
Since $(y,y^+)$ is a vertex of $\gg(q_m)$, and therefore
\begin{equation*}
f(y)\notin\set{\tilde a_{2n+1},\ldots,\tilde a_{4n}},
\end{equation*}
$f(y)\ne T(f(x)^+)$.
If $0<i<n-1$ or if $n<i<2n-1$, then we infer from the above formula and from Lemma \ref{l61} (ii) that $(T(f(x)^+),f(T(x^+)))$ is a vertex of $\gg(q_{m+1})$, and therefore $f(y)$ does not belong to this interval.
This proves \eqref{624} in these cases.

If $i=n-1$, then
\begin{equation*}
(T(f(x)^+),f(T(x^+)))=(\tilde a_{4n},\tilde a_n)
=(\tilde\theta_{\alpha_n^+},\tilde a_n),
\end{equation*}
and we infer from (ii) that $(\tilde\eta_{\alpha_n^+},\tilde a_n)$ is a vertex of $\gg(q_{m+1})$.
Since $(\tilde\theta_{\alpha_n^+},\tilde\eta_{\alpha_n^+})$ is a switch interval, we conclude by observing that $f(y)\ne \tilde\eta_{\alpha_n^+}$ by the definition of $f$.

Finally, if $i=2n-1$, then
\begin{equation*}
(T(f(x)^+),f(T(x^+)))=(\tilde a_{3n},\tilde a_{2n})
=(\tilde a_{3n},\tilde \eta_{\overline{\alpha_n}}),
\end{equation*}
and we infer from (ii) that $(\tilde a_{3n},\tilde \theta_{\overline{\alpha_n}})$ is a vertex of $\gg(q_{m+1})$.
Since $(\tilde \theta_{\overline{\alpha_n}},\tilde \eta_{\overline{\alpha_n}})$ is a switch interval, \eqref{624} will follow if we show that $f(y)$ is different from both $\tilde a_{3n}$ and $\tilde\theta_{\overline{\alpha_n}}$.
This follows from \eqref{611} if we  recall that $y$ is one of the points listed in \eqref{69}.

(iv) Applying Lemma \ref{l37} (iv) (with $N=4n$, so that $n$ is replaced by $2n$) and Lemma \ref{l61} (ii) we obtain that
$\hat\gg(q_m)$ and $\cc_{m+1}$ have the following vertices, respectively:
\begin{align*}
&(\tilde a_i,\tilde a_i^+)\qtq{for}0<i<2n,\\
&(\tilde\eta_j,\tilde\eta_j^+)\qtq{for}j\in\set{1,\ldots,M}\setminus\set{\alpha_n^+},\\
&(\tilde\theta_0,\tilde\theta_0^+),
\end{align*}
and
\begin{equation*}
(\tilde a_i,\tilde a_i^+)\qtq{for}2n<i<4n,\quad (\tilde\eta_{\alpha_n^+},\tilde\eta_{\alpha_n^+}^+).
\end{equation*}
Comparing the greedy expansions of the numbers $\tilde a_i$ and $\tilde\eta_{\alpha_n^+}$, we see that $\tilde\eta_{\alpha_n^+}=\tilde a_{2n}$ if $\alpha_n^+=\overline{\alpha_n}$, and $\tilde\eta_{\alpha_n^+}$ is different from all numbers $\tilde a_i$ if $\alpha_n^+\ne\overline{\alpha_n}$.
Therefore the two vertex sets are different.
Furthermore, since their total number $4n+M-1$ is equal to the number of vertices of $\gg(q_{m+1})$, they form a partition of the vertices of  $\gg(q_{m+1})$.
\medskip

(v) Since $n=2k$ is even and
\begin{equation*}
\alpha_1\cdots\alpha_n^+
=\alpha_1\cdots\alpha_k^+\ \overline{\alpha_1\cdots\alpha_k},
\end{equation*}
whence $\alpha_k^+=\overline{\alpha_n}$,  using the definition of the function $f$ we obtain that $\hat\cc_m$ and $\cc_{m+1}$ have the following vertices, respectively:
\begin{align*}
&(\tilde a_i,\tilde a_i^+)\qtq{for}n<i<2n,\quad (\tilde\eta_{\alpha_k^+},\tilde\eta_{\alpha_k^+}^+),
\intertext{and}
&(\tilde a_i,\tilde a_i^+)\qtq{for}2n<i<4n,\quad (\tilde\eta_{\alpha_n^+},\tilde\eta_{\alpha_n^+}^+).
\end{align*}
It follows that the path
\begin{equation*}
(\tilde a_{2n-1},\tilde a_{2n-1}^+)
\xrightarrow{\overline{\alpha_{n-1}}}
(\tilde a_{2n},\tilde a_{2n}^+)
\xrightarrow{\overline{\alpha_{n}}}(\tilde a_{2n+1}, \tilde a_{2n+1}^+)
\end{equation*}
of Lemma \ref{l61} (v) leads from $\hat\cc_m$ to $\cc_{m+1}$.
\end{proof}

\begin{proof}[Proof of Theorem \ref{t14}]
Applying Lemma \ref{l63} (iii) and identifying the graphs $\gg(q_m)$ and $\hat\gg(q_m)$ we may assume that $\gg(q_m)$ is a subgraph of $\gg(q_{m+1})$ for $m=1,2,\ldots.$
Let us denote by $V_1$ the set of vertices of $\gg(q_1)$, and by $V_m$ for $m\ge 2$ the set of vertices of $\gg(q_m)$ that are not vertices of the subgraph $\gg(q_{m-1})$.

Now we define the graph $\hat\gg(q_0^*)$ as follows.
Its vertices form the union of the sets $V_m$ for $m=1,2,\ldots .$
Furthermore, we draw an edge $I\xrightarrow{k}J$ in $\hat\gg(q_0^*)$ if there exists an $m$ such that $I\xrightarrow{k}J$ in $\gg(q_m)$; then by construction we also have $I\xrightarrow{k}J$ in $\gg(q_j)$ for all $j>m$ as well.

By Lemma \ref{l63} (iv) the subgraph spanned by $V_m$ is the cycle $\cc_m$ for every $m\ge 2$.
Furthermore, if $\beta(q_0)=\alpha_1\cdots\alpha_n^+\ 0^{\infty}$, then the last nonzero digit of $\beta(q_m)$ is the $2^mn$th digit by Proposition \ref{p25} (ii) for every $m\ge 1$, and the length of the cycle $\cc_m$ is equal to $2^{m-1}n$ by Lemma \ref{l61} (ii).
The rest of the theorem follows from Lemmas \ref{l61} (iii) and \ref{l63} (v).
\end{proof}

\section{Proof of Theorem \ref{t11} for $q\in\vv\setminus\uuu$}\label{s7}

The following lemma strengthens Proposition \ref{p313} (ii) for $q\in\vv\setminus\uuu$ because then $\alpha_n^+$ is not the last nonzero digit of $\beta(q)$.

\begin{lemma}\label{l71}
Let $q\in\vv\setminus\uuu$ and write $\alpha(q)=\left(\alpha_1\cdots\alpha_n^+\overline{\alpha_1\cdots\alpha_n^+}\right)^{\infty}$ with the smallest possible $n$.
If there exists a path
\begin{equation}\label{71}
I_1\xrightarrow{\alpha_1}
I_2\xrightarrow{\alpha_2}
\cdots
I_n\xrightarrow{\alpha_n^+}
I_{n+1}
\end{equation}
in $\gg(q)$, then $I_1\subset (a_1,\eta_{M+1})=(1,M/(q-1))$.
\end{lemma}

\begin{proof}
If $q=\min\vv$ is the generalized Golden Ratio and $M=2m-1$ or $M=2m$, then $n=1$ and $\alpha_n^+=m>0$.
By Example \ref{e310} there exists such a path $I_1\xrightarrow{m}I_2$ only if $m=M$, i.e., if $m=M=1$,  and in this case $I_1=(\eta_M,\eta_{M+1})=(a_1,\eta_{M+1})$.

Turning to the case $q>\min\vv$, with the notation of Proposition \ref{p28} we have $q=q_m$ for some $m\ge 1$, where $q_0\in\uuu\setminus\uu$.

If $m=1$, then $\beta(q_0)=\alpha_1\cdots\alpha_n^+\ 0^{\infty}$, and the lemma follows by applying  Proposition \ref{p313} (ii) for $q_0$, and using the isomorphism of $\gg(q_0)$ and $\gg(q_1)$ (Theorem \ref{t13}).

Proceeding by induction, we assume that the lemma is true for some $q_m$ with $m\ge 1$, and we prove its validity for $q_{m+1}$.
Writing
\begin{equation*}
\alpha(q_m)=\left(\alpha_1\cdots\alpha_n^+\overline{\alpha_1\cdots\alpha_n^+}\right)^{\infty}
\end{equation*}
we have
\begin{equation*}
\alpha(q_{m+1})=\left(\alpha_1\cdots\alpha_n^+\
\overline{\alpha_1\cdots\alpha_n\  \alpha_1\cdots\alpha_n^+}\ \alpha_1\cdots\alpha_n\right)^{\infty}.
\end{equation*}
By Lemma \ref{l63} the graph $\gg(q_{m+1})$ consists of a subgraph $\hat\gg(q_m)$ generating  $\gg_{q_m}'$,  a \emph{cycle} $\cc_{m+1}$ generating $\alpha(q_m)$,
and  there is no edge from $\cc_{m+1}$ to $\hat\gg(q_m)$.
Therefore a path of the form \eqref{71} in $\gg(q_{m+1})$ that generates a sequence starting with $\alpha_1\cdots\alpha_n^+\
\overline{\alpha_1\cdots\alpha_n}$  cannot start in the cycle $\cc_{m+1}$. Hence it must start in $\hat\gg(q_m)$, and the relation $I_1\subset (a_1,\eta_{M+1})$ follows by applying the induction hypothesis.
\end{proof}

Now we are ready to complete the proof of Theorem \ref{t11}.

\begin{proof}[Proof of Theorem \ref{t11} for $q\in\vv\setminus\uuu$]
Since now $\overline{\uu_q}=\uu_q$, the proof is simpler than in the earlier case $q\in\uuu\setminus\uu$.
The only new difficulty is to show that $(c_i)\in\gg_q'$ implies that $(c_i)_q\in\uu_q$.
Setting $x_i:=(c_ic_{i+1}\cdots)_q$ for all $i$, we have to show that they are  different from the points $\theta_i$ and $\eta_i$, $i=1,\cdots, M$.
By reflection it suffices to show that  $x_i\ne \theta_j$ for all $i$.
We recall from Lemma \ref{l34} (ii) and Proposition \ref{p210} (vii) that $\theta_j\in\vv_q$, and therefore $\theta_j$ has a unique doubly infinite expansion.

Assume on the contrary that $x_k=\theta_j$ for some $j$.
Then, since $j^-\alpha(q)$ and $(c_i)$ are doubly infinite, and $\theta_j$ has a unique doubly infinite expansion,
\begin{equation*}
c_{k}c_{k+1}\cdots=j^-\alpha(q),
\end{equation*}
and hence
\begin{equation*}\label{72}
\alpha(q)=c_{k+1}c_{k+2}\cdots\in\gg_q'.
\end{equation*}
Writing $\alpha(q)=\left(\alpha_1\cdots\alpha_n^+\overline{\alpha_1\cdots\alpha_n^+}\right)^{\infty}$ with the smallest possible $n$, hence we infer the
existence of a path of the form
\begin{equation*}
J_1\xrightarrow{\overline{\alpha_1}}
J_2\xrightarrow{\overline{\alpha_2}}
\cdots
J_n\xrightarrow{\overline{\alpha_n^+}}
J_{n+1}\xrightarrow{\alpha_1}
J_{n+2}\xrightarrow{\alpha_2}
\cdots
J_{2n}\xrightarrow{\alpha_n^+}
J_{2n+1}
\end{equation*}
in $\gg(q)$.
Applying Lemma \ref{l71} it follows that $J_{n+1}\subset (a_1,\eta_{M+1})$.
Applying Proposition \ref{p313} (i) hence we infer that
\begin{equation*}
J_1=\cdots=J_n=(\eta_M,\eta_{M+1}) \qtq{or} (1, \eta_{M+1}),
\end{equation*}
and $\overline{\alpha_1\cdots\alpha_n^+}=M^n$.
But this is impossible because $\overline{\alpha_n^+}<M$.
\end{proof}

\begin{corollary}\label{c72}
If $q\in\vv\setminus\uu$, then $\uu_q$, $\overline{\uu_q}$, $\vv_q$, $\gg_q$ and $\tilde\gg_q$
have the same Hausdorff dimension.
\end{corollary}

\begin{proof}
By Proposition \ref{p27} (i) and Theorem \ref{t11}, these sets $\uu_q$, $\overline{\uu_q}=\gg_q$, $\vv_q$ differ only by countable sets, hence they have the same Hausdorff dimension by Proposition \ref{p29}.
Since $\tilde\gg_q\subset\gg_q$ by the definition of a subgraph, the proof is completed by observing that $\gg_q$ has a countable cover by sets similar to $\tilde\gg_q$ by Proposition \ref{p313} (i).
\end{proof}

\section{Proof of Theorem \ref{t16}}\label{s8}

Since $j>1$, we have always $\uu_q^j\subset \left(0,\frac{M}{q-1}\right)$.
We will show that if $\uu_q^j\ne\varnothing$, then $0\in\overline{\uu_q^j}$, and hence  $\uu_q^j$ is not closed.

We use the generalized Golden Ratio $q_{GR}$, and we distinguish several cases:
\begin{enumerate}[\upshape (a)]
\item $q<q_{GR}$;
\item $\uu_q^j\cap (0,1)\ne\varnothing$;
\item $\uu_q^j\cap \left(\frac{M}{q-1}-1,\frac{M}{q-1}\right)\ne\varnothing$;
\item $M=2m$ is even, $q=m+1$ and $\uu_{m+1}^j=\set{1}$.
\end{enumerate}

(a) If $q<q_{GR}$, then every $x\in \left(0,\frac{M}{q-1}\right)$ has a continuum of expansions by \cite{EJK1990, B2014}.
Since $j>1$, the assumption $\uu_q^j\ne\varnothing$ implies that $j=2^{\aleph_0}$.
Then $\uu_q^j=\left(0,\frac{M}{q-1}\right)$ and therefore $0\in\overline{\uu_q^j}$ as required.

(b) If there exists a point $x\in\uu_q^j\cap (0,1)$, then $0\in \overline{\uu_q^j}$ again because $q^{-k}x\to 0$ as $k\to\infty$, and $q^{-k}x\in\uu_q^j$ for each positive integer $k$.
Indeed, if $(c_i)$ is an expansion  of $x$, then  $0^k(c_i)$ is an expansion of $q^{-k}x$, so that $q^{-k}x$ has at least $j$ expansions.
Conversely, since $q^{-k}x<q^{-k}$, each expansion $(d_i)$ of $q^{-k}x$ must start with $0^k$, and therefore $(d_{k+i})$ is an expansion of $x$.
This shows that $q^{-k}x$ has at most $j$ expansions.

(c) If there exists a point $y\in\uu_q^j\cap \left(\frac{M}{q-1}-1,\frac{M}{q-1}\right)$, then there exists also a point $x\in\uu_q^j\cap (0,1)$ because $y\in\uu_q^j\Longleftrightarrow \frac{M}{q-1}-y\in\uu_q^j$.

If $\uu_q^j\ne\varnothing$ and none of the conditions (a), (b) and (c) is satisfied, then the condition (d) is satisfied.
Indeed, we must have
\begin{equation*}
q\ge q_{GR}\qtq{and}\varnothing\ne\uu_q^j\subset\left[1,\frac{M}{q-1}-1\right].
\end{equation*}
This is impossible if $M=2m-1$ because \begin{equation*}
q\ge q_{GR}\Longrightarrow
\frac{M}{q-1}
\le\frac{M}{q_{GR}-1}
=\frac{4m-2}{m+\sqrt{m^2+4m}-2}
<\frac{4m-2}{m+(m+1)-2}=2,
\end{equation*}
so that the interval $\left[1,\frac{M}{q-1}-1\right]$ is empty.

This cannot happen either if $M=2m$ and $q>q_{GR}$ because
\begin{equation*}
q>q_{GR}\Longrightarrow
\frac{M}{q-1}
<\frac{M}{q_{GR}-1}
=\frac{2m}{(m+1)-1}=2,
\end{equation*}
and that the interval $\left[1,\frac{M}{q-1}-1\right]$ is empty again.

The only remaining possibility is $M=2m$ and $q=q_{GR}=m+1$.
Then the interval $\left[1,\frac{M}{q-1}-1\right]$ reduces to the one-point set $\set{1}$, i.e., to condition (d).

We complete the proof of the theorem by showing that the case (d) is impossible.
More precisely, we show that if
\begin{equation}\label{81}
M=2m,\quad
q=m+1\qtq{and}1\in\uu_q^j,
\end{equation}
then $\frac{1}{m+1}\in\uu_q^j$ and therefore $\uu_q^j\ne\set{1}$.

First we show that if $M=2m$ and $q=m+1$, then $1\in\uu_q^{\aleph_0}$, and therefore \eqref{81} implies that $j=\aleph_0$.
First of all, $m^{\infty}$ is an expansion of $1$ because
\begin{equation*}
\sum_{i=1}^{\infty}\frac{m}{(m+1)^i}=1.
\end{equation*}
Furthermore, for each $k\ge 0$ there exist exactly two  expansions $(c_i)$ of $1$ starting with exactly $k$ consecutive $m$ digits.
Indeed, if $c_{k+1}=m+1$, then
\begin{equation*}
\left(\sum_{i=1}^k\frac{m}{(m+1)^i}\right)+\frac{m+1}{(m+1)^{k+1}}=1,
\end{equation*}
so that we must have $c_i=0$ for all $i>k+1$.
This also shows that there are no expansions of $1$ with $c_{k+1}>m+1$.
On the other hand, if $c_{k+1}=m-1$, then
\begin{equation*}
\left(\sum_{i=1}^k\frac{m}{(m+1)^i}\right)+\frac{m-1}{(m+1)^{k+1}}
+\left(\sum_{i=k+2}^k\frac{2m}{(m+1)^i}\right)=1.
\end{equation*}
Since $2m$ is the maximal digit in the alphabet, hence we conclude that we must have $c_i=2m$ for all $i>k+1$, and that there are no expansions of $1$ with $c_{k+1}<m-1$.

It remains to show that $\frac{1}{m+1}\in\uu_q^{\aleph_0}$.
If $c_1c_2\cdots$ is an expansion of $1$, then $0c_1c_2\cdots$ is an expansion of $\frac{1}{m+1}$, so that $\frac{1}{m+1}$ has at least $\aleph_0$ expansions.
Conversely, if $d_1d_2\cdots$ is an expansion of $\frac{1}{m+1}$, then either $d_1=1$ and $d_i=0$ for all $i>1$, or $d_1=0$ and $d_2d_3\cdots$ is an expansion of $1$.
Hence $\frac{1}{m+1}$ has at most $1+\aleph_0=\aleph_0$ expansions.

\section{Preparation of the proofs of Theorems \ref{t18} and \ref{t110}}\label{s9}

Fix a positive integer $M$.
We recall that if $q\in\vv$, then the quasi-greedy expansion $\alpha(q)=\alpha_1\alpha_2\cdots$ of $x=1$ in base $q$ satisfies the lexicographic inequalities
\begin{equation}\label{91}
\alpha_{k+1}\alpha_{k+2}\cdots
\le \alpha(q)
\qtq{and}
\overline{\alpha_{k+1}\alpha_{k+2}\cdots}\le\alpha(q)
\end{equation}
hold for all $k\ge 0$.
If, moreover, $q\in \vv\setminus\uu$, then the sequence $\alpha(q)$ is periodic, and the last digit of the period is $<M$.

Henceforth we fix $q\in \vv\setminus\uu$, and we denote by $\alpha_1\cdots\alpha_N$ the shortest period of $\alpha(q)$.
Then $\alpha_1\cdots\alpha_{N-1}\alpha_N^+0^{\infty}$ is the greedy expansion of $x=1$ in base $q$.

Let $\ff'$ be a family of sequences (not necessarily  a subshift), and set
\begin{equation*}
\ff_q:=\set{(c_i)_q\ :\ (c_i)\in\ff'}.
\end{equation*}

\begin{lemma}\label{l91}
Assume that each $(c_i)\in\ff'$ satisfies the following lexicographic conditions:
\begin{align}
&c_{n+1}c_{n+2}\cdots<\alpha(q)\qtq{for}n= 0,\qtq{and whenever}c_n<M;\label{92}\\
&\overline{c_{n+1}c_{n+2}\cdots}<\alpha(q)\qtq{for}n= 0,\qtq{and whenever}c_n>0;\label{93}\\
&\alpha_{k+1}\cdots\alpha_N^+c_1c_2\cdots<\alpha(q)\qtq{whenever}1\le k<N\qtq{and}\alpha_k<M.\label{94}
\end{align}
Then
\begin{equation*}
\dim\uu_q^m\ge\dim\ff_q
\end{equation*}
for all $m\ge 1$.
\end{lemma}

\begin{proof}[Proof of Lemma \ref{l91}]
It suffices to show  for each $(c_i)\in\ff'$ the relation
\begin{equation*}
x_m:=(10^{(m-1)N}c_1c_2\cdots)_q\in\uu_q^m.
\end{equation*}
Indeed, then
\begin{equation*}
\frac{1}{q}+\frac{1}{q^{1+(m-1)N}}\ff_q\subset\uu_q^m,
\end{equation*}
i.e., $\uu_q^m$ contains a set similar to $\ff_q$.

For $m=1$ the claim follows by observing that the sequence $1c_1c_2\cdots$ satisfies the conditions of Corollary \ref{c22} by our assumptions \eqref{92} and \eqref{93}.

Since
\begin{equation*}
x_{m+1}=(10^{mN}c_1c_2\cdots)_q
=(0\alpha_1\cdots\alpha_N^+0^{(m-1)N}c_1c_2\cdots)_q
\end{equation*}
for all $m\ge 1$, we obtain by induction on $m$ that if $m\ge 2$, then $x_m$ has at least $m$ expansions: the given one starting with $1$, and the expansions
\begin{equation*}
0(\alpha_1\cdots\alpha_N)^j\alpha_1\cdots\alpha_N^+0^{(m-2-j)N}c_1c_2\cdots,\quad j=0,\ldots,m-2.
\end{equation*}

We prove by induction on $m$ that $x_m$ has no other expansions.
This is true for $m=1$.
Assuming that it is true for some $m\ge 1$, since
\begin{equation*}
x_{m+1}=(10^{mN}c_1c_2\cdots)_q
=(0\alpha_1\cdots\alpha_N^+0^{(m-1)N}c_1c_2\cdots)_q
\end{equation*}
and
\begin{equation*}
(10^{(m-1)N}c_1c_2\cdots)_q=x_m\in\uu_q^m,
\end{equation*}
it is sufficient to show that
\begin{equation*}
(0^{mN}c_1c_2\cdots)_q<1,
\end{equation*}
and that every expansion of
\begin{equation*}
(\alpha_1\cdots\alpha_N^+0^{(m-1)N}c_1c_2\cdots)_q=qx_{m+1}
\end{equation*}
starts necessarily with $\alpha_1\cdots\alpha_N^+$ or $\alpha_1\cdots\alpha_N$.
Indeed, the first property ensures that $x_{m+1}$ has no expansion starting with a digit $>1$, while the second condition ensures that if an expansion of $x_{m+1}$ starts with zero, then it has to start with $0\alpha_1\cdots\alpha_N^+$ or $0\alpha_1\cdots\alpha_N$.

The first condition follows by observing that $0^{mN}c_1c_2\cdots$ is a quasi-greedy expansion in base $q$ by \eqref{92} and Proposition \ref{p21} (ii).
Since
\begin{equation*}
0^{mN}c_1c_2\cdots<c_1c_2\cdots\le\alpha(q)
\end{equation*}
(the strict inequality holds because $c_1c_2\cdots\ne 0^{\infty}$ by \eqref{93}), we conclude that
\begin{equation*}
(0^{mN}c_1c_2\cdots)_q
<(\alpha(q))_q=1.
\end{equation*}

For the second condition it suffices to show that
\begin{equation*}
(b_i):=\alpha_1\cdots\alpha_N^+0^{(m-1)N}c_1c_2\cdots
\end{equation*}
and
\begin{equation*}
(d_i):=(\alpha_1\cdots\alpha_N)^{m-1}\alpha_1\cdots\alpha_N^+c_1c_2\cdots
\end{equation*}
are the greedy and lazy, i.e, the lexicographically largest and smallest expansions of $qx_{m+1}$, respectively.

Applying Proposition \ref{p21} (i) it suffices to show that
\begin{equation*}
b_{n+1}b_{n+2}\cdots<\alpha(q)\qtq{whenever}b_n<M
\end{equation*}
and
\begin{equation*}
\overline{d_{n+1}d_{n+2}\cdots}<\alpha(q)\qtq{whenever}\overline{d_n}<M.
\end{equation*}

The conditions on $(b_i)$ are satisfied by the assumptions \eqref{92} and \eqref{94} of the lemma.
The conditions on $(d_i)$ are satisfied for $n\ge mN$ by \eqref{93}.
For $n<mN$ they follow from the equality $\alpha(q)=(\alpha_1\cdots\alpha_N)^{\infty}$  and from \eqref{91} because
\begin{equation*}
(d_i)=(\alpha_1\cdots\alpha_N)^{m-1}\alpha_1\cdots\alpha_N^+c_1c_2\cdots
\end{equation*}
with $\alpha_1\cdots\alpha_N<\alpha_1\cdots\alpha_N^+$, and therefore
\begin{equation*}
\overline{d_{n+1}d_{n+2}\cdots}
<\overline{\alpha_{n+1}\alpha_{n+2}\cdots}
\le\alpha(q).
\end{equation*}
In the last step we used the assumption $q\in\vv$.
\end{proof}

The conclusion of Lemma \ref{l91} remains valid if we only assume the weak inequalities in of \eqref{92}, \eqref{93} and \eqref{94}:

\begin{corollary}\label{c92}
Assume that each $(c_i)\in\ff'$ satisfies the following lexicographic conditions:
\begin{align}
&c_{n+1}c_{n+2}\cdots\le\alpha(q)\qtq{for}n=0,\qtq{and whenever}c_n<M;\label{95}\\
&\overline{c_{n+1}c_{n+2}\cdots}\le\alpha(q)\qtq{for}n= 0,\qtq{and whenever}c_n>0;\label{96}\\
&\alpha_{k+1}\cdots\alpha_N^+c_1c_2\cdots\le\alpha(q)\qtq{whenever}1\le k<N\qtq{and}\alpha_k<M.\label{97}
\end{align}
Then
\begin{equation*}
\dim\uu_q^m\ge\dim\ff_q
\end{equation*}
for all $m\ge 1$.
\end{corollary}

\begin{proof}
By weakening the conditions \eqref{92}, \eqref{93}, \eqref{94} to \eqref{95}, \eqref{96}, \eqref{97} the family $\ff'$ may be increased only by countably many new sequences, so that $\dim\ff_q$ remains the same.
\end{proof}

In some cases the condition \eqref{97} of Corollary \ref{c92} may be replaced by a simpler one:

\begin{lemma}\label{l93}
Assume that
\begin{equation}\label{98}
c_1\cdots c_k\le\overline{\alpha_1\cdots\alpha_k}
\end{equation}
and
\begin{equation}\label{99}
c_{k+1}c_{k+2}\cdots\le\alpha(q)
\end{equation}
whenever $1\le k<N$ and $\alpha_k<M$.
Then the assumption \eqref{97} of Corollary \ref{c92} is satisfied.
\end{lemma}

\begin{proof}
Since $\alpha(q)=(\alpha_1\cdots\alpha_N)^{\infty}$, it suffices to show that if $1\le k<N$ and $\alpha_k<M$, then
\begin{equation*}
\alpha_{k+1}\cdots\alpha_N^+c_1\cdots c_k
\le\alpha_1\cdots\alpha_N.
\end{equation*}
Indeed,  using \eqref{99}  we will then obtain that
\begin{equation*}
\alpha_{k+1}\cdots\alpha_N^+c_1c_2\cdots
\le(\alpha_1\cdots\alpha_N)\alpha(q)=\alpha(q).
\end{equation*}
In view of \eqref{98} it suffices to prove that
\begin{equation*}
\alpha_{k+1}\cdots\alpha_N^+\overline{\alpha_1\cdots\alpha_k}
\le \alpha_1\cdots\alpha_N,
\end{equation*}
or equivalently that
\begin{equation}\label{910}
\alpha_{k+1}\cdots\alpha_N^+\overline{\alpha_1\cdots\alpha_k}
<\alpha_1\cdots\alpha_N^+.
\end{equation}

Since  $\alpha_1\cdots\alpha_N^+$ is the greedy expansion of a number $\le 1$, we have
\begin{equation*}
\alpha_{k+1}\cdots\alpha_N^+\le \alpha_1\cdots\alpha_{N-k}.
\end{equation*}
If this inequality is strict, then \eqref{910} follows.
Otherwise we have to show that $\overline{\alpha_1}<\alpha_N^+$ if $k=1$, and
\begin{equation*}
\overline{\alpha_1\cdots\alpha_k}<\alpha_{N-k+1}\cdots\alpha_N^+\qtq{if}1<k<N.
\end{equation*}
or equivalently (by taking reflections) that $\overline{\alpha_N^+}<\alpha_1$ if $k=1$, and
\begin{equation*}
\overline{\alpha_{N-k+1}\cdots\alpha_N^+}<\alpha_1\cdots\alpha_k\qtq{if}1<k<N.
\end{equation*}
This follows by applying Proposition \ref{p25} (v) for $q^+:=\min\set{p\in\vv\ :\ p>q}$ because

\begin{equation*}
\alpha(q^+)=\left(\alpha_1\cdots\alpha_N^+\overline{\alpha_1\cdots\alpha_N^+}\right)^{\infty}
\end{equation*}
by Proposition \ref{p25} (ii).
\end{proof}

Now fix an arbitrary base $p_L\in\vv\setminus\uu$, and we denote by $a_1\cdots a_m$ the shortest period of $\alpha(p_L)$. Then
\begin{equation*}
\beta(p_L)=a_1\cdots a_m^+\ 0^{\infty}.
\end{equation*}
Define the bases $r_0, r_1,\ldots$ and $p_R$ by the formulas
\begin{align*}
&\left(a_1\cdots a_m^+\left(\overline{a_1\cdots a_m}\right)^j0^{\infty}\right)_{r_j}=1,\quad j=0,1,\ldots
\intertext{and}
&\left(a_1\cdots a_m^+\left(\overline{a_1\cdots a_m}\right)^{\infty}\right)_{p_R}=1.
\end{align*}

\begin{lemma}\label{l94}
We have $\alpha(q_R)=a_1\cdots a_m^+(\overline{a_1 \cdots a_m})^\infty$
and
\begin{equation*}
p_L=r_0<r_1<\cdots<p_R.
\end{equation*}
Furthermore, $p_R\in\uu$, $r_j\in\uuu\setminus\uu$ for all $j\ge 2$, and
$r_1\in\vv\setminus\uuu.$
\end{lemma}

\begin{proof} The following inequalities \eqref{911} and  $p_R\in\uu$  are from  \cite[Lemmas 4.1, 4.10]{ABBK2019}, respectively. For the reader's convenience, we give the proof here. First we prove
\begin{equation}\label{911}
\overline{a_1\cdots a_{m-i}}\leq a_{i+1}\cdots a_m<a_{i+1}\cdots a_m^+\leq a_1\cdots a_{m-i} \qtq{for all}  1\le i<m.
\end{equation}
The second and the third inequalities follows from
$\beta(p_L)=a_1\cdots a_m^+$ implying that
$$a_{i+1}\cdots a_m<a_{i+1}\cdots a_m^+\leq a_1\cdots a_{m-i} \qtq{for all}  1\le i<m.$$
The first inequality is obtained by $q_L\in\vv\setminus\uu$ .
So, by  \eqref{911} and Proposition \ref{p24}, we obtain  $\alpha(q_R)=a_1\cdots a_m^+(\overline{a_1 \cdots a_m})^\infty$ and $q_R\in\uu$.
The relation \begin{equation*}
p_L=r_0<r_1<\cdots<p_R.
\end{equation*}  and $r_1\in \vv\setminus \uuu$ are clear.

 It remains to prove $r_j\in\uuu\setminus\uu$ for all $j\ge 2$.  Let $$(c_i)=(a_1\cdots a_m^+\left(\overline{a_1\cdots a_m}\right)^{j-1}\overline{a_1\cdots a_m^+})^\infty.$$ We need to prove $$(\overline{a_1\cdots a_m^+}\left(a_1\cdots a_m\right)^{j-1}a_1\cdots a_m^+
  )^\infty<c_{i+1}c_{i+2}\cdots \leq (a_1\cdots a_m^+\left(\overline{a_1\cdots a_m}\right)^{j-1}\overline{a_1\cdots a_m^+})^\infty$$ for all $i\geq 0$.

  We only  prove the left inequality. The right inequality can be proved similarly. Since $(c_i)$ is periodic, it suffices to prove  the left inequality holds for $0\leq i<(j+1)m$.
  For $i=0$, it follows from $\overline{a_1}<a_1$, see Proposition \ref{p25} (iii). For $1\le i<m$, it follows from $\overline{a_1\cdots a_{m-i}}<a_{i+1}\cdots a_m^+$  of \eqref{911}. For $i=km, k\in\{1, \cdots, j-1\}$, it is clear. For $km<i<(k+1)m$, it follows from $a_{i+1}\cdots a_{m}< a_1\cdots a_{m-i}$ of \eqref{911}. For $i=jm$, it is clear. For $jm<i<(j+1)m$, it follows from the third inequalities and $\overline{a_1\cdots a_{m-i}}< a_{i+1}\cdots a_m^+$.
\end{proof}

In the remainder of this section we write, as usual,
\begin{equation}\label{912}
w:=a_1\cdots a_m,\quad
w^+:=a_1\cdots a_m^+,\quad
\overline{w}=\overline{a_1\cdots a_m},\quad
\overline{w^+}=\overline{a_1\cdots a_m^+}.
\end{equation}

\begin{lemma}\label{l95}
Let us denote by $\ff'$ the set of sequences $(c_i)$ starting with $\overline{a_1\cdots a_m}$ and satisfying the following conditions:
\begin{align}
&c_{n+1}c_{n+2}\cdots\le\alpha(p_L)\qtq{for} n=0 \qtq{and  whenever}c_n<M;\label{913}\\
&\overline{c_{n+1}c_{n+2}\cdots}\le\alpha(p_L)
\qtq{for} n=0 \qtq{and  whenever}c_n>0.\label{914}
\end{align}
Fix $j\ge 0$ arbitrarily, and write $\beta(r_j)=\alpha_1\cdots\alpha_N^+\ 0^{\infty}$.
Then we have
\begin{equation}\label{915}
\alpha_{k+1}\cdots\alpha_N^+c_1c_2\cdots\le\alpha(p_L)\qtq{whenever}1\le k<N\qtq{and}\alpha_k<M.
\end{equation}
\end{lemma}
\noindent In \eqref{915} and in the sequel $\alpha_{k+1}\cdots\alpha_N^+$ means $\alpha_N^+$ when $k=N-1$.
\begin{proof}
Using the notation in \eqref{912}, we write
\begin{equation}\label{916}
\alpha(p_L)=w^{\infty}
\qtq{and}
\alpha(p_R)=w^+\ \overline{w}^{\infty}.
\end{equation}
First we deduce from \eqref{913} and \eqref{914} the seemingly stronger relations
\begin{equation}\label{917}
\overline{\alpha(p_L)}\le
c_{n+1}c_{n+2}\cdots\le\alpha(p_L)\qtq{for all}n\ge 0.
\end{equation}
By symmetry we only prove the second inequality.
Since $\alpha(p_L)<M^{\infty}$, by \eqref{913} there exist infinitely many integers $k\ge 1$ with $c_k<M$.
Therefore, if $c_n=M$ for some $n\ge 1$, then there exist two integers $k$ and $\ell$ such that $0\le k<n\le\ell$,
\begin{equation*}
(k=0\text{ or }c_k<M),\quad
c_{k+1}=c_{k+2}=\cdots=c_{\ell}=M\qtq{and}
c_{\ell+1}<M.
\end{equation*}
Then
\begin{equation*}
c_{n+1}\cdots c_{\ell+1}<M^{\ell-n+1}=c_{k+1}\cdots c_{\ell+k-n+1},
\end{equation*}
and therefore
\begin{equation*}
c_{n+1}c_{n+2}\cdots<c_{k+1}c_{k+2}\cdots
\le\alpha(p_L)
\end{equation*}
by \eqref{913}.

It follows from \eqref{916} and \eqref{917} that
\begin{equation}\label{918}
\overline{w^j}\ c_{k+1}c_{k+2}\cdots\\
\le c_{k+1}c_{k+2}\cdots
\end{equation}
for all integers $j,k\ge 0$.
Indeed, if we do not have an equality here, then $j\ge 1$, and $c_{k+1}c_{k+2}\cdots>\overline{\alpha(p_L)}=\overline{w^{\infty}}$ by \eqref{916} and \eqref{917}.
Therefore there exists an integer $n\ge 0$ and a word $u>\overline{w}$ of length $m$ such that $c_{k+1}c_{k+2}\cdots$ starts with $\overline{w^n}\ u$.
Our claim follows by observing that $\overline{w^j}\ c_{k+1}c_{k+2}\cdots$ starts with $\overline{w^{n+1}}<\overline{w^n}\ u$.

If $q=r_0=p_L$, then \eqref{915} follows by applying Lemma \ref{l93} because its assumption \eqref{99} is satisfied by \eqref{917}.

If $q=r_j$ with some $j\ge 1$, then $N=(j+1)m$.
Let $1\le k<N$ be such that $\alpha_k<M$.
If $1\le k\le m$, then using \eqref{918} we get
\begin{equation*}
\alpha_{k+1}\cdots\alpha_N^+c_1c_2\cdots
=a_{k+1}\cdots a_m^+ \overline{w^j}c_1c_2\cdots
\le a_{k+1}\cdots a_m^+ c_1c_2\cdots
\le \alpha(p_L);
\end{equation*}
the last inequality follows from the previous case $q=p_L$.

If $im+1\le k\le (i+1)m$ for some $i=1,\ldots, j$, then writing $\ell:=k-im$ and using the equality $c_1\cdots c_m=\overline{w}$ and \eqref{918} we get
\begin{align*}
\alpha_{k+1}\cdots\alpha_N^+c_1c_2\cdots
&=\overline{a_{\ell+1}\cdots a_m} \ (\overline{w})^{j-i}c_1c_2\cdots\\
&=\overline{a_{\ell+1}\cdots a_m} \ (\overline{w})^{j-i+1}c_{m+1}c_{m+2}\cdots\\
&\le \overline{a_{\ell+1}\cdots a_m} \ c_{m+1}c_{m+2}\cdots\\
&=c_{\ell+1}c_{\ell+2}\cdots
\le \alpha(p_L);
\end{align*}
the last inequality holds by \eqref{917}.
\end{proof}

Next we  show that  Lemma \ref{l95} does not hold for any base $q\in (\vv\setminus\uu)\cap [p_L,p_R)$ different from the numbers $r_k$.
In the following two results we describe the quasi-greedy and the greedy expansions  $\alpha(q)$ and $\beta(q)$ of the numbers  $q\in (\vv\setminus\uu)\cap (p_L,p_R)$.

\begin{proposition}\label{p96}
If $q\in (\vv\setminus\uu)\cap (p_L,p_R)$, then there exists a periodic sequence $k_1, k_2,\ldots$ of nonnegative integers such that
\begin{equation*}
\alpha(q)=
w^+\left( \overline{{w^{k_1}w^+}}\right)
\left( w^{k_2}w^+\right)
\left(\overline{{w^{k_3}w^+}}\right) \left(w^{k_4}w^+\right)
\cdots,
\end{equation*}
and $k_j\le k_1$ for all $j\ge 2$.
\end{proposition}

\begin{proof}
Since $\beta(p_L)=w^+0^{\infty}$ and $p_L<q\le p_R$, $\alpha(q)$ starts with $w^+$.
Furthermore, since $p_R\in\uu$ by Lemma \ref{l94} and hence $q\ne p_R$, we have $\alpha(q)<\alpha(p_R)=w^+(\overline{w})^{\infty}$.
Therefore there exists a maximal integer $k_1\ge 0$ and a word $u<\overline{w}$ of length $m$ such that $\alpha(q)$ starts with $w^+\overline{w^{k_1}}u$.

Applying \eqref{91} it follows that $u\ge \overline{w^+}$.
Therefore we have $u=\overline{w^+}$, and  $\alpha(q)$ starts with $w^+\left(\overline{w^{k_1}w^+}\right)$.

Now there exists an integer $k_2'\ge 0$ and a word $u$ of length $m$ such that $\alpha(q)$ starts with $w^+\left(\overline{w^{k_1}w^+}\right)w^{k_2'}u$.
Applying \eqref{91} we see that $u\le w^+$ and
$\overline{w^+}w^{k_2'}u\ge \overline{w^+}w^{k_1}w^+$.
Hence $k_2'\le k_1$, and
\begin{equation*}
u=
\begin{cases}
w^+&\text{if $k_1=k_2'$,}\\
w^+\text{ or }w&\text{if $k_1>k_2'$.}
\end{cases}
\end{equation*}
There exists therefore an integer $k_2$ satisfying $0\le k_2\le k_1$, and such that $\alpha(q)$ starts with $w^+\left(\overline{w^{k_1}w^+}\right) \left(w^{k_2}w^+\right)$.

Continuing by induction, assume for some positive integer $j$ that $\alpha(q)$ starts with
\begin{equation*}
w^+\left(\overline{w^{k_1}w^+}\right) \left(w^{k_2}w^+\right)
\cdots\left(\overline{w^{k_{2j-1}}w^+}\right) \left(w^{k_{2j}}w^+\right).
\end{equation*}
Then there exists an integer $k_{2j+1}'\ge 0$ and a word $u$ of length $m$ such that $\alpha(q)$ starts with
\begin{equation*}
w^+\left(\overline{w^{k_1}w^+}\right) \left(w^{k_2}w^+\right)
\cdots\left(\overline{w^{k_{2j-1}}w^+}\right) \left(w^{k_{2j}}w^+\right)\left(\overline{w^{k_{2j+1}'}}\right)u.
\end{equation*}
Applying \eqref{91} we obtain the lexicographic inequalities
\begin{equation*}
u\ge \overline{w^+}\qtq{and}w^+\left(\overline{w^{k_{2j+1}'}}\right)u\le w^+\left(\overline{w^{k_1}}\right)\overline{w^+}.
\end{equation*}
Hence $k_{2j+1}'\le k_1$, and
\begin{equation*}
u=
\begin{cases}
\overline{w^+}&\text{if $k_{2j+1}'=k_1$,}\\
\overline{w^+}\text{ or }\overline{w}&\text{if $k_{2j+1}'<k_2$.}
\end{cases}
\end{equation*}
There exists therefore an integer $k_{2j+1}$ satisfying $0\le k_{2j+1}\le k_1$ and such that $\alpha(q)$ starts with
\begin{equation*}
w^+\left(\overline{w^{k_1}w^+}\right) \left(w^{k_2}w^+\right)
\cdots\left(\overline{w^{k_{2j-1}}w^+}\right) \left(w^{k_{2j}}w^+\right)\left(\overline{w^{k_{2j+1}}w^+}\right).
\end{equation*}

Next, there exists an integer $k_{2j+2}'\ge 0$ and a word $u$ of length $m$ such that $\alpha(q)$ starts with
\begin{equation*}
w^+\left(\overline{w^{k_1}w^+}\right) \left(w^{k_2}w^+\right)
\cdots\left(\overline{w^{k_{2j-1}}w^+}\right) \left(w^{k_{2j}}w^+\right)\left(\overline{w^{k_{2j+1}}w^+}\right)
\left(w^{k_{2j+2}'}\right)u.
\end{equation*}
Applying \eqref{91} we obtain the lexicographic inequalities
\begin{equation*}
u\le w^+\qtq{and}\overline{w^+}\left(w^{k_{2j+2}'}u\right)
\ge \overline{w^+}\left(w^{k_1}w^+\right).
\end{equation*}
Hence $k_{2j+2}'\le k_1$, and
\begin{equation*}
u=
\begin{cases}
w^+&\text{if $k_{2j+2}'=k_1$,}\\
w^+\text{ or }w&\text{if $k_{2j+2}'<k_2$.}
\end{cases}
\end{equation*}
There exists therefore an integer $k_{2j+2}$ satisfying $0\le k_{2j+2}\le k_1$ and such that $\alpha(q)$ starts with
\begin{equation*}
w^+\left(\overline{w^{k_1}w^+}\right) \left(w^{k_2}w^+\right)
\cdots\left(\overline{w^{k_{2j-1}}w^+}\right) \left(w^{k_{2j}}w^+\right)\left(\overline{w^{k_{2j+1}}w^+}\right)
\left(w^{k_{2j+2}}w^+\right).
\end{equation*}
We complete the proof of the proposition by recalling that $\alpha(q)$ is periodic for every $q\in\vv\setminus\uu$; therefore the sequence $(k_j)$ is also periodic.
\end{proof}

\begin{corollary}\label{c97}
If $q\in (\vv\setminus\uu)\cap [p_L,p_R)$ and $\beta(q)=b_1\cdots b_N^+0^{\infty}$, then
$b_1\cdots b_N^+$ ends with $w^+\overline{w^{\ell}}$ for some $\ell\ge 0$.
More precisely, $b_1\cdots b_N^+=w^+$; otherwise there exist a positive integer $j$ and nonnegative integers $k_1,\ldots, k_{2j}$ such that
\begin{equation*}
b_1\cdots b_N^+=
\begin{cases}
w^+\left( \overline{{w^{k_1}w^+}}\right)
\left( w^{k_2}w^+\right)\cdots
w^+\left(\overline{{w^{k_{2j-1}}w^+}}\right) w^{k_{2j}-1}w^+
&\text{if $k_{2j}\ge 1$,}\\
w^+\left( \overline{{w^{k_1}w^+}}\right)
\left( w^{k_2}w^+\right)\cdots
w^+\left(\overline{{w^{k_{2j-1}+1}}}\right)
&\text{if $k_{2j}=0$.}
\end{cases}
\end{equation*}
\end{corollary}

\begin{proof}
The case $q=p_L$ is obvious.
Otherwise it follows from the structure of $\alpha(q)$ in Proposition \ref{p96} that its period has to end just before a word $w^+$, so it is  of the form
\begin{equation*}
w^+\left( \overline{{w^{k_1}w^+}}\right)
\left( w^{k_2}w^+\right)\cdots
w^+\left(\overline{{w^{k_{2j-1}}w^+}}\right) w^{k_{2j}}
\end{equation*}
for some $j\ge 1$.
This implies our claim.
\end{proof}

Now we can show that  Lemma \ref{l95} does not hold for any base $q\in (\vv\setminus\uu)\cap [p_L,p_R)$ different from the numbers $r_k$:

\begin{corollary}\label{c98}
Let $q\in (\vv\setminus\uu)\cap [p_L,p_R)$
with $\beta(q)=b_1\cdots b_N^+0^{\infty}$.
Let us denote by $\ff'$ the set of sequences $(c_i)$ starting with $\overline{w}=\overline{a_1\cdots a_m}$ and satisfying the following conditions:
\begin{align*}
&c_{n+1}c_{n+2}\cdots\le\alpha(p_L)\qtq{for} n=0 \qtq{and  whenever}c_n<M;\\
&\overline{c_{n+1}c_{n+2}\cdots}\le\alpha(p_L)
\qtq{for} n=0 \qtq{and  whenever}c_n>0.\end{align*}
If $\beta(q)$ is not of the form $a_1\cdots a_m^+\left(\overline{a_1\cdots a_m}\right)^{\ell}0^{\infty}$ for any $\ell\ge 0$, then the condition \eqref{915} of Lemma \ref{l95} is not satisfied.
\end{corollary}

\begin{proof}
We infer from Proposition \ref{p96} and Corollary \ref{c97} that there exists a $k$ satisfying $1\le k<N$ and $b_k<M$ such that $b_{k+1}\cdots b_N^+=w^+\overline{w^{\ell}}$ for some $\ell\ge 0$.
Then we have
\begin{equation*}
b_{k+1}\cdots b_N^+c_1c_2\cdots
=w^+\overline{w^{\ell}}c_1c_2\cdots
\ge w^+\overline{w^{\ell}\alpha(p_L)}
=\alpha(p_R)
>\alpha(q).\qedhere
\end{equation*}
\end{proof}

We end this section by deducing from Proposition \ref{p96} a property that will be used in the next section for the proof of Theorem \ref{t110} .
We need a lemma.

\begin{lemma}\label{l99}
Consider the interval  $[p_L,p_R]$ with an  arbitrary $p_L\in\vv\setminus\uu$, and pick an arbitrary $q_L\in [p_L,p_R]\cap(\vv\setminus\uu)$.
Write
\begin{equation*}
\beta(q_L)=b_1\cdots b_n^+\ 0^{\infty}
\end{equation*}
and define a base $q_R$ by the usual formula
\begin{equation*}
\beta(q_R)=\alpha(q_R)=b_1\cdots b_n^+\left(\overline{b_1\cdots b_n}\right)^{\infty}.
\end{equation*}
Then $[q_L,q_R]\subseteq[p_L,p_R]$.
\end{lemma}

\begin{proof}
We  have $q_L<p_R$ because $p_R\in\uu$ by Lemma \ref{l94}.
If $q_L=p_L$, then $q_R=p_R$, and we have an equality.
Henceforth we assume that $p_L<q_L<p_R$.
Then, writing $w:=a_1\cdots a_m$ for brevity, by Corollary \ref{c97} we have either $\beta(q_L)=w^+\overline{w^{k_1+1}}\ 0^{\infty}$, or $\beta(q_L)$ starts with the word $w^+ \overline{{w^{k_1}w^+}}$.
In the first case $\beta(q_R)$ starts with
\begin{equation*}
w^+\overline{w^{k_1+1}w^+}<w^+\overline{w^{k_1+2}},
\end{equation*}
while in the second case $\beta(q_R)$ starts with
\begin{equation*}
w^+\overline{w^{k_1}w^+}<w^+\overline{w^{k_1+1}}.
\end{equation*}
Since $\beta(p_R)=w^+\overline{w^{\infty}}$, we have $\beta(q_R)<\beta(p_R)$ in both cases, and therefore $q_R<p_R$ by Proposition \ref{p23} (i).
\end{proof}
In the proof of the following proposition we use an important theorem of Alcaraz Barrera et al. \cite{ABBK2019}.

\begin{proposition}\label{p910}
The topological entropy $h(\vv_q')$ is constant in $[q_L,q_R]$ for every $q_L\in\vv\setminus\uu$.
\end{proposition}

\begin{proof}
Since  $h(\vv_q')=0$ for all
$1<q\le q_{KL}$ by \cite[Lemma 2.2]{KKL2017}, we may assume that $q_L\ge q_{KL}$.
Then there exists a $q_0\in\uuu\setminus\uu$ such that $q_L\in [q_0,q_0^*]$, and $q_L<q_0^*$ because $q_0^*\in\uu$. Here $(q_0,q_0^*)$  denotes the  arbitrary connected component of $(1,M+1)\setminus\uuu$, see Proposition \ref{p28}.

We know from \cite[Lemma 2.11]{KKL2017} that  $[q_0,q_0^*]$ is a stability interval for the topological entropy, and hence it is contained in a maximal stability interval.
Since the maximal stability intervals in $[q_{KL},M+1)$ are of the form $[p_L,p_R]$ with a suitable $p_L\in\uuu\setminus\uu$ by  \cite[Theorem 2]{ABBK2019}, we have $q_L\in[q_0,q_0^*]\subset[p_L,p_R]$ for some  $p_L\in\uuu\setminus\uu$.
Applying Lemma \ref{l99} hence we infer that $[q_L,q_R]\subseteq[p_L,p_R]$.
Since the topological entropy $h(\vv_q')$ is constant in $[p_L,p_R]$, it  is also constant in the subinterval $[q_L,q_R]$.
\end{proof}

\section{Proofs of Theorems \ref{t18} and \ref{t110}}\label{s10}
First we prove Theorems \ref{t18} and \ref{t110} for $q=q_0$.

\begin{proof}[Proof of Theorem \ref{t18}   for $q=q_0$]
It follows from Lemma \ref{l41} that  $\tilde\gg(q_0)$ generates the words $\alpha_1\cdots\alpha_N$ and $\overline{\alpha_1\cdots\alpha_N}$.
Let $\ff'$ be the set of sequences $(c_i)$ in $\tilde\gg_{q_0}'$ that start with $\overline{\alpha_1\cdots\alpha_N}$, and consider an arbitrary $(c_i)\in\ff'$. Then $$(c_i)\in\ff'\subseteq \tilde\gg_{q_0}'=\tilde \vv_{q_0}'.$$
The last equality is from Corollary \ref{c43}.
Since $c_1\cdots c_N=\overline{\alpha_1\cdots\alpha_N}$,
 the condition \eqref{98} of  Lemma \ref{l93} is automatically satisfied.
By Lemma \ref{l42} the condition \eqref{99} of Lemma \ref{l93} and the conditions \eqref{95} and \eqref{96} are also satisfied (see \eqref{41}).
Therefore, in view of Lemma \ref{l93} we may apply Corollary \ref{c92} to conclude that  $\dim\uu_{q_0}^m\ge\dim\ff_{q_0}$ for all $m\ge 1$.
Since $\dim\uu_{q_0}^m\le\dim\uu_{q_0}$, it remains to show that $\dim\uu_{q_0}\le \dim\ff_{q_0}$.

Since $\tilde \gg({q_0})$ is strongly connected by assumption,  $\tilde\gg_{q_0}$ may be covered by a finite number of sets, each similar to $\ff_{q_0}$, so that $\dim\tilde\gg_{q_0}\le \dim\ff_{q_0}$.
The theorem follows because $\dim\uu_{q_0}=\dim\tilde\gg_{q_0}$ by Corollary \ref{c72}.
\end{proof}

For the proof of Theorem \ref{t110} we need another preliminary result:

\begin{lemma}\label{l101}
Assume  $q_0\in\uuu\setminus\uu$ and $\beta(q_0)=\alpha_1\cdots\alpha_N^+$. Then the subgraph $\tilde\gg_1(q_0)$ formed by the vertices of the form $(a_i^-, a_i)$ and $(b_i, b_i^+)$ for $i=1,\ldots, N$
 is strongly connected. Furthermore $\tilde\gg_1(q_0)$ contains the words $\alpha_1\cdots\alpha_N$ and $\overline{\alpha_1\cdots\alpha_N}$.
\end{lemma}

\begin{proof}
It follows from Lemma \ref{l41} that the vertices of $\tilde\gg_1(q_0)$ are formed of the vertices of two cycles of $N$ vertices that generate the words $\alpha_1\cdots\alpha_N$ and $\overline{\alpha_1\cdots\alpha_N}$, respectively.
It remains to show that there is a path from some vertex of the first cycle to some vertex of the second cycle, and there is a path from some vertex of the second cycle to some vertex of the first cycle.
By reflection it suffices to consider the first case.

The vertices of the first cycle are of the form $(a_k, a_i)$, $(b_n, a_i)$, $(\eta_j, a_i)$ with $i=1,\ldots,N$ and $k,n\in\set{1,\ldots,N}$ and $j\in\set{1,\ldots,M}$.
Observe that $(a_N,a_N^+)=(\theta_{\alpha_N^+},\eta_{\alpha_N^+})$ is a switch interval and hence is not a vertex of $\gg(q_0)$.
Furthermore, $(a_1,a_1^+)$ is not a vertex of $\tilde\gg(q_0)$ by definition, and hence it is not a vertex of the subgraph $\tilde\gg_1(q_0)$ either.

Thus we have at most $N-2$ vertices of the form $(a_k, a_i)$ in $\tilde\gg_1(q_0)$, and therefore the first cycle of $\tilde\gg_1(q_0)$ contains at least one vertex of the form $(b_n, a_i)$ or $(\eta_j, a_i)$.
In the first case $(b_n, a_i)$ is a common vertex of the two cycles.
In the second case we have
\begin{equation*}
(\eta_j, a_i)\xrightarrow{j}(b_1, b_1^+)
\end{equation*}
by Lemma \ref{l33} (iii); since $(b_1, b_1^+)$ is a vertex of the second cycle, this is an edge from the first cycle to the second one.
\end{proof}

\begin{proof}[Proof of Theorem \ref{t110} for $q=q_0$]
Applying Lemma \ref{l101}, let $\ff'$ be the set of sequences in $\tilde\gg_{1,q_0}'$ starting with $\overline{\alpha_1\cdots\alpha_N}$.
Since $\tilde\gg_{1,q_0}'\subseteq\tilde\gg_{q_0}'$, we may repeat the preceding proof to obtain $\dim\uu_{q_0}^m\ge\dim\ff_{q_0}$ for all $m\ge 1$, and it remains to show that $\dim\uu_{q_0}\le \dim\ff_{q_0}$.
Since $\dim\tilde\gg_{1, q_0}=\dim\uu_{q_0}$ by our assumption, it is sufficient to show that $\dim\tilde \gg_{1,q_0}\le \dim\ff_{q_0}$.
This follows similarly to the proof of Theorem \ref{t18} for $q=q_0$ above because   Lemma \ref{l101} implies that $\tilde\gg_{1,q_0}$ may be covered by a finite number of sets, each similar to $\ff_{q_0}$.
\end{proof}

\begin{proof}[Proof of Theorem \ref{t18} for $q=r_k$ with $k\ge 1$]
Let us consider $\ff'$ as defined in Lemma \ref{l95} with $p_L:=p<q$.
Then $\ff'\subseteq \tilde \gg_{p}'\subseteq \tilde \gg_{q}'$.
It follows from Lemma \ref{l95} and  Corollary \ref{c92} that
\begin{equation*}
\dim\uu_q^m\ge\dim\ff_q\qtq{for all}m\ge 1.
\end{equation*}

Since $\dim\uu_q^m\le\dim\uu^1_q\le \dim\vv_q$, we end the proof by showing that $\dim\vv_q=\dim\ff_q$.
This follows from the equalities
\begin{equation}\label{101}
\dim\vv_q
=\frac{h(\vv_q')}{\log q}
=\frac{h(\vv_p')}{\log q}
=\frac{h(\uuu_p')}{\log q}
=\frac{h(\tilde\gg_p')}{\log q}=\frac{h(\ff')}{\log q}
=\dim\ff_q.
\end{equation}
The first and last equalities are from Proposition \ref{p29}.
The second equality follows from Lemma \ref{l94} and Proposition \ref{p910}.
The third equality is  from   Proposition \ref{p29} (i).

The  fourth and fifth equalities are equivalent to
\begin{equation*}
\frac{h(\uuu_p')}{\log p}
=\frac{h(\tilde\gg_p')}{\log p}=\frac{h(\ff')}{\log p}
\qtq{or equivalently to}
\dim \uuu_p=\dim \tilde\gg_p=\dim \ff_p.
\end{equation*}
Here the first one follows from Corollary \ref{c72}.
The second one follows from the fact that $\ff_p\subset\gg_p$, and that $\tilde\gg_p$ may be covered by finitely many sets similar to $\ff_p$ by the strong connectedness of $\tilde \gg(p)$.
\end{proof}

\begin{proof}[Proof of Theorem \ref{t110} for $q=r_k$ with $k\ge 1$]
As in the proof for $q=q_0=r_0$, let $\ff'$ be the set of sequences in $\tilde\gg_{1,q_0}'$ starting with $\overline{\alpha_1\cdots\alpha_N}$.
Since $\tilde\gg_{1,q_0}'\subseteq\tilde\gg_{q_0}'$, we may repeat the preceding proof to obtain $\dim\uu_{q_0}^m\ge\dim\ff_{q_0}$ for all $m\ge 1$, and it remains to show that $\dim\uu_{q_0}\le \dim\ff_{q_0}$.
This follows by observing that the equalities \eqref{101} hold again with $p:=q_0$ and $q:=r_k$ if we replace $\tilde\gg_p'$ by $\tilde\gg_{1,p}'$.
There are only two changes in the justification of these equalities: we have
$h(\vv_q')=h(\vv_p')$, i.e., $h(\vv_{r_k}')=h(\vv_{q_0}')$ by Proposition \ref{p910}, and we have $h(\uuu_p')=h(\tilde\gg_{1,p}')$ by the assumption of the theorem.
\end{proof}

\begin{remark}\label{r1002}
Our proofs of Theorems \ref{t18} and \ref{t110} do not remain valid for $q\in\vv\setminus\uuu$ because if $q\in\vv\setminus\uuu$ and $\beta(q)=w\overline{w^-}0^{\infty}$, then $\tilde\gg_q'$, and even the larger set $\gg_q'$, does not contain any sequence starting with $\overline{w}w$.
Indeed, assume on the contrary that there exists a sequence $(c_i)\in\gg_q'$ starting with $\overline{w}w$.
Then, since $\gg_q'\subset\vv_q'$ and $\alpha(q)=(w\overline{w})^{\infty}$, applying \cite[Lemma 3.4 (a)]{DKL2016} we obtain that $(c_i)=(\overline{w}w)^{\infty}$, and therefore $(c_i)\notin\uu_q'$ by Corollary \ref{c22}.
This is a contradiction because $\uu_q'=\gg_q'$ by Theorem \ref{t11} (ii).
\end{remark}

\section{An example related to the structure of univoque graphs}\label{s11}

In this section we will discuss an example
related to Theorems \ref{t11}, \ref{t13} and \ref{t14}.  This allows us to explain Proposition \ref{p28} (iii)-(v) in a
transparent way, instead of the former combinatorial and less intuitive arguments.
\begin{example}\label{e1101}
Let $M=3$ and
\begin{equation*}
\begin{split}
&\beta(q_0)=331\ 0^{\infty}\\
&\beta(q_1)=331\  003\ 0^{\infty}\\
&\beta(q_2)=331\  003\  002\ 331\ 0^{\infty}\\
&\beta(q_3)=331\  003\  002\ 331\ 002\ 330\ 331\ 003\  0^{\infty}
\end{split}
\end{equation*}
We have $q_0\in\uuu\setminus\uu$ and $q_1, q_2, q_3\in\vv\setminus\uuu$.
For $q_0$, by the usual arguments, we have the following relations
\begin{equation*}
\theta_0<b_1<b_2<\theta_1=a_3<\eta_1<\theta_2<\eta_2<\theta_3<\eta_3=b_3<a_2<a_1<\eta_{4}.
\end{equation*}

Figure 11.1 (a) shows that $\tilde \gg(q_0)$ is strongly connected.
(This will also follow from Proposition \ref{p125} in the next section.) Furthermore, by Theorem \ref{t11} (i) we have
\begin{equation*}
\gg_{q_0}'=\vv_{q_0}'.
\end{equation*}

It follows from Theorem \ref{t13} that $\gg(q_0)$ is isomorphic to $\gg(q_1)$, and by Theorem \ref{t11} (ii) we have
$\gg_{q_1}'=\uu_{q_1}'$.
Thus we have
\begin{equation*}
\gg_{q_1}'=\uu_{q_1}'=\vv_{q_0}'=\gg_{q_0}'.
\end{equation*}
This is illustrated by Figures 11.1 (a) and (b).

Next we discuss the graph $\gg(q_2)$, shown in Figure 11.2.
We infer from Lemmas \ref{l61} (ii)-(iii) and \ref{l63} (iii)-(iv) the following properties, that may be read off from the figures:

\begin{itemize}
\item $\hat{\gg}(q_2)$ is isomorphic to $\gg(q_1)$; more precisely,
$\gg_{q_1}'=\uu_{q_1}'=\hat{\gg}_{q_2}'$.
\item $\gg(q_2)$ contains a cycle $\cc_2$ with $6$ vertices, and all sequences generated by $\cc_2$ end with $(331002)^\infty$.
\item The vertices of  the subgraph $\hat{\gg}(q_2)$ and of the cycle $\cc_2$  form a partition of the vertices of $\gg(q_2)$.
\item There exist edges from $\hat{\gg}(q_2)$ to $\cc_2$, but there is no  edge from $\cc_2$ to $\hat{\gg}(q_2)$.
Therefore
\begin{equation*}
\vv_{q_1}'=\gg_{q_2}'=\uu_{q_2}'.
\end{equation*}
\end{itemize}

Finally, we investigate the graph $\gg(q_3)$ shown in Figure 11.3.
Lemmas \ref{l61} (ii)-(iii) and \ref{l63} (iii)-(iv) imply the following properties that may be verified on the figures:
\begin{itemize}
\item $\hat{\gg}(q_3)$ is isomorphic to $\gg(q_2)$; more precisely,
$\gg_{q_2}'=\uu_{q_2}'=\hat{\gg}_{q_3}'$.
\item $\gg(q_3)$ contains a cycle $\cc_3$ with $12$ vertices, and all the sequences generated by $\cc_3$ end with $(331003002330)^\infty$.
\item The vertices of  the subgraph $\hat{\gg}(q_3)$ and of the cycle $\cc_3$  form a partition of the vertices of $\gg(q_3)$.
\item There exist edges from $\hat{\gg}(q_3)$ to $\cc_3$, but there is no  edge from $\cc_3$ to $\hat{\gg}(q_3)$.
Therefore
\begin{equation*}
\vv_{q_2}'=\gg_{q_3'}=\uu_{q_3}'.
\end{equation*}
\end{itemize}
\end{example}

\begin{figure}[ht]\label{f1101}
\centering
\includegraphics[scale=0.4]{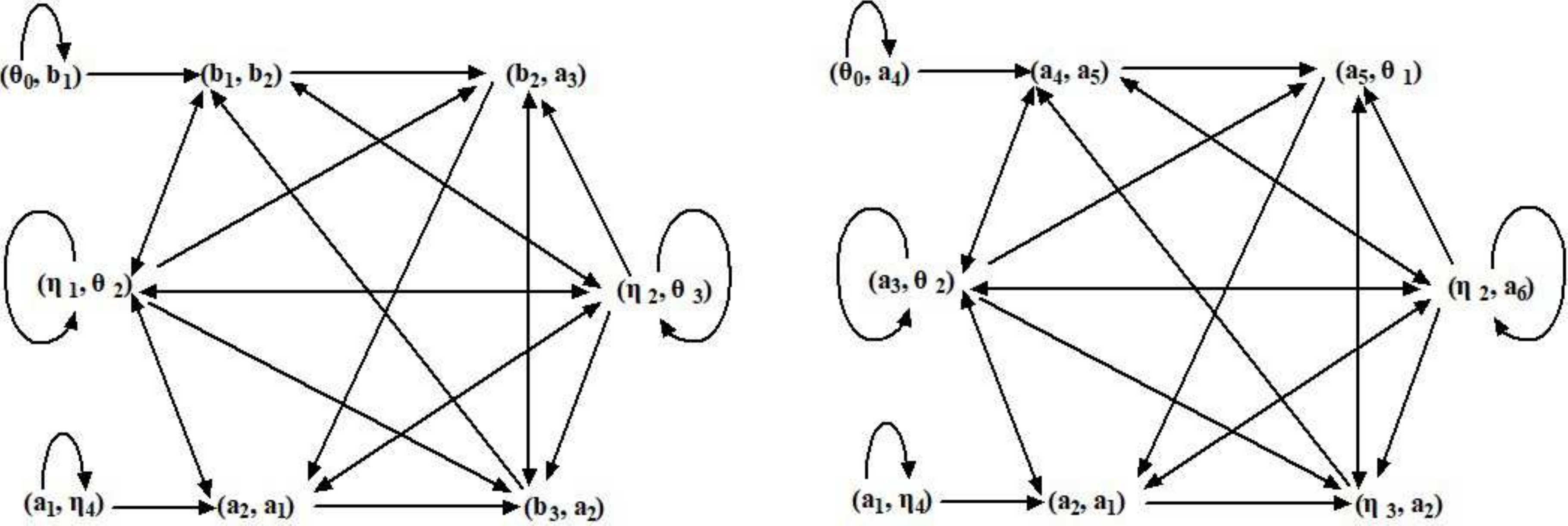}
\caption{(a) The left graph $\gg(q_0)$ associated with $\beta(q_0)=331\ 0^{\infty}$.
(b) The right graph $\gg(q_1)$ associated with $\beta(q_1)=331\ 003\ 0^{\infty}$.}
\end{figure}

\begin{figure}[ht]\label{f1102}
\centering
\includegraphics[scale=0.9]{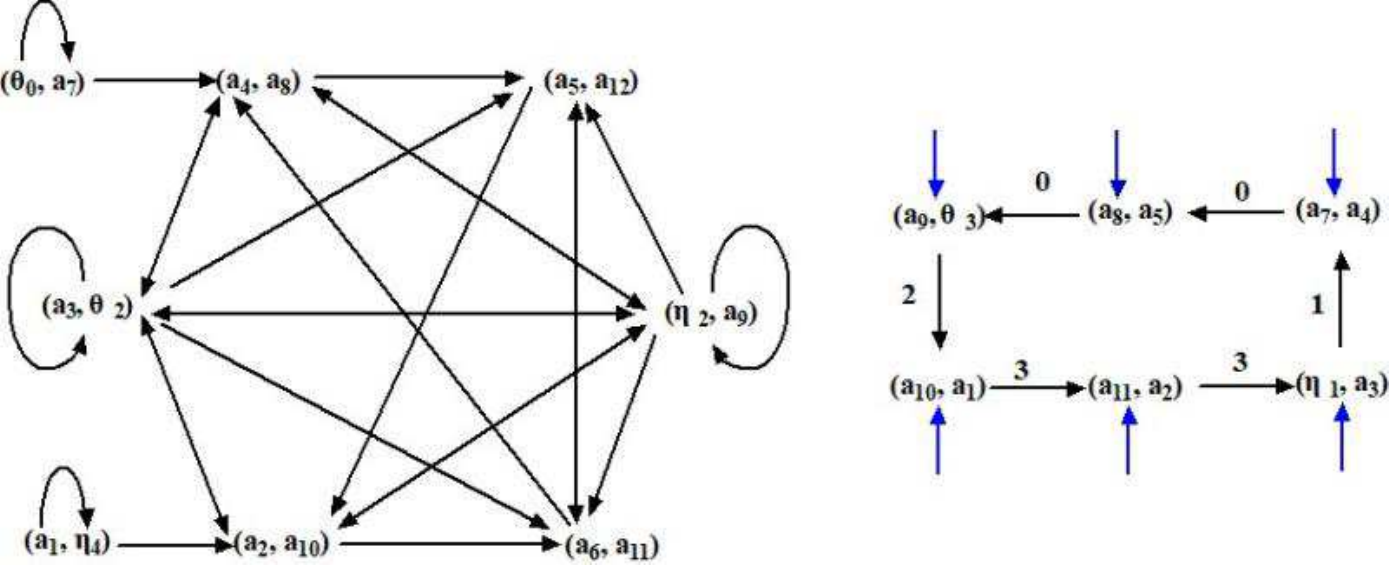}
\caption{The graph $\gg(q_2)$ associated with $$\beta(q_2)=331\ 003\ 002\ 331\ 0^{\infty}.$$
The blue arrows denote the edges from $\hat{\gg}(q_2)$ to $\cc_2$.}
\end{figure}

\begin{figure}[ht]\label{f1103}
%\centering
\includegraphics[scale=0.45]{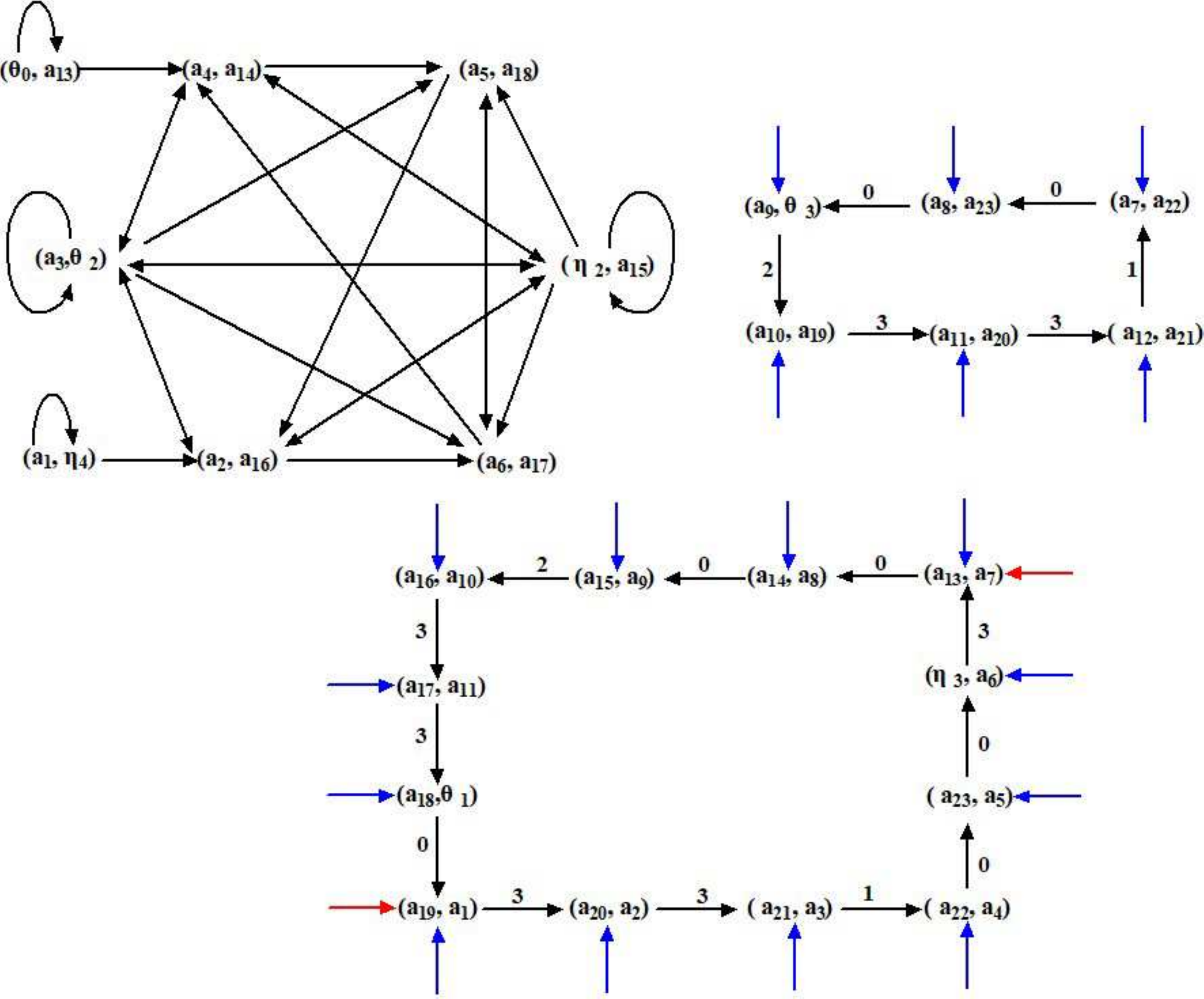}
\caption{The graph $\gg(q_3)$ associated with $$\beta(q_3)=331\ 003\ 002\ 331\ 002\ 330\ 331\ 003\ 0^{\infty}.$$
The blue arrows denote the edges from $\hat{\gg}(q_1)$ to $\cc_2$ and  $\cc_3$, respectively.
The red arrows denote the edges from the image $\hat\cc_2$ of $\cc_2$ to $\cc_3$. }
\end{figure}

\section{Comments and examples related to Theorems \ref{t18} and \ref{t110}}\label{s12}

\subsection{Strong connectedness of $\tilde\gg(q)$}\label{ss1201}

In connection with Theorems \ref{t18} and \ref{t110} we investigate the strong connectedness of $\tilde\gg(q)$ for $q\in \uuu\setminus\uu$ in view of Theorems \ref{t13} and \ref{t14}  in this section.

Let's recall some notation  in Sections 4 and  6.
We denote by $\tilde\gg(q)$ the subgraph that is  obtained from $\gg(q)$ by keeping only the vertices that are subintervals of $(b_1,a_1)$ and $\tilde\gg_1(q)$ the subgraph of $\tilde\gg(q)$,  formed by the vertices of the form $(a_i^-, a_i)$ and $(b_i, b_i^+)$.
Set $\tilde\gg'_q$ is the set of sequences generated by $\tilde\gg(q)$ and  $\tilde \gg_{q}:=\set{(c_i)_q: (c_i)\in \tilde \gg_{q}'}$.
$\tilde \gg_{1,q}', \tilde \gg_{1,q}$  have the same meaning as $\tilde \gg_{q}', \tilde \gg_{q}$.

There may be five different types of vertices of $\tilde \gg(q)$ as follows:
\begin{equation*}
(a_i^-, a_i),\quad
(b_i, b_i^+),\quad
(a_i, b_j),\quad
(\theta_i^-,\theta_i) \qtq{and}
(\eta_i, \eta_i^+).
\end{equation*}

The following result allows us to decide in many cases whether Theorem \ref{t18} may be applied.

\begin{proposition}\label{p121}
Let $q\in\uuu\setminus\uu$.
The graph $\tilde\gg(q)$ is strongly connected if and only if there exists from
$\tilde\gg_1(q)$ a path to each vertex of the form of  $(a_i, b_j)$ and $(\theta_i^-, \theta_i) $.
\end{proposition}

\begin{proof}
By Lemma \ref{l101} $\tilde\gg(q)$ contains a strongly connected subgraph $\tilde\gg_1(q)$, formed by the vertices of the form $(a_i^-, a_i)$ and $(b_i, b_i^+)$.
The vertices outside $\tilde\gg_1(q)$ are of the form of $(a_i, b_j)$, $(\theta_i^-, \theta_i)$ or $(\eta_i, \eta_i^+)$.
The first path of  Lemma \ref{l41} shows that there is a path from $(a_i, b_j)$ to $\tilde\gg_1(q)$.
Next, the relation $(\theta_i^-, \theta_i)\xrightarrow{i-1}(a_1^-, a_1)$ of Lemma \ref{l33} (ii) shows that there is a path from $(\theta_i^-, \theta_i)$ to the graph $\tilde\gg_1(q)$.
By reflection, there is also a path from $(\eta_i, \eta_i^+)$ to $\tilde\gg_1(q)$.

Therefore, $\tilde\gg(q)$ is strongly connected if and only if, conversely, there exists a path from $\tilde\gg_1(q)$ to each vertex of the form of  $(a_i, b_j)$, $(\theta_i^-, \theta_i)$ and $(\eta_i, \eta_i^+)$.
By reflection, we only need to consider   $(\theta_i^-, \theta_i) $.
\end{proof}

\begin{remark}\label{r122}
A different and deep necessary and sufficient condition follows from some recent results of Alcaraz Barrera, Baker and Kong \cite{ABBK2019}.
They characterized the transitivity of a special subshift $\tilde\vv_q'$ (see \eqref{41}).
\begin{itemize}
\item
If $q\in\uuu\setminus\uu$,
then $\tilde\vv_q'=\tilde\gg_q'$ by Corollary \ref{c43}, and therefore the transitivity of $\tilde\vv_q'$ is equivalent to the strong connectedness of $\tilde\gg(q)$.

\item If $q\in\vv\setminus\uuu$, then $\tilde\gg_q'$ is a proper subset of $\tilde\vv_q'$,  and $\tilde\vv_q'$ is not transitive by \cite[Lemma 3.16]{ABBK2019}, while $\tilde\gg_q'$ is transitive for some $q$.
More precisely, using the notations of Theorems \ref{t13} and \ref{t14} for any fixed $q_0\in\uuu\setminus\uu$, $\tilde\gg_{q_1}'=\tilde\gg_{q_0}'$ is transitive, while $\tilde\vv_{q_1}'=\tilde\gg_{q_2}'$ is not transitive because $\tilde\gg(q_m)$ is not strongly connected for any $m\ge 2$ by Theorem \ref{t14}; see, e.g., Figure 1.3 or Example \ref{e1101} and Figure 11.2 below.
\end{itemize}
\end{remark}

\begin{corollary}\label{c123}
Let $M=1$ and $q\in\uuu\setminus\uu$.
The graph $\tilde\gg(q)$ is strongly connected if and only if there exists from
$\tilde\gg_1(q)$ a path to each vertex of the form of $(a_i, b_j)$.

In particular, if $\tilde\gg(q)$ does not contain any vertex of the form of $(a_i, b_j)$, then $\tilde\gg(q)$ is strongly connected.
\end{corollary}

\begin{proof}
In this case $\tilde\gg(q)$  contains only one vertex of the form of $(\theta_i^-, \theta_i)$, namely  $(\theta_1^-, \theta_1)=(a_N^-,a_N)$, and it belongs to $\tilde\gg(q)$.
Therefore, in applying Proposition \ref{p121} to check the strong connectedness of $\tilde\gg(q)$, it suffices to consider the vertices of the form of $(a_i, b_j)$.
\end{proof}

\begin{example}\label{e124}
We recall that the \emph{Multinacci numbers} are defined for $M=1$ by $\beta(q):=1^N0^{\infty}$, $N\ge 3$.
In particular, for $M=1$ and $N=3$ we get the \emph{Tribonacci number}.

We consider, more generally, the numbers defined by $\beta(q):=M^N\ 0^{\infty}$ for any $M\ge 1$ and $N\ge 3$.
By Lemma \ref{l34} (i) we have
\begin{equation*}
a_i=(\alpha_i\cdots\alpha_N^+)_q<
(\alpha_{i-1}\cdots\alpha_N^+
)_q=a_{i-1}
\qtq{and hence} b_{i-1}<b_i
\end{equation*}
for $i=2,\ldots,N$.
Using Proposition \ref{p21} (i) now  we get
\begin{equation}\label{121}
\theta_0<b_1<\cdots<b_{N-1}<a_N=\theta_1<b_N=\eta_1<a_{N-1}<\cdots<a_1<\eta_{M+1}
\end{equation}
if $M=1$, and
\begin{multline}\label{122}
\theta_0<b_1<\cdots<b_{N-1}<\theta_1<b_N=\eta_1<\theta_2<\eta_2<\cdots<\eta_{M-1}\\
<\theta_M=a_N<\eta_M<a_{N-1}<\cdots<a_1<\eta_{M+1}
\end{multline}
if $M>1$.

It follows from  \eqref{121} that $\tilde\gg(q)$ does not contain any vertex of the form of $(a_i,b_j)$.
Therefore  $\tilde\gg(q)$ is strongly connected by Corollary  \ref{c123}.

Next we show that $\tilde\gg(q)$ is strongly connected for $M>1$, too.
By  \eqref{122},  $\tilde\gg(q)$  does not contain any vertex of the form of $(a_i,b_j)$.
By Proposition \ref{p121} it remains to show that there exists a path from $\tilde \gg_1(q)$ to each vertex $(\theta_i^-, \theta_i)$.
This follows by deducing from \eqref{122} that
\begin{itemize}
\item $(b_{N-1}, \theta_1)$ is a vertex of  $\tilde\gg_1(q)$,
\item $T_0((b_{N-1}, \theta_1))=(b_N, a_1)$,
\item $(\theta_i^-, \theta_i)\subset (b_N, a_1)$ for $i=2,\ldots,M$.
\end{itemize}
\end{example}

We have also  a simple \emph{sufficient} condition  for the strong connectedness $\tilde \gg(q)$:

\begin{proposition}\label{p125}
If $q\in\uuu\setminus\uu$ and
\begin{equation}\label{123}
b_2<\min\set{a_i\ :\ 1<i< N},
\end{equation}
then $\tilde\gg(q)$ is strongly connected.
\end{proposition}

\begin{proof}
By Proposition  \ref{p121}, it suffices to show that there exists a path from $\tilde\gg_1(q)$ to every vertex of $\tilde\gg(q)$.
We distinguish two cases:
\begin{enumerate}[\upshape (a)]
\item There exists a  $k\ge 1$ such that $ b_1<\theta_k<b_2$;
\item  There exists a $k\ge 1$ such that $\theta_{k-1}<b_1<b_2<\theta_k$.
(We recall that $\theta_0=0$.)
\end{enumerate}

In Case (a) let $\ell\ge 1$ be the least integer satisfying $b_1<\theta_\ell$.
Then by our assumption \eqref{123} every vertex of $\tilde\gg(q)$, contained in $(b_1, \theta_\ell)$ is also a vertex of $\tilde\gg_1(q)$.
Since
\begin{equation*}
T_{\ell-1}((b_1, \theta_\ell))=(b_2, a_1),
\end{equation*}
it follows that there exists a path from $\tilde\gg_1(q)$ to every vertex of $\tilde\gg(q)$, contained in $(b_2, a_1)$.

By reflection, there exists a path from $\tilde\gg_1(q)$ to every vertex of $\tilde\gg(q)$, contained in $(b_1, a_2)$.

We conclude by observing that
\begin{equation*}
(b_1, a_1)=(b_1, a_2)\cup (b_2,a_1)
\end{equation*}
by \eqref{123}, and that every vertex of $\tilde\gg(q)$ is contained in $(b_1, a_1)$.
\medskip

In case (b) there exists a largest integer $i\ge 2$  such that
\begin{equation*}
\theta_{k-1}<b_1<b_2<\cdots<b_{i-1}<b_i<\theta_k,
\end{equation*}
and we have the following relations:
\begin{equation}\label{124}
\begin{split}
&T_{k-1}((b_1,b_2))=(b_2, b_3),\ldots, T_{k-1}((b_{i-1},b_i))=(b_i, b_{i+1})\supset(b_i,\theta_k)\\
&T_{k-1}((b_i, \theta_k))=(b_{i+1}, a_1).
\end{split}
\end{equation}
By our assumption \eqref{123} each vertex of $\tilde\gg(q)$, contained in  $(b_1, b_2)$ is also a vertex of  $\tilde\gg_1(q)$.
Using \eqref{124} it follows that there exists a path from $\tilde\gg_1(q)$ to every vertex of $\tilde\gg(q)$, contained in $(b_2,a_1)$.
By reflection there exists also a path from $\tilde\gg_1(q)$ to every vertex of $\tilde\gg(q)$, contained in $(b_2, a_1)$, and we conclude as in the preceding case.
\end{proof}

\begin{example}\label{e126}
Let $M=1$, $N=5$ and $\beta(q)=11011\ 0^\infty$. Then $q\in \uuu\setminus\uu$, and the greedy expansions of $a_1,\ldots,a_5$ and $b_1,\ldots,b_5$ are
\begin{align*}
&11011\ 0^{\infty},\quad
1011\ 0^{\infty},\quad
011\ 0^{\infty},\quad
11\ 0^{\infty},\quad
1\  0^{\infty}
\intertext{and}
&00101(00101)^\infty,\quad
0101(00101)^\infty,\quad
101(00101)^\infty,\quad
01(00101)^\infty,\quad
1(00101)^\infty,
\end{align*}
respectively.
Applying Proposition \ref{p21} (i)
we obtain the relations
\begin{equation*}
\theta_0<b_1<b_4<b_2<a_3<a_5=\theta_1<b_5=\eta_1<b_3<a_2<a_4<a_1<\eta_2.
\end{equation*}
The graph $\tilde\gg(q)$ is strongly connected by Corollary  \ref{c123}.

%Note that the word $11010$ is irreducible.
\end{example}

\begin{example}\label{e127}
Let $M=4$, $N=4$ and $\beta(q)=4331\ 0^\infty$. Then $q\in \uuu\setminus\uu$. By Lemmas \ref{l34} (i) and \ref{l35} (ii) the greedy expansions of $a_1,\ldots,a_{4}$ and $b_1,\ldots,b_{4}$ are
\begin{equation*}
4331\ 0^\infty,\ 331\ 0^\infty,\ 31\ 0^\infty,\ 1\ 0^\infty,
\end{equation*}
and
\begin{equation*}
0114\ (0114)^\infty,\ 114\ (0114)^\infty,\ 14\ (0114)^\infty,\  4\ (0114)^\infty.
\end{equation*}
Applying Proposition \ref{p21}(i)
we obtain the relations
\begin{multline*}
\theta_0<b_1<a_4=\theta_1<\eta_1<b_2<b_3<\theta_2\\
<\eta_2<\theta_3<\eta_3<a_3<a_2<\theta_4<\eta_4=b_4<a_1<\eta_{5}.
\end{multline*}
The $\tilde\gg(q)$ is strongly connected by Proposition \ref{p125}.
Note that we cannot apply Corollary \ref{c123} because $M>1$.
\end{example}

\begin{example}\label{e128}
Let $M=1$, $N=10$ and $\alpha(q)=1110011011\ 0^\infty$. By the usual argument we have the following  inequalities:
\begin{multline*}
\theta_0<b_1<b_6<b_2<a_4<b_9<b_7<a_8<b_3<a_5<a_{10}=\theta_1\\
<b_{10}=\eta_1<b_5<a_3<b_8<a_7<a_9<b_4<a_2<a_6<a_1<\eta_2.
\end{multline*}
Although $\tilde\gg(q)$ contains intervals of the form $(a_i,b_j)$, the graph is still strongly connected by Proposition \ref{p125}.
%The word $1110011010$ is irreducible again.
\end{example}

\begin{example}\label{e129}
Let $M=1$ and $\beta(q)=111\ 001\ 01\ 0^\infty$. We have the usual relations:
\begin{multline*}
0<b_1<a_4<b_2<a_7<a_5<b_6<b_3<\theta_1=a_8\\
<\eta_1=b_8<a_3<a_6<b_5<b_7<a_2<b_4<a_1<\eta_2.
\end{multline*}
Since $a_4<b_2$, we cannot apply Proposition \ref{p125}.
However, we will apply Proposition \ref{p121} to show that $\tilde\gg(q)$ is strongly connected.

We infer from the above relations that $\tilde \gg (q)$ has four vertices that do not belong to $\tilde\gg_1(q)$:
\begin{equation*}
(a_4, b_2),\quad (a_5, b_6),\quad
(a_6, b_5)\qtq{and}(a_2, b_4).
\end{equation*}
The relations show that $(a_3, a_6)$ is a vertex of $\tilde\gg_1(q)$, and
\begin{equation*}
T_1((a_3, a_6))=(a_4, a_7)\supset (a_4, b_2),
\end{equation*}
so that there is vertex from $\tilde\gg_1(q)$ to $(a_4, b_2)$, and then by reflection to $(a_2, b_4)$, too.

Furthermore,
\begin{equation*}
T_0((a_4, b_2))=(a_5, b_3)\supset (a_5, b_6),
\end{equation*}
so that there is path from $\tilde\gg_1(q)$ to $(a_5, b_6)$, and then by reflection to $(a_6, b_5)$, too.
\end{example}

\subsection{Comments on the validity of Theorems \ref{t18} and \ref{t110}}\label{ss1202}

In Example \ref{e1210} the graph $\gg(q)$ is not strongly connected, but Theorem \ref{t110} may be applied.
In Example \ref{e1211} we exhibit a base $q\in\uuu\setminus\uu$ for which none of Theorems \ref{t18} and \ref{t110} applies.

\begin{example}\label{e1210}
Let $M=4$ and $\beta(q_0)=322\ 0^\infty$.
By the usual arguments, we have the following relations:
\begin{equation*}
\theta_0<\theta_1<\eta_1<b_1<\theta_2=a_3<\eta_2<a_2<b_2<\theta_3<\eta_3=b_3<a_1<\theta_4<\eta_4<\eta_5.
\end{equation*}
The subgraph $\tilde \gg(q_0)$ has five vertices:
\begin{equation*}
(b_1, a_3),\ (\eta_2, a_2),\  (a_2, b_2),\ (b_2, \theta_3),\ (b_3, a_1),
\end{equation*}
and the following edges:
\begin{align*}
&(b_1, a_3)\xrightarrow {1}(b_2, \theta_3),
&&(b_1, a_3)\xrightarrow{1} (b_3, a_1),\\
&(\eta_2, a_2)\xrightarrow{2} (b_1, a_3),
&&(b_2, \theta_3)\xrightarrow{2} (b_3, a_1),\\
&(b_3, a_1)\xrightarrow {3}(b_1, a_3),
&&(b_3, a_1)\xrightarrow{3} (\eta_2, a_2),
\end{align*}
and
\begin{align*}
&(a_2, b_2)\xrightarrow {2}(\eta_2, a_2),
&&(a_2, b_2)\xrightarrow {2}(a_2, b_2),
&&(a_2, b_2)\xrightarrow{2} (b_2, \theta_3).
\end{align*}
Since there is no path to $(a_2, b_2)$ from any other vertex, $\tilde \gg(q_0)$ is not strongly connected.
On the other hand, we have two strongly connected components: $\tilde \gg_1(q_0)$ formed by the four vertices
\begin{equation*}
(b_1, a_3),\ (\eta_2, a_2),\ (b_2, \theta_3),\ (b_3, a_1),
\end{equation*}
and the subgraph $\tilde \gg_2(q_0)$ consisting of the only vertex $(a_2, b_2)$ and the only edge $(a_2, b_2)\xrightarrow {2}(a_2, b_2)$.
We have obviously $\dim \tilde \gg_{2,q_0}=0$, and therefore
\begin{equation*}
\dim \tilde \gg_{1,q_0}
=\dim \tilde \gg_{q_0}
=\dim \uu_{q_0},
\end{equation*}
so that Theorem \ref{t110} may be applied for  $q=r_k, k\geq 0$ with $\beta(r_k)=322\ (123)^k\ 0^\infty$.
%Note that $321$ is not an irreducible word.
\end{example}

\begin{example}\label{e1211}
Let $N=15$ and $\beta(q)=MMM\ 00M\ 000\ MMM\ 00M\ 0^\infty$, then $q\in \uuu\setminus\uu$.
%The word $MMM\ 00M\ 000\ MMM\ 00M^-$ is neither irreducible nor %*-irreducible if $M=1$.
We claim that the graph  $\tilde\gg(q)$ is  strongly connected if and only if $M>1$ and  for $M=1$ Theorems \ref{t18} and \ref{t110} do not apply.

First we show that  $\tilde\gg(q)$ is not strongly connected if  $M=1$.
We obtain by the usual arguments the relations
\begin{multline*}
\theta_0<b_1<b_{10}<a_7<a_{13}<a_4<b_2<b_{11}<a_8<a_{14}<a_5\\
<b_3<b_{12}<b_6<a_9<\theta_1=a_{15}
<\eta_1=b_{15}<b_{9}<a_6<a_{12}<a_3\\
<b_5<b_{14}<b_8<a_{11}<a_2<b_4<b_{13}<b_7<a_{10}<a_1<\eta_2.
\end{multline*}

These relations  implies that the subgraph $\tilde\gg_2(q)$ formed by the four vertices of the form $(a_i, b_j)$ is  strongly connected because of the relations
\begin{equation}\label{125}
\begin{split}
T_0((a_4, b_2))=(a_5,b_3),
T_0((a_5, b_3))=(a_6,b_4)\supset ((a_3,b_5)\cup(a_2,b_4)),\\
T_1((a_2, b_4))=(a_3,b_5),
T_1((a_3, b_5))=(a_4,b_6)\supset ((a_4,b_2)\cup(a_5,b_3))
\end{split}
\end{equation}
Furthermore, a sequence $(c_i)$ is generated by the graph  $\tilde\gg_2(q)$ if and only if
\begin{equation}\label{126}
(c_{n+i})\leq \alpha(q)=(110)^\infty
\qtq{and}
\overline{(c_{n+i})}\leq \alpha(q)=(110)^\infty
\qtq{for all} n\geq 0.
\end{equation}

The subgraph $\tilde\gg_1(q)$ formed by the vertices of the form $(a_i^-, a_i)$ and $(b_j, b_j^+
)$ is also strongly connected, and it contains the words $000\ 110\ 111\ 000\ 111$ by Lemma \ref{l101}.

We claim that there is no path going from $\tilde\gg_1(q)$ to $\tilde\gg_2(q)$, and therefore $\tilde\gg(q)$ is not strongly connected.
By the monotonicity of maps $T_0(x)=qx$ and $T_1(x)=qx-1$, there are at most two edges ending at an arbitrarily chosen vertex of $\tilde\gg(q)$.
Since
\begin{equation*}
T_0((0,b_1))\supset (a_4,b_2),
\quad T_1((a_3,b_5))\supset (a_4,b_2)\cup (a_5, b_3)\qtq{and}T_0((a_4, b_2))=(a_5,b_3)
\end{equation*}
(see \eqref{125}), there are no paths  from $\tilde\gg_1(q)$ to the vertices $(a_4, b_2)$ and $(a_5,b_3)$.
By symmetry, there are no paths  from $\tilde\gg_1(q)$ to $(a_2,b_4)$ and $(a_3,b_5)$ either.
Next we show that $\dim \tilde\gg_{2,q}>\dim\tilde\gg_{1,q}$ and therefore $\dim\uu_q=\dim \tilde\gg_{2,q}$.
%(This equality also follows from \eqref{126} by applying \cite[Theorem 2]{ABBK2019}.)
The inequality follows from the relation $r(A_1)<r(A_2)$ between the spectral radii of the adjacency matrices corresponding to the subgraph $\tilde\gg_1(q)$ and $\tilde\gg_2(q)$, respectively.
Indeed, a direct computation shows that
$r(A_1)\approx 1.14798$ and $r(A_2)\approx 1.61803.$

If we want to apply Lemma \ref{l93}, then the condition \eqref{98} implies that each $(c_i)\in\ff'$ has to start with $c_1\cdots c_3\le 000$.
Since no sequence generated by $\tilde\gg_2(q)$ contains the word $000$ (see \eqref{126}), we have $\dim\ff_q\le\dim\tilde\gg_{1,q}<\dim\uu_q$.

Finally, we show that the graph  $\tilde\gg(q)$ is  strongly connected if  $M>1$.
We obtain by the usual arguments the relations
\begin{align*}
&\theta_0<b_1<b_{10}<a_7<a_{13}<a_4<b_2<b_{11}<a_8<a_{14}<a_5<b_3<b_{12}<b_6<a_9\\
&<\theta_1<\eta_1=b_{15}<\theta_2<\eta_2<\cdots<\theta_{M-1}<\eta_{M-1}<\theta_{M}=a_{15}<\eta_{M}\\
&<b_{9}<a_6<a_{12}<a_3<b_5<b_{14}<b_8<a_{11}<a_2<b_4<b_{13}<b_7<a_{10}<a_1<\eta_{2}.
\end{align*}

We recall that $(\theta_1, \eta_1),\ldots, (\theta_M, \eta_M)$ are  the switch intervals.
In view of Proposition \ref{p121} we need to prove that there exists a path from
$\tilde \gg_1(q)$  to each vertex of the form $(a_i, b_j)$ and $(\theta_i^-, \theta_i)$ of $\tilde \gg(q)$.
The above order relations show that
 $(b_{14}, b_8)$ is a vertex of $\tilde\gg_1(q)$, and
\begin{equation*}
(b_{14}, b_8)\xrightarrow{M}(\eta_1, \theta_{2}).
\end{equation*}
Furthermore, we have
\begin{equation*}
T_1((\eta_1, \theta_{2}))=(b_1, a_1)
\end{equation*}
by Lemma \ref{l33} (ii) and (iii),
Since $(b_1, a_1)$ contains all the vertices of $\tilde\gg(q)$, we conclude that there exists a path from the vertex $(b_{14}, b_8)$ to each vertex of $\tilde\gg(q)$.
\end{example}

\end{document}